\documentclass[11pt,preprint]{imsart}

\usepackage{amsmath,amsfonts}
\usepackage{algorithmic}
\usepackage{algorithm}
\usepackage{array}
\usepackage{caption,subcaption}
\usepackage{textcomp}
\usepackage[normalem]{ulem}
\usepackage{stfloats}
\usepackage{url}
\usepackage{verbatim}
\usepackage{graphicx}
\usepackage{cite}
\usepackage{xcolor}
\usepackage{bbm,bm}
\usepackage{tikz}
\RequirePackage[OT1]{fontenc}
\RequirePackage{amsthm,amsmath,amssymb,graphicx,epstopdf,enumerate,lmodern}
\RequirePackage[round,colon,authoryear]{natbib}
\RequirePackage[colorlinks,citecolor=blue,urlcolor=blue]{hyperref}
\usepackage{bm,color,fancyvrb,mathtools, undertilde}
\newtheorem{theorem}{Theorem}[section]
\newtheorem*{theorem*}{Theorem}
\newtheorem{proposition}[theorem]{Proposition}
\newtheorem{lemma}{Lemma}
\newtheorem*{lemma*}{Lemma}

\newtheorem{corollary}[theorem]{Corollary}

\newtheorem{example}{Example}[section]

\DeclarePairedDelimiter{\norm}{\lVert}{\rVert}
\DeclarePairedDelimiter{\abs}{\lvert}{\rvert}
\providecommand{\abs}[1]{\left\lvert#1\right\rvert}
\providecommand{\norm}[1]{\left\lVert#1\right\rVert}
\renewcommand{\hat}{\widehat}

\newcommand{\bfm}[1]{\ensuremath{\mathbf{#1}}}

   \def\bA{\bfm A}  
   \def\bB{\bfm B}  
   \def\bC{\bfm C}  
   \def\bD{\bfm D}  
   \def\bE{\bfm E}  \def\EE{\mathbb{E}}
\def\bff{\bfm f}  \def\bF{\bfm F}  
\def\bg{\bfm g}   \def\bG{\bfm G}  
\def\bh{\bfm h}   \def\bH{\bfm H}  
   \def\bI{\bfm I}  \def\II{\mathbb{I}}
   \def\bJ{\bfm J}  
     
   \def\bL{\bfm L}  
   \def\bM{\bfm M}

     \def\PP{\mathbb{P}}
     
   \def\bR{\bfm R}  \def\RR{\mathbb{R}}
     \def\SS{\mathbb{S}}
     
\def\bu{\bfm u}   \def\bU{\bfm U}  
\def\bv{\bfm v}   \def\bV{\bfm V}  
\def\bw{\bfm w}     
\def\bx{\bfm x}   \def\bX{\bfm X}  
\def\by{\bfm y}     
\def\bz{\bfm z}     \def\ZZ{\mathbb{Z}}

  \def\starbA{\stackrel{\star}{\bA}}
\def\starbDelta{\stackrel{\star}{\bDelta}}
\def\starepsilon{\stackrel{\star}{\epsilon}}
\def\ccbv{\check{\check{\bv}}}

\def\calA{{\cal  A}} 
\def\calB{{\cal  B}} \def\cB{{\cal  B}}
 \def\cC{{\cal  C}}
\def\calD{{\cal  D}} \def\cD{{\cal  D}}
\def\calE{{\cal  E}} 
\def\calF{{\cal  F}} 
\def\calG{{\cal  G}}

\def\calK{{\cal  K}} \def\cK{{\cal  K}}
\def\calL{{\cal  L}} \def\cL{{\cal  L}}
 
\def\calN{{\cal  N}}

\def\calS{{\cal  S}} 
\def\calT{{\cal  T}} 
 
 \def\cV{{\cal  V}}

\def\hbw{{\hat{\bw}}}
\def\hbv{{\hat{\bv}}}

\newcommand{\bfsym}[1]{\ensuremath{\boldsymbol{#1}}}

\def\bbeta{\bfsym \beta}
             \def\bGamma{\bfsym \Gamma}
\def\bdelta{\bfsym {\delta}}           \def\bDelta {\bfsym {\Delta}}

\def\btheta{\bfsym {\theta}}           
          
             \def\bSigma{\bfsym \Sigma}

\def\bxi{\bfsym {\xi}}

\def\soft{{\rm soft}}
\def\sign{{\rm sign}}

\def\hbeta{\hat{\beta}}                \def\hbbeta{\hat{\bfsym \beta}}

\def\hmu{\hat{\mu}}

\DeclareMathOperator{\argmax}{argmax}
\DeclareMathOperator{\argmin}{argmin}

\DeclareMathOperator{\diag}{diag}

\DeclareMathOperator{\sgn}{sgn}

\DeclareMathOperator{\var}{var}

\def\var{\mbox{var}}

\newcommand\numberthis{\addtocounter{equation}{1}\tag{\theequation}}

\def\eps{\varepsilon}

\def\polylog{{\rm PolyLog}}
\def\Ldd{\ddot{\bL}}
\def\Rdd{\ddot{\bR}}

\begin{document}

\markboth{IEEE Transactions on Information Theory ,~Vol.~14, No.~8, August~2021}%
{Auddy \MakeLowercase{\textit{et al.}}: Approximate Leave-one-out Cross Validation for Regression with $\ell_1$ Regularizer}



\begin{frontmatter}

\title{Approximate Leave-one-out Cross Validation for Regression with $\ell_1$ Regularizers (extended version)}



\runtitle{ALOCV for Regression with $\ell_1$ Regularizer}
\runauthor{Auddy \MakeLowercase{\textit{et al.}}}

\begin{aug}
			\author{\fnms{Arnab} \snm{Auddy,}\thanksref{m1}}
			\author{\fnms{Haolin} \snm{Zou,}\thanksref{m2}}
                \author{\fnms{Kamiar} \snm{Rahnama Rad,}\thanksref{m3}}
			\and
			\author{\fnms{Arian} 
				\snm{Maleki}\thanksref{m2}}
			\affiliation{
				Department of Statistics, Columbia University\thanksmark{a1}
			}
		\address{\thanksmark{m1}
			University of Pennsylvania\\
                \thanksmark{m2}Columbia University\\
                \thanksmark{m3}
                Baruch College, City University of New York
                }
		\end{aug}

\begin{abstract}
The out-of-sample error (OO) is the main quantity of interest in risk estimation and model selection. Leave-one-out cross validation (LO) offers a (nearly) distribution-free yet computationally demanding approach to estimate OO. Recent theoretical work showed that approximate leave-one-out cross validation (ALO) is a computationally efficient and statistically reliable estimate of LO (and OO) for generalized linear models with differentiable regularizers. For problems involving non-differentiable  regularizers, despite significant empirical evidence, the theoretical understanding of ALO's error remains unknown. In this paper, we present a novel theory for a wide class of problems in the generalized linear model family with non-differentiable regularizers. We bound the error \(|{\rm ALO}-{\rm LO}|\) in terms of intuitive metrics such as the size of leave-\(i\)-out perturbations in active sets, sample size $n$, number of features $p$ and regularization parameters. As a consequence, for the $\ell_1$-regularized problems, we show that  $|{\rm ALO}-{\rm LO}| \xrightarrow{p\rightarrow \infty} 0$ while $n/p$ and SNR are fixed and bounded.
\end{abstract}

\end{frontmatter}


\section{Introduction}

We observe a dataset $\mathcal{D}=\{(y_1,\bx_1),\dots,(y_n,\bx_n)\}$ where $\bx_i\in\RR^p$ and $y_i\in \RR$ denote the features and response, respectively. We assume observations are independent and identically distributed draws from some unknown joint distribution $q(y_i|\bx_i^\top\bbeta^*)p(\bx_i)$. We estimate $\bbeta^*$ using
the optimization problem
\begin{equation}\label{eq:opt:prob}
	\hat{\bbeta}_{\lambda}:={\rm argmin}_{\bbeta\in \RR^p}\sum_{i=1}^n\ell(y_i|\bx_i^\top\bbeta)+\lambda r(\bbeta)
\end{equation}
where $\ell(y|z)$ is the loss function, and $r(\bbeta)$ is the regularizer. Consider the problem of estimating the out-of-sample prediction error (OO) 
$$
{\rm OO}_{\lambda} :=\EE[\phi(y_{\rm new},\bx_{\rm new}^\top\hat{\bbeta}_{\lambda})|\calD]
$$
where $\phi(y,z)$ is another loss function (possibly but not necessarily the same as $\ell(y|z)$),  $(y_{\rm new},\bx_{\rm new})$ is a sample from the same joint distribution $q(y|\bx^\top\bbeta^*)p(\bx)$, but is independent of the training set $\calD$. As demonstrated through empirical and theoretical studies, in high-dimensional settings ($n,p \rightarrow \infty$ while $n/p$ and SNR are fixed and bounded), the leave-one-out cross validation (LO) estimator
\begin{equation}\label{eq:loocv}
	{\rm LO}:=\dfrac{1}{n}\sum_{i=1}^n\phi(y_i;\bx_i^\top\hat{\bbeta}_{/i})
\end{equation}
where 
\begin{equation}\label{eq:beta_wo_i}
	\hat{\bbeta}_{/i}={\rm argmin}_{\bbeta\in \RR^p} \left\{\sum_{j\neq i}\ell(y_j|\bx_j^\top\bbeta)+\lambda r(\bbeta)\right\},
\end{equation}
provides an accurate estimate of the risk: \cite{rad2020error,patil2021}. 

A significant limitation of LO is its necessity to fit the model repeatedly $n$ times, making it computationally impractical for many high-dimensional problems. As a result, recently several researches have considered the problem of approximating LO \citep{beirami2017optimal,stephenson2020approximate,rad2018scalable,giordano2019a,giordano2019b,wang2018approximate,rad2020error,patil2021,patil2022}. For instance in \citep{rad2018scalable,rad2020error} it was theoretically and empirically shown that for differentiable regularizers and generalized linear models approximate leave-one-out cross validation (ALO) is a statistically reliable and computationally efficient approach to estimate LO, and OO. Regarding non-differentiable regularizers, while the extensive simulations in \citep{wang2018approximate,rad2018scalable} provided empirical evidence, the theoretical understanding of ALO's error remains unknown. 

In this paper, we present a novel theory for non-differentiable regularizers applied to a wide class of problems in the generalized linear model family, e.g., Poisson and logistic. Using intuitive metrics such as the size of leave-\(i\)-out perturbations in active sets we bound the error \(|{\rm ALO}-{\rm LO}|\) in terms of fundamental quantities such as sample size $n$, number of features $p$, and  the regularization parameters. For elastic net problems we place bounds on the size of leave-\(i\)-out perturbations in active sets, and as a consequence, show that  $|{\rm ALO}-{\rm LO}| \xrightarrow{p\rightarrow \infty}0$ when $n/p$ and SNR remain fixed and finite\footnote{In Section \ref{sec:bdd-snr} we rigorously define a fixed finite SNR. Roughly speaking, we mean that $\var[\bx_j^\top\bbeta^*]$ and $\var[y_i|\bx_j^\top\bbeta^*]$ remain fixed and finite regardless of problem dimensions.}.  

The remainder of this paper is organized as follows. In Section \ref{sec:alo} we briefly present and review the key idea of ALO and the challenge to use it for non-differentiable regularizers. We review related work in section \ref{sec:rel}. In section \ref{sec:cont} we briefly describe the main theoretical contributions of this paper. 
Section \ref{sec:main} formally introduces our theoretical framework and presents  Theorem \ref{th:alo-main}  which allows us to bound the error \(|{\rm ALO}-{\rm LO}|\) in terms of intuitive metrics such as the size of leave-\(i\)-out perturbations in active sets for wide array of models models such as linear Gaussian, logistic and Poisson. In Section \ref{sec:linearregression} we present Theorem \ref{thm:size-diff-S} which bounds metrics related to the size of leave-\(i\)-out perturbations in active sets for $\ell_1$ regularized problems. Next we explain how Theorem \ref{th:alo-main} and Theorem \ref{thm:size-diff-S} together prove that for $\ell_1$ regularized problems, when $n$ and $p$ grow, as $n/p$ and SNR remains fixed and finite, ALO is a consistent estimate of LO. Detailed proofs are given in Section \ref{sec:proofs}. Concluding remarks are given in section \ref{sec:conc}. Finally, the Appendix~\ref{app-enet} prepares some results on the elastic net regularized least squares optimization problem.



\section{Approximate leave-one-out cross validation}\label{sec:alo}
In this section, we briefly review the main idea of ALO and we describe the challenges of using it for non-differentiable regularizers.
\subsection{ALO for twice-differentiable losses and regularizers}
ALO replaces the computationally demanding procedure of repeatedly fitting the model with finding an approximate model that is easy to compute. Instead of exactly computing  $\hat{\bbeta}_{/i}$ as in \eqref{eq:beta_wo_i}, ALO adjusts the estimate $\hat{\bbeta}$ based on the entire dataset $\calD$, and uses one Newton step to compute the approximation $\tilde{\bbeta}_{/i}$ as follows: 
\begin{align*}\label{eq:beta_nwtn}
	&~  \tilde{\bbeta}_{/i}:=\hat{\bbeta} +\left(\sum_{j\neq i}\bx_j\bx_j^\top \ddot{\ell}(\by_j|\bx_j^\top\hat{\bbeta})+\lambda\nabla^2 r(\hat{\bbeta}) \right)^{-1}
	\bx_i\dot{\ell}(y_i|\bx_i^\top \hat{\bbeta}),\numberthis
\end{align*}
where $\dot{\ell} (y|z)$ and $\ddot{\ell} (y|z)$ denote the first and second derivatives of $\ell(y|z)$ with respect to its second argument. Furthermore, 
\[
[\nabla^2 r({\bw})]_{ij} := \left.\frac{\partial^2 r(\bbeta)}{ \partial \beta_i \partial \beta_j} \right|_{\bbeta =  \bw}. 
\]
It might seem that the matrix inversion required in \eqref{eq:beta_nwtn} is computationally (nearly) as demanding as refitting the model, but this can actually be bypassed by using the Woodbury lemma (see Lemma~\ref{lem:woodberry}), leading to the following approximation
\begin{eqnarray}\label{eq:alodef}
	\lefteqn{{\rm ALO}:=\dfrac{1}{n}\sum_{i=1}^n\phi(y_i;\bx_i^\top\tilde{\bbeta}_i) } \\
	&=&\dfrac{1}{n}\sum_{i=1}^n\phi\left(y_i;\bx_i^\top\hat{\bbeta}
	+\left(\dfrac{\dot{\ell}(y_i|\bx_i^\top \hat{\bbeta})}{\ddot{\ell}(y_i|\bx_i^\top \hat{\bbeta})}\right)\left(\dfrac{H_{ii}}
	{1-H_{ii}}\right)\right) \nonumber
\end{eqnarray}
where $H_{ii}$ is the $(i,i)$ element of the matrix $\bH$ defined as
$$\bH := \bX
	\big(\bX^\top[{\rm diag}(\ddot{\ell}(\hat{\bbeta}))]\bX + \lambda \nabla^2 r(\hat{\bbeta})\big)^{-1}
	\bX^\top \diag[\ddot{\ell}(\hbbeta)]$$  
with 
$\ddot{\ell}(\bw) 
:= \left [\ddot{\ell}(y_1|\bx_1^\top \bw), \cdots, \ddot{\ell}(y_n|\bx_n^\top \bw) \right]^\top$ and $\diag[\ddot{\ell}(\hbbeta)]$ being the diagonal matrix with $\ddot{\ell}(\bw)$ as its diagonal elements.
Most of the theoretical work about the consistency of ALO in estimating LO (and OO) has focused on differentiable regularizers, such as ridge, smoothed LASSO and Huber loss when $n/p$ remain fixed and finite \citep{rad2020error,rad2018scalable,patil2022,xu2021consistent}.

\subsection{ALO for non-differentiable regularizers}
\label{ssec:ALO_non_diff}
The ALO formula above and the corresponding theory supporting it require twice differentiability but some of the most important loss functions and regularizers, such as elastic net, nuclear norm, or hinge loss are not twice-differentiable. 
For example, consider the following estimate:
\begin{align}
\label{eq: elastic net}
	\hat{\bbeta}:= ~
 &\underset{\bbeta\in \RR^p}{\rm argmin} \ h(\bbeta),
\end{align}
where $h(\bbeta) := 
 \sum_{i=1}^n\ell(y_i|\bx_i^\top\bbeta)+\lambda (1- \eta)  \|\bbeta\|_1+ \lambda \eta\|\bbeta\|_2^2$, 
and suppose that as before our goal is to approximate  
\begin{equation}\label{eq:loocv-2}
	{\rm LO}:=\dfrac{1}{n}\sum_{i=1}^n\phi(y_i;\bx_i^\top\hat{\bbeta}_{/i})
\end{equation}
where 
\begin{align*}\label{eq:enet_wo_i}
& \hat{\bbeta}_{/i}:=\\
 &\underset{\bbeta\in \RR^p}{\rm argmin} \left\{\sum_{j\neq i}\ell(y_j|\bx_j^\top\bbeta)+\lambda(1-\eta) \|\bbeta\|_1+ \lambda\eta \|\bbeta\|_2^2 \right\}.\numberthis
\end{align*}
Due to the regularizer's non-differentiability we cannot use the one step Newton approximation as proposed in the previous section. However, the following heuristic argument motivates our new method. Let $\calS$ denote the active set of $\hat{\bbeta}$, i.e.,
\[
\calS = \{i \ : \ \hat{\beta}_i \neq 0\}. 
\]
Suppose that the active set of $\hat{\bbeta}_{/i}$ remains the same as $\calS$ for all $i$. Then, we can solve \eqref{eq:beta_wo_i} on the set $\calS$ only. The validity of this heuristic assumption, and our remedies when it is mildly violated, are discussed in later sections. For now, since the regularizer is twice differentiable on $\calS$, we can use the Newton method to obtain the following approximation for LO  of the elastic-net:
\begin{eqnarray}\label{eq:ALO:LASSO}
		{\rm ALO} = \dfrac{1}{n}\sum_{i=1}^n\phi\left(y_i,\bx_i^\top\hat{\bbeta}
		+\left(\tfrac{\dot{\ell}_i(\hat{\bbeta})}{\ddot{\ell}_i(\hat{\bbeta})}\right)\left(\dfrac{H_{ii}}{1-H_{ii}}\right)\right)
\end{eqnarray}
where $H_{ii}$ is the $(i,i)$ element of
  $$
  \bH:=\bX_\calS\left(2 \lambda \eta \II+\bX_\calS^\top{\rm diag}[\ddot{\ell}(\hat{\bbeta})]\bX_\calS\right)^{-1}\bX_\calS^\top{\rm diag}[\ddot{\ell}(\hat{\bbeta})].
  $$
with $\bX_{\calS}$ contains the columns of $X$ that are in $\calS$.

Unfortunately, the assumption that led to \eqref{eq:ALO:LASSO}, i.e., the assumption that the active set does not change when a data point is removed, is not correct. While it is true that some of leave-$i$-out estimated coefficients retain the active set, most estimated active sets do change. Figure~\ref{fig:sim_change_active_set} confirms this claim. 
 
\begin{figure}[h]
    \centering
    \begin{tikzpicture}[x=0.3pt,y=0.3pt]
\definecolor{fillColor}{RGB}{255,255,255}
\begin{scope}
\definecolor{drawColor}{RGB}{255,255,255}
\definecolor{fillColor}{RGB}{255,255,255}

\end{scope}
\begin{scope}
\definecolor{fillColor}{RGB}{255,255,255}

\path[fill=fillColor] ( 31.71, 30.69) rectangle (536.52,355.85);
\definecolor{drawColor}{gray}{0.92}

\path[draw=drawColor,line width= 0.3pt,line join=round] ( 31.71, 96.43) --
	(536.52, 96.43);

\path[draw=drawColor,line width= 0.3pt,line join=round] ( 31.71,198.36) --
	(536.52,198.36);

\path[draw=drawColor,line width= 0.3pt,line join=round] ( 31.71,300.30) --
	(536.52,300.30);

\path[draw=drawColor,line width= 0.3pt,line join=round] (105.80, 30.69) --
	(105.80,355.85);

\path[draw=drawColor,line width= 0.3pt,line join=round] (192.78, 30.69) --
	(192.78,355.85);

\path[draw=drawColor,line width= 0.3pt,line join=round] (279.77, 30.69) --
	(279.77,355.85);

\path[draw=drawColor,line width= 0.3pt,line join=round] (366.75, 30.69) --
	(366.75,355.85);

\path[draw=drawColor,line width= 0.3pt,line join=round] (453.74, 30.69) --
	(453.74,355.85);

\path[draw=drawColor,line width= 0.6pt,line join=round] ( 31.71, 45.47) --
	(536.52, 45.47);

\path[draw=drawColor,line width= 0.6pt,line join=round] ( 31.71,147.40) --
	(536.52,147.40);

\path[draw=drawColor,line width= 0.6pt,line join=round] ( 31.71,249.33) --
	(536.52,249.33);

\path[draw=drawColor,line width= 0.6pt,line join=round] ( 31.71,351.26) --
	(536.52,351.26);

\path[draw=drawColor,line width= 0.6pt,line join=round] ( 62.31, 30.69) --
	( 62.31,355.85);

\path[draw=drawColor,line width= 0.6pt,line join=round] (149.29, 30.69) --
	(149.29,355.85);

\path[draw=drawColor,line width= 0.6pt,line join=round] (236.28, 30.69) --
	(236.28,355.85);

\path[draw=drawColor,line width= 0.6pt,line join=round] (323.26, 30.69) --
	(323.26,355.85);

\path[draw=drawColor,line width= 0.6pt,line join=round] (410.25, 30.69) --
	(410.25,355.85);

\path[draw=drawColor,line width= 0.6pt,line join=round] (497.23, 30.69) --
	(497.23,355.85);
\definecolor{fillColor}{RGB}{0,0,255}

\path[fill=fillColor] ( 54.66, 45.47) rectangle ( 69.96,188.17);

\path[fill=fillColor] ( 69.96, 45.47) rectangle ( 85.25,290.10);

\path[fill=fillColor] ( 85.25, 45.47) rectangle (100.55,341.07);

\path[fill=fillColor] (100.55, 45.47) rectangle (115.85,228.94);

\path[fill=fillColor] (115.85, 45.47) rectangle (131.15,132.11);

\path[fill=fillColor] (131.15, 45.47) rectangle (146.44,254.43);

\path[fill=fillColor] (146.44, 45.47) rectangle (161.74,310.49);

\path[fill=fillColor] (161.74, 45.47) rectangle (177.04,223.85);

\path[fill=fillColor] (177.04, 45.47) rectangle (192.33,106.63);

\path[fill=fillColor] (192.33, 45.47) rectangle (207.63,218.75);

\path[fill=fillColor] (207.63, 45.47) rectangle (222.93,162.69);

\path[fill=fillColor] (222.93, 45.47) rectangle (238.23,177.98);

\path[fill=fillColor] (238.23, 45.47) rectangle (253.52, 91.34);

\path[fill=fillColor] (253.52, 45.47) rectangle (268.82,172.88);

\path[fill=fillColor] (268.82, 45.47) rectangle (284.12,106.63);

\path[fill=fillColor] (284.12, 45.47) rectangle (299.42, 91.34);

\path[fill=fillColor] (299.42, 45.47) rectangle (314.71, 81.14);

\path[fill=fillColor] (314.71, 45.47) rectangle (330.01, 55.66);

\path[fill=fillColor] (330.01, 45.47) rectangle (345.31, 81.14);

\path[fill=fillColor] (345.31, 45.47) rectangle (360.61, 76.05);

\path[fill=fillColor] (360.61, 45.47) rectangle (375.90, 70.95);

\path[fill=fillColor] (375.90, 45.47) rectangle (391.20, 50.56);

\path[fill=fillColor] (391.20, 45.47) rectangle (406.50, 55.66);

\path[fill=fillColor] (406.50, 45.47) rectangle (421.79, 55.66);

\path[fill=fillColor] (421.79, 45.47) rectangle (437.09, 45.47);

\path[fill=fillColor] (437.09, 45.47) rectangle (452.39, 50.56);

\path[fill=fillColor] (452.39, 45.47) rectangle (467.69, 45.47);

\path[fill=fillColor] (467.69, 45.47) rectangle (482.98, 45.47);

\path[fill=fillColor] (482.98, 45.47) rectangle (498.28, 50.56);

\path[fill=fillColor] (498.28, 45.47) rectangle (513.58, 50.56);
\definecolor{drawColor}{gray}{0.20}

\path[draw=drawColor,line width= 0.6pt,line join=round,line cap=round] ( 31.71, 30.69) rectangle (536.52,355.85);
\end{scope}
\begin{scope}
\definecolor{drawColor}{gray}{0.30}

\node[text=drawColor,anchor=base east,inner sep=0pt, outer sep=0pt, scale=  0.88] at ( 26.76, 42.44) {0};

\node[text=drawColor,anchor=base east,inner sep=0pt, outer sep=0pt, scale=  0.88] at ( 26.76,144.37) {0.04};

\node[text=drawColor,anchor=base east,inner sep=0pt, outer sep=0pt, scale=  0.88] at ( 26.76,246.30) {0.08};

\node[text=drawColor,anchor=base east,inner sep=0pt, outer sep=0pt, scale=  0.88] at ( 26.76,348.23) {0.12};
\end{scope}
\begin{scope}
\definecolor{drawColor}{gray}{0.20}

\path[draw=drawColor,line width= 0.6pt,line join=round] ( 28.96, 45.47) --
	( 31.71, 45.47);

\path[draw=drawColor,line width= 0.6pt,line join=round] ( 28.96,147.40) --
	( 31.71,147.40);

\path[draw=drawColor,line width= 0.6pt,line join=round] ( 28.96,249.33) --
	( 31.71,249.33);

\path[draw=drawColor,line width= 0.6pt,line join=round] ( 28.96,351.26) --
	( 31.71,351.26);
\end{scope}
\begin{scope}
\definecolor{drawColor}{gray}{0.20}

\path[draw=drawColor,line width= 0.6pt,line join=round] ( 62.31, 27.94) --
	( 62.31, 30.69);

\path[draw=drawColor,line width= 0.6pt,line join=round] (149.29, 27.94) --
	(149.29, 30.69);

\path[draw=drawColor,line width= 0.6pt,line join=round] (236.28, 27.94) --
	(236.28, 30.69);

\path[draw=drawColor,line width= 0.6pt,line join=round] (323.26, 27.94) --
	(323.26, 30.69);

\path[draw=drawColor,line width= 0.6pt,line join=round] (410.25, 27.94) --
	(410.25, 30.69);

\path[draw=drawColor,line width= 0.6pt,line join=round] (497.23, 27.94) --
	(497.23, 30.69);
\end{scope}
\begin{scope}
\definecolor{drawColor}{gray}{0.30}

\node[text=drawColor,anchor=base,inner sep=0pt, outer sep=0pt, scale=  0.88] at ( 62.31, 0.68) {0};

\node[text=drawColor,anchor=base,inner sep=0pt, outer sep=0pt, scale=  0.88] at (149.29, 0.68) {10};

\node[text=drawColor,anchor=base,inner sep=0pt, outer sep=0pt, scale=  0.88] at (236.28, 0.68) {20};

\node[text=drawColor,anchor=base,inner sep=0pt, outer sep=0pt, scale=  0.88] at (323.26, 0.68) {30};

\node[text=drawColor,anchor=base,inner sep=0pt, outer sep=0pt, scale=  0.88] at (410.25, 0.68) {40};

\node[text=drawColor,anchor=base,inner sep=0pt, outer sep=0pt, scale=  0.88] at (497.23, 0.68) {50};
\end{scope}
\begin{scope}
\definecolor{drawColor}{RGB}{0,0,0}

\node[text=drawColor,anchor=base,inner sep=0pt, outer sep=0pt, scale=  1.10] at (284.12,  -24.64) {\small size of change in active set};
\end{scope}
\begin{scope}
\definecolor{drawColor}{RGB}{0,0,0}

\node[text=drawColor,rotate= 90.00,anchor=base,inner sep=0pt, outer sep=0pt, scale=  1.10] at ( -33.08,208.27) {\small probability};
\end{scope}
\end{tikzpicture}
    \caption{\small The histogram above shows the size of the difference in active sets when performing linear regression with the elastic net penalty on the entire dataset, vs leave-$i$-out (i.e., leaving out the $i$-th observation for $i=1,\dots,n$). The parameters are $n=500$, $p=1000$ with 20\% of the true coefficients being non-zero. The design matrix $\bX$ has iid $N(0,1/n)$ rows, and the nonzero coefficients are iid $N(0,1)$. The penalty strengths are $\lambda = 2, \eta=0.5$. We use the \texttt{ElasticNet} function from the Python library \texttt{scikit-learn}.}
    \label{fig:sim_change_active_set}
\end{figure}
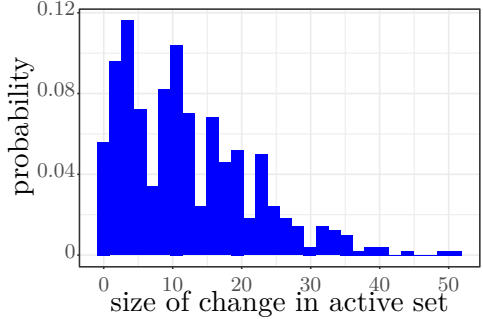
\begin{figure}[h]
    \centering
    \begin{subfigure}{0.45\textwidth}
      \begin{tikzpicture}[x=0.5pt,y=.5pt]
\definecolor{fillColor}{RGB}{255,255,255}
\begin{scope}
\definecolor{drawColor}{RGB}{255,255,255}
\definecolor{fillColor}{RGB}{255,255,255}

\end{scope}
\begin{scope}
\path[clip] ( 27.31, 30.69) rectangle (301.65,355.85);
\definecolor{fillColor}{RGB}{255,255,255}

\path[fill=fillColor] ( 27.31, 30.69) rectangle (301.65,355.85);
\definecolor{drawColor}{gray}{0.92}

\path[draw=drawColor,line width= 0.3pt,line join=round] ( 27.31, 85.78) --
	(301.65, 85.78);

\path[draw=drawColor,line width= 0.3pt,line join=round] ( 27.31,166.39) --
	(301.65,166.39);

\path[draw=drawColor,line width= 0.3pt,line join=round] ( 27.31,247.01) --
	(301.65,247.01);

\path[draw=drawColor,line width= 0.3pt,line join=round] ( 27.31,327.63) --
	(301.65,327.63);

\path[draw=drawColor,line width= 0.6pt,line join=round] ( 27.31, 45.47) --
	(301.65, 45.47);

\path[draw=drawColor,line width= 0.6pt,line join=round] ( 27.31,126.09) --
	(301.65,126.09);

\path[draw=drawColor,line width= 0.6pt,line join=round] ( 27.31,206.70) --
	(301.65,206.70);

\path[draw=drawColor,line width= 0.6pt,line join=round] ( 27.31,287.32) --
	(301.65,287.32);

\path[draw=drawColor,line width= 0.6pt,line join=round] ( 53.86, 30.69) --
	( 53.86,355.85);

\path[draw=drawColor,line width= 0.6pt,line join=round] ( 98.11, 30.69) --
	( 98.11,355.85);

\path[draw=drawColor,line width= 0.6pt,line join=round] (142.36, 30.69) --
	(142.36,355.85);

\path[draw=drawColor,line width= 0.6pt,line join=round] (186.61, 30.69) --
	(186.61,355.85);

\path[draw=drawColor,line width= 0.6pt,line join=round] (230.86, 30.69) --
	(230.86,355.85);

\path[draw=drawColor,line width= 0.6pt,line join=round] (275.10, 30.69) --
	(275.10,355.85);
\definecolor{drawColor}{gray}{0.20}

\path[draw=drawColor,line width= 0.6pt,line join=round] ( 53.86,207.23) -- ( 53.86,262.14);

\path[draw=drawColor,line width= 0.6pt,line join=round] ( 53.86,139.37) -- ( 53.86, 45.47);
\definecolor{fillColor}{RGB}{248,118,109}

\path[draw=drawColor,line width= 0.6pt,fill=fillColor] ( 37.27,207.23) --
	( 37.27,139.37) --
	( 70.46,139.37) --
	( 70.46,207.23) --
	( 37.27,207.23) --
	cycle;

\path[draw=drawColor,line width= 1.1pt] ( 37.27,179.81) -- ( 70.46,179.81);

\path[draw=drawColor,line width= 0.6pt,line join=round] ( 98.11,210.77) -- ( 98.11,277.12);

\path[draw=drawColor,line width= 0.6pt,line join=round] ( 98.11,150.01) -- ( 98.11, 85.91);

\path[draw=drawColor,line width= 0.6pt,fill=fillColor] ( 81.52,210.77) --
	( 81.52,150.01) --
	(114.70,150.01) --
	(114.70,210.77) --
	( 81.52,210.77) --
	cycle;

\path[draw=drawColor,line width= 1.1pt] ( 81.52,179.81) -- (114.70,179.81);

\path[draw=drawColor,line width= 0.6pt,line join=round] (142.36,220.25) -- (142.36,294.99);

\path[draw=drawColor,line width= 0.6pt,line join=round] (142.36,159.00) -- (142.36, 85.91);

\path[draw=drawColor,line width= 0.6pt,fill=fillColor] (125.77,220.25) --
	(125.77,159.00) --
	(158.95,159.00) --
	(158.95,220.25) --
	(125.77,220.25) --
	cycle;

\path[draw=drawColor,line width= 1.1pt] (125.77,190.45) -- (158.95,190.45);

\path[draw=drawColor,line width= 0.6pt,line join=round] (186.61,241.93) -- (186.61,309.29);

\path[draw=drawColor,line width= 0.6pt,line join=round] (186.61,166.79) -- (186.61, 85.91);

\path[draw=drawColor,line width= 0.6pt,fill=fillColor] (170.01,241.93) --
	(170.01,166.79) --
	(203.20,166.79) --
	(203.20,241.93) --
	(170.01,241.93) --
	cycle;

\path[draw=drawColor,line width= 1.1pt] (170.01,210.77) -- (203.20,210.77);

\path[draw=drawColor,line width= 0.6pt,line join=round] (230.86,254.55) -- (230.86,329.01);

\path[draw=drawColor,line width= 0.6pt,line join=round] (230.86,179.81) -- (230.86, 85.91);

\path[draw=drawColor,line width= 0.6pt,fill=fillColor] (214.26,254.55) --
	(214.26,179.81) --
	(247.45,179.81) --
	(247.45,254.55) --
	(214.26,254.55) --
	cycle;

\path[draw=drawColor,line width= 1.1pt] (214.26,223.10) -- (247.45,223.10);

\path[draw=drawColor,line width= 0.6pt,line join=round] (275.10,266.26) -- (275.10,335.43);

\path[draw=drawColor,line width= 0.6pt,line join=round] (275.10,195.12) -- (275.10,109.57);

\path[draw=drawColor,line width= 0.6pt,fill=fillColor] (258.51,266.26) --
	(258.51,195.12) --
	(291.70,195.12) --
	(291.70,266.26) --
	(258.51,266.26) --
	cycle;

\path[draw=drawColor,line width= 1.1pt] (258.51,235.56) -- (291.70,235.56);

\path[draw=drawColor,line width= 0.6pt,line join=round,line cap=round] ( 27.31, 30.69) rectangle (301.65,355.85);
\end{scope}
\begin{scope}
\definecolor{drawColor}{gray}{0.30}

\node[text=drawColor,anchor=base east,inner sep=0pt, outer sep=0pt, scale=  0.88] at ( 22.36, 42.44) {0};

\node[text=drawColor,anchor=base east,inner sep=0pt, outer sep=0pt, scale=  0.88] at ( 22.36,123.05) {3};

\node[text=drawColor,anchor=base east,inner sep=0pt, outer sep=0pt, scale=  0.88] at ( 22.36,203.67) {6};

\node[text=drawColor,anchor=base east,inner sep=0pt, outer sep=0pt, scale=  0.88] at ( 22.36,284.29) {9};
\end{scope}
\begin{scope}
\definecolor{drawColor}{gray}{0.20}

\path[draw=drawColor,line width= 0.6pt,line join=round] ( 24.56, 45.47) --
	( 27.31, 45.47);

\path[draw=drawColor,line width= 0.6pt,line join=round] ( 24.56,126.09) --
	( 27.31,126.09);

\path[draw=drawColor,line width= 0.6pt,line join=round] ( 24.56,206.70) --
	( 27.31,206.70);

\path[draw=drawColor,line width= 0.6pt,line join=round] ( 24.56,287.32) --
	( 27.31,287.32);
\end{scope}
\begin{scope}
\definecolor{drawColor}{gray}{0.20}

\path[draw=drawColor,line width= 0.6pt,line join=round] ( 53.86, 27.94) --
	( 53.86, 30.69);

\path[draw=drawColor,line width= 0.6pt,line join=round] ( 98.11, 27.94) --
	( 98.11, 30.69);

\path[draw=drawColor,line width= 0.6pt,line join=round] (142.36, 27.94) --
	(142.36, 30.69);

\path[draw=drawColor,line width= 0.6pt,line join=round] (186.61, 27.94) --
	(186.61, 30.69);

\path[draw=drawColor,line width= 0.6pt,line join=round] (230.86, 27.94) --
	(230.86, 30.69);

\path[draw=drawColor,line width= 0.6pt,line join=round] (275.10, 27.94) --
	(275.10, 30.69);
\end{scope}
\begin{scope}
\definecolor{drawColor}{gray}{0.30}

\node[text=drawColor,anchor=base,inner sep=0pt, outer sep=0pt, scale=  0.88] at ( 53.86, 9.68) {15};

\node[text=drawColor,anchor=base,inner sep=0pt, outer sep=0pt, scale=  0.88] at ( 98.11, 9.68) {16};

\node[text=drawColor,anchor=base,inner sep=0pt, outer sep=0pt, scale=  0.88] at (142.36, 9.68) {17};

\node[text=drawColor,anchor=base,inner sep=0pt, outer sep=0pt, scale=  0.88] at (186.61, 9.68) {18};

\node[text=drawColor,anchor=base,inner sep=0pt, outer sep=0pt, scale=  0.88] at (230.86, 9.68) {19};

\node[text=drawColor,anchor=base,inner sep=0pt, outer sep=0pt, scale=  0.88] at (275.10, 9.68) {20};
\end{scope}
\begin{scope}
\definecolor{drawColor}{RGB}{0,0,0}

\node[text=drawColor,anchor=base,inner sep=0pt, outer sep=0pt, scale=  1.10] at (164.48,  -10.64) {\tiny $\log(p)$};
\end{scope}
\begin{scope}
\definecolor{drawColor}{RGB}{0,0,0}

\node[text=drawColor,rotate= 90.00,anchor=base,inner sep=0pt, outer sep=0pt, scale=  1.10] at ( -3.08,193.27) {\tiny $\log(\Delta_p)$};
\end{scope}
\end{tikzpicture}
      \caption{\small $p=n$}
    \end{subfigure}
    \hfill
    \begin{subfigure}{0.45\textwidth}
       \begin{tikzpicture}[x=0.5pt,y=0.5pt]
\definecolor{fillColor}{RGB}{255,255,255}
\begin{scope}
\definecolor{drawColor}{RGB}{255,255,255}
\definecolor{fillColor}{RGB}{255,255,255}

\end{scope}
\begin{scope}
\path[clip] ( 31.71, 30.69) rectangle (301.65,355.85);
\definecolor{fillColor}{RGB}{255,255,255}

\path[fill=fillColor] ( 31.71, 30.69) rectangle (301.65,355.85);
\definecolor{drawColor}{gray}{0.92}

\path[draw=drawColor,line width= 0.3pt,line join=round] ( 31.71, 84.02) --
	(301.65, 84.02);

\path[draw=drawColor,line width= 0.3pt,line join=round] ( 31.71,161.14) --
	(301.65,161.14);

\path[draw=drawColor,line width= 0.3pt,line join=round] ( 31.71,238.25) --
	(301.65,238.25);

\path[draw=drawColor,line width= 0.3pt,line join=round] ( 31.71,315.37) --
	(301.65,315.37);

\path[draw=drawColor,line width= 0.6pt,line join=round] ( 31.71, 45.47) --
	(301.65, 45.47);

\path[draw=drawColor,line width= 0.6pt,line join=round] ( 31.71,122.58) --
	(301.65,122.58);

\path[draw=drawColor,line width= 0.6pt,line join=round] ( 31.71,199.69) --
	(301.65,199.69);

\path[draw=drawColor,line width= 0.6pt,line join=round] ( 31.71,276.81) --
	(301.65,276.81);

\path[draw=drawColor,line width= 0.6pt,line join=round] ( 31.71,353.92) --
	(301.65,353.92);

\path[draw=drawColor,line width= 0.6pt,line join=round] ( 57.84, 30.69) --
	( 57.84,355.85);

\path[draw=drawColor,line width= 0.6pt,line join=round] (101.37, 30.69) --
	(101.37,355.85);

\path[draw=drawColor,line width= 0.6pt,line join=round] (144.91, 30.69) --
	(144.91,355.85);

\path[draw=drawColor,line width= 0.6pt,line join=round] (188.45, 30.69) --
	(188.45,355.85);

\path[draw=drawColor,line width= 0.6pt,line join=round] (231.99, 30.69) --
	(231.99,355.85);

\path[draw=drawColor,line width= 0.6pt,line join=round] (275.53, 30.69) --
	(275.53,355.85);
\definecolor{drawColor}{gray}{0.20}

\path[draw=drawColor,line width= 0.6pt,line join=round] ( 57.84,208.28) -- ( 57.84,265.98);

\path[draw=drawColor,line width= 0.6pt,line join=round] ( 57.84,135.29) -- ( 57.84, 45.47);
\definecolor{fillColor}{RGB}{248,118,109}

\path[draw=drawColor,line width= 0.6pt,fill=fillColor] ( 41.51,208.28) --
	( 41.51,135.29) --
	( 74.16,135.29) --
	( 74.16,208.28) --
	( 41.51,208.28) --
	cycle;

\path[draw=drawColor,line width= 1.1pt] ( 41.51,179.29) -- ( 74.16,179.29);

\path[draw=drawColor,line width= 0.6pt,line join=round] (101.37,217.97) -- (101.37,273.97);

\path[draw=drawColor,line width= 0.6pt,line join=round] (101.37,145.46) -- (101.37, 45.47);

\path[draw=drawColor,line width= 0.6pt,fill=fillColor] ( 85.05,217.97) --
	( 85.05,145.46) --
	(117.70,145.46) --
	(117.70,217.97) --
	( 85.05,217.97) --
	cycle;

\path[draw=drawColor,line width= 1.1pt] ( 85.05,188.61) -- (117.70,188.61);

\path[draw=drawColor,line width= 0.6pt,line join=round] (144.91,233.39) -- (144.91,300.20);

\path[draw=drawColor,line width= 0.6pt,line join=round] (144.91,154.06) -- (144.91, 45.47);

\path[draw=drawColor,line width= 0.6pt,fill=fillColor] (128.59,233.39) --
	(128.59,154.06) --
	(161.24,154.06) --
	(161.24,233.39) --
	(128.59,233.39) --
	cycle;

\path[draw=drawColor,line width= 1.1pt] (128.59,200.20) -- (161.24,200.20);

\path[draw=drawColor,line width= 0.6pt,line join=round] (188.45,242.27) -- (188.45,307.79);

\path[draw=drawColor,line width= 0.6pt,line join=round] (188.45,173.97) -- (188.45, 84.15);

\path[draw=drawColor,line width= 0.6pt,fill=fillColor] (172.12,242.27) --
	(172.12,173.97) --
	(204.78,173.97) --
	(204.78,242.27) --
	(172.12,242.27) --
	cycle;

\path[draw=drawColor,line width= 1.1pt] (172.12,212.65) -- (204.78,212.65);

\path[draw=drawColor,line width= 0.6pt,line join=round] (231.99,259.14) -- (231.99,331.43);

\path[draw=drawColor,line width= 0.6pt,line join=round] (231.99,179.29) -- (231.99, 84.15);

\path[draw=drawColor,line width= 0.6pt,fill=fillColor] (215.66,259.14) --
	(215.66,179.29) --
	(248.32,179.29) --
	(248.32,259.14) --
	(215.66,259.14) --
	cycle;

\path[draw=drawColor,line width= 1.1pt] (215.66,222.83) -- (248.32,222.83);

\path[draw=drawColor,line width= 0.6pt,line join=round] (275.53,276.69) -- (275.53,338.59);

\path[draw=drawColor,line width= 0.6pt,line join=round] (275.53,212.65) -- (275.53,122.83);

\path[draw=drawColor,line width= 0.6pt,fill=fillColor] (259.20,276.69) --
	(259.20,212.65) --
	(291.86,212.65) --
	(291.86,276.69) --
	(259.20,276.69) --
	cycle;

\path[draw=drawColor,line width= 1.1pt] (259.20,249.93) -- (291.86,249.93);

\path[draw=drawColor,line width= 0.6pt,line join=round,line cap=round] ( 31.71, 30.69) rectangle (301.65,355.85);
\end{scope}
\begin{scope}
\definecolor{drawColor}{gray}{0.30}

\node[text=drawColor,anchor=base east,inner sep=0pt, outer sep=0pt, scale=  0.88] at ( 26.76, 42.44) {0};

\node[text=drawColor,anchor=base east,inner sep=0pt, outer sep=0pt, scale=  0.88] at ( 26.76,119.55) {3};

\node[text=drawColor,anchor=base east,inner sep=0pt, outer sep=0pt, scale=  0.88] at ( 26.76,196.66) {6};

\node[text=drawColor,anchor=base east,inner sep=0pt, outer sep=0pt, scale=  0.88] at ( 26.76,273.78) {9};

\node[text=drawColor,anchor=base east,inner sep=0pt, outer sep=0pt, scale=  0.88] at ( 26.76,350.89) {12};
\end{scope}
\begin{scope}
\definecolor{drawColor}{gray}{0.20}

\path[draw=drawColor,line width= 0.6pt,line join=round] ( 28.96, 45.47) --
	( 31.71, 45.47);

\path[draw=drawColor,line width= 0.6pt,line join=round] ( 28.96,122.58) --
	( 31.71,122.58);

\path[draw=drawColor,line width= 0.6pt,line join=round] ( 28.96,199.69) --
	( 31.71,199.69);

\path[draw=drawColor,line width= 0.6pt,line join=round] ( 28.96,276.81) --
	( 31.71,276.81);

\path[draw=drawColor,line width= 0.6pt,line join=round] ( 28.96,353.92) --
	( 31.71,353.92);
\end{scope}
\begin{scope}
\definecolor{drawColor}{gray}{0.20}

\path[draw=drawColor,line width= 0.6pt,line join=round] ( 57.84, 27.94) --
	( 57.84, 30.69);

\path[draw=drawColor,line width= 0.6pt,line join=round] (101.37, 27.94) --
	(101.37, 30.69);

\path[draw=drawColor,line width= 0.6pt,line join=round] (144.91, 27.94) --
	(144.91, 30.69);

\path[draw=drawColor,line width= 0.6pt,line join=round] (188.45, 27.94) --
	(188.45, 30.69);

\path[draw=drawColor,line width= 0.6pt,line join=round] (231.99, 27.94) --
	(231.99, 30.69);

\path[draw=drawColor,line width= 0.6pt,line join=round] (275.53, 27.94) --
	(275.53, 30.69);
\end{scope}
\begin{scope}
\definecolor{drawColor}{gray}{0.30}

\node[text=drawColor,anchor=base,inner sep=0pt, outer sep=0pt, scale=  0.88] at ( 57.84, 9.68) {15};

\node[text=drawColor,anchor=base,inner sep=0pt, outer sep=0pt, scale=  0.88] at (101.37, 9.68) {16};

\node[text=drawColor,anchor=base,inner sep=0pt, outer sep=0pt, scale=  0.88] at (144.91, 9.68) {17};

\node[text=drawColor,anchor=base,inner sep=0pt, outer sep=0pt, scale=  0.88] at (188.45, 9.68) {18};

\node[text=drawColor,anchor=base,inner sep=0pt, outer sep=0pt, scale=  0.88] at (231.99, 9.68) {19};

\node[text=drawColor,anchor=base,inner sep=0pt, outer sep=0pt, scale=  0.88] at (275.53, 9.68) {20};
\end{scope}
\begin{scope}
\definecolor{drawColor}{RGB}{0,0,0}

\node[text=drawColor,anchor=base,inner sep=0pt, outer sep=0pt, scale=  1.10] at (166.68,  -10.64) {\tiny $\log(p)$};
\end{scope}
\begin{scope}
\definecolor{drawColor}{RGB}{0,0,0}

\node[text=drawColor,rotate= 90.00,anchor=base,inner sep=0pt, outer sep=0pt, scale=  1.10] at ( -3.08,193.27) {\tiny $\log(\Delta_p)$};
\end{scope}
\end{tikzpicture} 
       \caption{\small $p=2n$}
    \end{subfigure}
    \caption{\small The figure shows boxplots of the change in sizes of leave-$i$-out active sets (denoted by $\Delta_p$), plotted against the dimension, (on a logarithmic scale) when performing linear regression with the elastic net penalty. The upper and lower edges of the box in each boxplot represent the 1st and the 3rd quartiles respectively, and the black line represents the median. The whiskers extend up to the most extreme value in 1.5 times the interquartile range. The parameters are taken as $p=n$ (left) and $p=2n$ (right), and in either figure $p$ is then varied from 1000 to 10000 with six equal increments on the log scale. We take  20\% of the true coefficients to be non-zero. The design matrix $\bX$ has iid $N(0,1/n)$ rows, and the nonzero coefficients are iid $N(0,1)$. The penalty strengths are $\lambda = 2, \eta=0.5$. We use the \texttt{ElasticNet} function from the Python library \texttt{scikit-learn}.}
    \label{fig:change-size}
\end{figure}
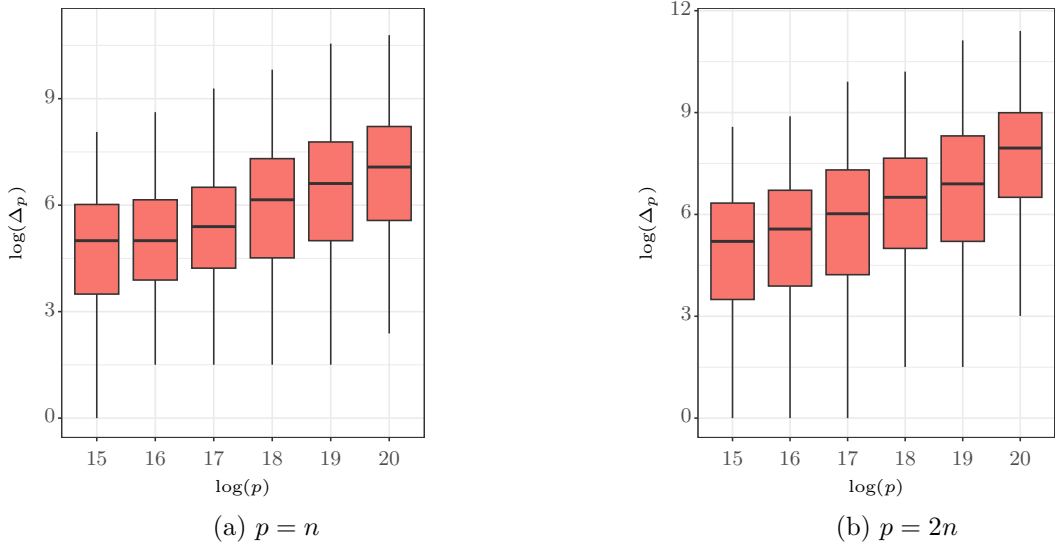
 
Despite the observation in Figure \ref{fig:sim_change_active_set}, extensive empirical results presented in \cite{beirami2017optimal, wang2018approximate, rad2018scalable, stephenson2020approximate} confirm that \eqref{eq:ALO:LASSO} offer an accurate estimation of LO (and OO).  

To understand these two contradictory observations, we ran another simulation that appears in Figure~\ref{fig:change-size}. In this figure  we find that, while the active set does change, the number of changes in the active set (denoted by $\Delta_p$) scales at a sub-linear rate with respect to $p$ (and hence $n$) as $p,n$ increases. Indeed, the comparable points in the boxplots (e.g., the median, maximum) of the logarithms, lie on a line that has slope smaller than one. A linear regression of the medians of $\log(\Delta_p)$ on $\log(p)$ showed a slope of 0.43 when $p=n$ and 0.52 when $p=2n$.
 
This implies that $\Delta_p/p\to0$ as $p,n$ increases. As will be clarified later,  this sub-linear rate of growth that will be proved in Theorem \ref{thm:size-diff-S},  is the main reason that the simulation results confirm the accuracy of \eqref{eq:ALO:LASSO}.


\subsection{Related work}\label{sec:rel}
Various approaches to estimating the out-of-sample error have been proposed. Examples include (but are not limited to) cross validation  \citep{stone1974},  predicted residual error sum of squares \citep{A74}, and generalized cross validation \citep{CW79,golub1979}, just to name a few.

Until recently, the use of $n$-fold cross validation (also known as LO) has been limited due to the high computational cost of repeatedly refitting the model $n$ times, and due to concerns about the high variance, especially when $n$ and $p$ grow unboundedly. Recently, these concerns have been (mostly) alleviated by a large body of work that showed: 1) that the variance of LO in estimating OO goes to zero as $n$ and $p$ grow \citep{kumar2013,bayle2020,rad2020error,patil2022,luo2023iterative}, and 2) that computationally efficient approximations to the leave-one-out cross validation (can) provide statistically reliable estimates of LO and OO \citep{beirami2017optimal,wang2018approximate,rad2018scalable,rad2020error,patil2022,stephenson2020approximate,obuchi2019cross,OW00,CT08,MG13,vehtari2017,VMTSW16,obuchi2016,OK18,xu2021consistent}. Most theoretical work about the consistency of ALO in estimating LO has focused on differentiable regularizers, such as ridge, smoothed LASSO and Huber loss \citep{rad2020error,rad2018scalable,patil2022}. Assuming that the SNR grows unboundedly, as $n$ grows, \cite{stephenson2020approximate} considered $\ell_1$ regularizers. In this regime the optimal value of the regularization parameter $\lambda$ goes to zero, and tuning becomes (nearly) unimportant because of the unbounded SNR.

Despite significant empirical evidence, to the best of our knowledge, there is no theoretical study of the consistency of ALO for non-differentiable regularizers in a regime where $n$ and $p$ grow to infinity while $n/p$ and SNR remains bounded, a framework typical in high dimensional risk estimation problems \citep{DMM11,donoho2016high, maleki2011,MMB15, WaWeMa20, wang2022does,xu2021consistent,guo2022signal,rad2020error,rad2018scalable,patil2022}. Given the importance of risk estimation for non-differertiable estimation problems, this paper addresses the problem in the finite-SNR regime where tuning significantly impacts the selected model and estimated coefficients (as we discuss in Section \ref{assu:A}).

\subsection{Our technical contributions}\label{sec:cont}

 In this paper, our primary focus is on risk estimation for $\ell_1$ regularized problems within the generalized linear model family. Specifically, we aim to establish an upper bound for the error $|{\rm ALO} -{\rm LO}|$ under the conditions of large $n$ and $p$, while maintaining fixed and bounded values for $n/p$ and SNR. In the following, we outline some of the key theoretical innovations that have enabled us to undertake a comprehensive analysis of $|{\rm ALO}-{\rm LO}|$ in this context.

Our initial step involves a smooth approximation $r_{\alpha}(\bz)$ for the $\ell_1$-norm $\|\bz\|_1$, where $\alpha$ is a parameter such that, as $\alpha$ approaches infinity, this approximation becomes increasingly accurate. While smoothing techniques have been extensively employed for deriving approximate minimizers of non-differentiable convex functions, their application as proof techniques in high-dimensional statistics has been unexplored. The primary challenge arises from the fact that as $\alpha \rightarrow \infty$, it becomes considerably difficult to bound the quantities that are of particular interest to statisticians, such as $|{\rm ALO}^{\alpha} -{\rm LO}^{\alpha}|$    in our specific problem. Here, ${\rm ALO}^{\alpha}$ and ${\rm LO}^{\alpha}$ represent the smoothed approximations of ${\rm ALO}$ and ${\rm LO}$, respectively.

In our paper (in the proof of Theorem \ref{th:alo-main}), we have devised an innovative method to bound such quantities, which we anticipate to have broader applications in studying non-differentiable losses and regularizers in high-dimensional settings.

Our smoothing technique, in essence, simplifies the task of bounding $|{\rm ALO}-{\rm LO}|$ by transforming it into the challenge of limiting the changes in the active set between the full-data estimate and the leave-one-out estimate. Hence, in this paper, we develop new techniques for understanding the relationships between two estimates that are using the same samples (proof of Theorem \ref{thm:size-diff-S}). As an example of our approach, we employ our technique to establish an upper bound for the disparity between $\mathcal{S}$ and $\mathcal{S}_{\slash i}$, where $\mathcal{S}$ and $\mathcal{S}_{\slash i}$ denote the locations of non-zero coefficients in $\hat{\bbeta}$ and $\hat{\bbeta}_{\slash i}$, respectively. We anticipate that this technique will prove valuable for other problems, such as the analysis of ensemble methods within the context of high-dimensional settings.

\section{Main theoretical result} \label{sec:main}
We begin by introducing our notation in Section \ref{sec:not} and assumptions in Section \ref{assu:A}. We discuss the assumptions in Section \ref{sec:discus-ass} and show these assumption encompass a large class of typical problems. We discuss the issue of bounded signal-to-noise ratio in Section~\ref{sec:bdd-snr}. Then we present our main theorem in Section \ref{sec:main-th}.

\subsection{Notations}\label{sec:not}
In this manuscript, vectors are denoted with boldfaced lowercase letter, such as $\bx$. Matrices are represented with Boldfaced capital letters, such as $\bX$. Calligraphic letters, such as $\calF$ are used for sets and events. One exception is that we use $[n]$ for $\{1,2,\cdots,n\}$. For matrix $\bX$, $\sigma_{\min}(\bX)$, $\|\bX\|$, $\|\bX\|_{HS}$, ${\rm Tr}(\bX)$ denote the minimum signular value, the spectral norm (maximum singular value), the Hilbert-Schmidt norm, and the trace of the matrix $\bX$ respectively. Suppose $\calF$ represents a subset of indices corresponding to the columns of matrix $\bX$. In such a case, the notation $\bX_{\calF}$ refers to a matrix formed by selecting only those columns of $\bX$ whose indices belong to the set $\calF$. The subscript ``$/i$" refers a quantity related to the leave-$i$-out model, e.g. $\bX_{/i}$ refers to matrix $\bX$ after deleting the $i$th row, and $\hbbeta_{/i}$, $\hbbeta_{/i}^\alpha$ refer to the leave-$i$-out estimate and  the smoothed  leave-$i$-out estimate, respectively. For two probability measures $\mu, \nu$, $W_q(\mu,\nu)$ denote their Wasserstein-$q$ distance. Moreover, we use the following definitions in this paper:
\begin{eqnarray*}
      \dot{\ell}_i(\bbeta) &:=& \frac{\partial \ell(y_i | z) }{\partial z}\bigg|_{z=\bx_i^\top\bbeta}, \quad     \ddot{\ell}_i(\bbeta):= \frac{\partial^2 \ell(y_i | z)}{\partial z^2} \bigg|_{z=\bx_i^\top\bbeta} \\
      \dot{\ell}_{/i}(\bbeta) &:=& \left [\dot{\ell}_1(\bbeta), \cdots, \dot{\ell}_{i-1}(\bbeta), \dot{\ell}_{i+1}(\bbeta), \cdots, , \dot{\ell}_n(\bbeta) \right]^\top, \\ 
      \ddot{\ell}_{/i}(\bbeta) &:=& \left [\ddot{\ell}_1(\bbeta), \cdots, \ddot{\ell}_{i-1}(\bbeta), \ddot{\ell}_{i+1}(\bbeta), \cdots, , \ddot{\ell}_n(\bbeta) \right]^\top.
\end{eqnarray*}
We also denote any polynomials of $\log(n)$ by $\polylog(n)$. For $x,y\in\RR$, we write $x\wedge y$ and $x\vee y$ to denote $\min\{x,y\}$ and $\max\{x,y\}$ respectively. We use the following classic notations for convergence rates: for a sequence of real numbers $a_n$, $a_n = o(1)$ iff $a_n\to 0$, $a_n = O(1)$ iff $|a_n|\le C$ with $C<\infty$ for all sufficiently large $n$, $a_n = \omega(1)$ iff $a_n \to \infty$, and $a_n = \Theta(1)$ iff $0<C_1<|a_n|<C_2$ for all sufficiently large $n$, with $C_2<\infty$. For two sequences $a_n$ and $b_n$, $a_n=o(b_n)$ iff $\frac{a_n}{b_n}=o(1)$, and the other notations are similar. We use similar notations for convergence in probability, e.g. $X_n = o_p(1)$ iff $X_n\to 0$ in probability, and so forth.

\subsection{Assumption group A}\label{assu:A}
The following assumptions have been extensively used in the literature of high-dimensional statistics. We will explain the rationale for making these assumptions in the next section.

\begin{enumerate}

\item[A1] $\bX = (\bx_1,\cdots, \bx_n)^\top$ where $\bx_i \in \mathbb{R}^p$ are iid $N(0,\bSigma)$. Moreover, there exists constants $0<c_{X}\leq C_{X}$ such that $p^{-1} c_X \leq \sigma_{\min}(\bSigma)\leq \sigma_{\max}(\bSigma)\leq p^{-1}C_X$.

\item[A2] $n/p = \gamma_0 \in (0, \infty)$. 

\item[A3] $\phi$ has continuous derivative $\dot{\phi}$, and $\ell(y|z)$ has continuous second derivative w.r.t. $z$.

\item[A4] There exists $\starepsilon>0$, and $q_n, \check{q}_n, \bar{q}_n \in[0,1)$, such that 

\begin{equation}\label{eq:assn_loss}
    \PP\left(\sup_{
    \underset{\bv\in\cD}{
    1\le i\le n}}\dot{\ell}_i(\bv)
    \le \polylog(n)\right)\ge 1-\check{q}_n,
\end{equation}
\begin{equation}\label{eq:assn_loss-2}
    \PP\left(\sup_{
    \underset{\bv\in\cD}{
    1\le i\le n}}\ddot{\ell}_i(\bv)
    \le \polylog(n)\right)\ge 1-q_n,
\end{equation}
\begin{align}\label{eq:assn_loss-3}
    &\PP\left( \sup_{\begin{subarray}{c} 1\leq i\leq n,\\  \bv,\bv'\in \cD \end{subarray}} \frac{\| \ddot{\ell}_i (\bv) - \ddot{\ell}_{/i} (\bv')\|}{\|\bv- \bv'\|_2} \leq \polylog (n) \right)\nonumber\\
     &~\ge 1-\bar{q}_n
\end{align}
where
\begin{eqnarray*}
\cD &:=& \underset{1 \le i\le n}{\cup} \underset{t \in [0,1]}{\cup}
\calB(t\hbbeta +(1-t)\hbbeta_{/i}, \starepsilon), \nonumber \\
\calB(\bw,r) &:=& \left \{\bz: \|\bz-\bw \|_2 \leq r \right \}.
\end{eqnarray*}

\item[A5] $\eta\in (0,1)$, and $\lambda \in (0,\lambda_{\max}]$ for an arbitrary constant $\lambda_{\max}>0$.

\end{enumerate}


\subsection{Discussion of the assumptions} \label{sec:discus-ass}

In the last section, we made five assumptions that will be used for our theoretical results. In this section, we would like to clarify our rationale for making these assumptions. 

\subsubsection{About assumption A1 and A2}\label{ssec:scaling}
To ensures that the $\bx_i^{\top} \bbeta^*$ typically remains finite as $n$ and $p$ grow unboundedly as long as $\|\bbeta^*\|^2/p$ is fixed: 
\[
\frac{c_X}{p} \|\bbeta^*\|^2 \leq \mathbb{E} (\bx_i^{\top} \bbeta^*)^2 \leq  \frac{C_X}{p} \|\bbeta^*\|^2.
\]
For instance, if each element of $\bbeta^*$ remains bounded, then $\bx_i^{\top} \bbeta^*$ will be $O_p(1)$.

In addition to the aforementioned rationale for selecting the presented, there is another compelling justification that we explain below.

Let $\lambda^* (n,p)$ denote the value of $\lambda$ that minimizes the prediction error of $\hat{\bbeta}_{\lambda}$. Suppose that we are interested in the asymptotic setting $n,p \rightarrow \infty$, such that $n/p= \gamma_0$ remains fixed. Then, if Assumption A1 holds, for many problems it has been shown that $\lambda^* (n,p) \rightarrow \bar{\lambda}$, in probability, where $\bar{\lambda}$ is a fixed number in the range $(0, \infty)$. See for instance \cite{MMB15, WaWeMa20, wang2022does}. Inutitively speaking, if under another scaling $\lambda^* (n,p)$ goes to zero as $n,p \rightarrow \infty$, that is an indication of the fact that the estimation problem is becoming easier as $p$ grows and hence a regularizer is not required. Similarly, if under another scaling $\lambda^* (n,p)$ goes to infinity as $n,p \rightarrow \infty$, that is an indication of the fact that the estimation problem is becoming so difficult that we end up choosing zero estimator as the best one. Hence, the scaling we have chosen here seems to be one of the most useful scalings for practice. 

In summary, we believe that Assumptions A1 and A2 provide a good scaling regime for studying risk estimation (or the related problem of hyperparameter tuning) problem. 

\subsubsection{About assumption A4} To introduce a regularity condition for the data generating mechanism and the loss function $\ell$. To clarify this point, we present a proposition below, which outlines a sufficient condition based on simpler regularity conditions for $\ell$ and the data generation mechanism to satisfy A4.

\begin{proposition}\label{prop:assumptionA4}
Suppose that $\PP(|y_i|> \polylog(n)) \leq q_n^y$ for some $q_n^y = o(1/n)$. Furthermore, suppose that $\ell(y|z)$ is three times differentiable with respect to $z$ and that $\ell(y|z)$, $\dot{\ell}(y|z)$, $\ddot{\ell}(y|z),$ and $\dddot{\ell}(y|z)$ grow polynomially in $y, z$, i.e., there exists a positive integer $m$ and a constant $C>0$ such that 
\[
\max\{
|\ell(y|z)|, |\dot{\ell}(y|z)|, |\ddot{\ell}(y|z)|, |\dddot{\ell}(y|z)|
\}
\leq C(1+|y|^m+ |z|^m)
\]
for all $(y,z)$. Then, Assumption A4 holds. 
\end{proposition}
The proof of this result is presented in Section \ref{proof:prop:A4}. The condition of polynomial growth for the loss function is not unduly restrictive, as it encompasses many commonly used loss functions. To illustrate this point, we present a few examples of popular loss functions, viz., squared error, logistic, and Poisson. Therefore, Assumption A4 holds for the majority of loss functions encountered in applications.

For simplicity we assume $\bx_i \sim N(0,\frac1n \II_p)$, but the examples are still valid for a general covarance matrix as in Assumption A1.
\begin{example}[Linear Model]
    Suppose $y_i|\bx_i \sim N(\bx_i^\top \bbeta^*,\sigma^2)$, then $y_i \sim N(0,\sigma^2 + \frac1n \|\bbeta^*\|^2)$. Denote $\nu^2:=\sigma^2 + \frac1n \|\bbeta^*\|^2$, we then have, for arbitrary $q>1$:
    \[
        \PP(|y_i|>\nu\sqrt{2q\log(n)})\leq 2e^{-\frac{2q\nu^2\log(n)}{2\nu^2}}=n^{-q}.
    \]
    If we use negative log-likelihood as loss function, then $\ell(y|z)$ and its derivatives w.r.t. $z$ are:
    \begin{align*}
        \ell(y|z) &= \frac{1}{2\sigma^2} (y-z)^2,\\
        \dot{\ell}(y|z) = \frac{1}{\sigma^2} (z-y)\,;\,
        \ddot{\ell}(y|z) &= \frac{1}{\sigma^2}\,;\,
        \dddot{\ell}(y|z) = 0,
    \end{align*}
    and hence they are all dominated by $\frac{1}{\sigma^2}(1+y^2+z^2)$.
\end{example}
\begin{example}[Logistic regression]
    Suppose $y_i|\bx_i \sim Bernoulli(1/1+e^{-\bx_i^\top\bbeta^*})$, then the boundedness of $|y_i|$ is natually satisfied since $y_i\in\{0,1\}$. The negative log-likelihood loss and its derivatives w.r.t. $z$ are
    \begin{align*}
        |\ell(y|z)| &=
        \abs*{y\log(1+e^{-z}) + (1-y)\log(1+e^z)|}\\
            &\leq 2\log(2) + 2|z|,\\
        |\dot{\ell}(y|z)| &=
        \abs*{\frac{e^z}{1+e^z} - y}\leq 1+|y|,\\
        |\ddot{\ell}(y|z)| &=  
        \abs*{\frac{e^{-z}}{(1+e^{-z})^2}}\leq 1,\\
        |\dddot{\ell}(y|z)| &= 
        \abs*{\frac{e^z - e^{-z}}{(e^{-z}+e^z+2)^2}}\leq 4.
    \end{align*}
    The bound of $\ell(y|z)$ uses the fact that $y\in \{0,1\}$.
\end{example}
\begin{example}[Poisson Regression]
    Suppose $y_i \sim Poisson(\lambda)$ where $\lambda=\log(1+e^{\bx_i^\top\bbeta^*})$, then we have, by Chernoff bound (see, e.g., Exercise 2.3.3 of \cite{vershynin2018high}):
    \[
        \PP(y_i> t)\leq e^{-\lambda} \left(\frac{e\lambda}{t}\right)^t.
        \label{eq:chernoff_poisson}\numberthis
    \]
    Since $\bx_i^\top\bbeta^*\sim N(0, \frac1n \|\bbeta^*\|^2)$, we have 
    \[
        \PP(|\bx_i^\top\bbeta^*|>2\nu \sqrt{\log(n)})\leq n^{-2}
    \]
    where $\nu^2=\frac1n \|\bbeta^*\|^2$. 
    So we have, with probability at least $1-n^{-2}$, that
    \begin{align*}
        \lambda &= \log(1+e^{\bx_i^\top\bbeta^*})
        \leq \log(2) + |\bx_i^\top\bbeta^*|\\
        &\leq \log(2) + 2n^{-1/2}\|\bbeta^*\| \sqrt{\log(n)}.
    \end{align*}
    Setting $t=\log(n)$ in \eqref{eq:chernoff_poisson}, since $\lambda>0$, we have
    
    \begin{align*}
        \PP(y_i>\log(n))&\leq  \left(\frac{e(\log(2)+2\nu\sqrt{\log(n)})}{\log(n)}\right)^{\log(n)}\\
        &\leq \left(\frac{C}{\sqrt{\log(n)}}\right)^{\log(n)}
        =o(n^{-1}).
    \end{align*}
    It can be checked that the negative log-likelihood loss is
    \begin{align*}
        |\ell(y|z)| &= |\log(y!) + \log(1+e^z) - y \log\log(1+e^z)|\\
            &\leq C(y^2+z^2+1)
    \end{align*}
    where $C=1+\log\log(2) + (2\log(2))^{-1}$.

    The derivatives of the loss function w.r.t. $z$ satisfy
    \begin{align*}
        |\dot{\ell}(y|z)| =& 
        \abs*{
        \frac{1}{1+e^{-z}} - \frac{ye^z}{(1+e^z)\log(1+e^z)}}\leq 1+|y|,\\
        |\ddot{\ell}(y|z)| =&~
        \bigg\vert
        y\left(
        \dfrac{e^z}{(1+e^z)\log(1+e^z)}
        \right)^2\\
        &~-
        \dfrac{ye^{z}}{(1+e^z)^2\log(1+e^z)}
        +\dfrac{e^z}{(1+e^z)^2}\bigg\vert\\
        \le&~1+2|y|,\\
        |\dddot{\ell}(y|z)|
        \le&~3+14|y|.
    \end{align*}
    We use the fact that $((1+e^{-z})\log(1+e^z))^{-1}\le 1$ for all $z\in\RR$. We omit the exact expression of $\dddot{\ell}(y|z)$ for brevity.

\end{example}
    
\subsection{Bounded Signal to Noise Ratio}\label{sec:bdd-snr}

We briefly explain the point of `bounded signal to noise ratio (SNR)' mentioned in the introduction. Although our method and the conclusion of this paper do not rely explicitly on the ${\rm SNR}$, the notion of keeping a delicate balance between the signal and noise is one of the most important foundations of our theory. We define the signal to noise ratio as
\[
    {\rm SNR}= \frac{\var(\bx_i^\top\bbeta^*)}{\var(y_i|\bx_i^\top\bbeta^*)}.
\]
It can be shown that the ${\rm SNR}$ is $O_p(1)$ in the three GLM examples in the previous section. In fact, under Assumption A1
\[
    \var(\bx_i^\top\bbeta^*) = (\bbeta^*)^\top \bSigma \bbeta^*\leq \frac{C_X}{p}\|\bbeta^*\|_2^2 = O(1)
\]
and for the three generalized linear models (linear, logistic, and Poisson regression) in the previous subsection, we have
\begin{equation*}
    \dfrac{1}{
    \var(y_i|\bx_i^\top\bbeta^*))}
    =
    \begin{cases}
        \sigma^{-2} \, &\text{for linear}\\
        \left({\rm e}^{-\frac{\bx_i^\top\bbeta^*}{2}}
    +{\rm e}^{\frac{\bx_i^\top\bbeta^*}{2}}\right)^2\,
    &\text{for logistic}\\
        (\log(1+\bx_i^\top\bbeta^*))^{-1}\,
        &\text{for Poisson}
    \end{cases}
\end{equation*}
all of which are $O_p(1)$ using the fact that 
\begin{align*}
    \bx_i^{\top}\bbeta^*&\sim N(0,C_X\|\bbeta^*\|_2^2/p)=O_p(1),\\
    2\le {\rm e}^{-\bx_i^\top\bbeta^*}&
   +{\rm e}^{\bx_i^\top\bbeta^*}+2
   \le 2({\rm e}^{|\bx_i^{\top}\bbeta^*|}+1)=O_p(1).
\end{align*}
Moreover if we fix $\frac1p \|\bbeta^*\|_2^2 = \xi$, then ${\rm SNR}^{-1}$ is also $O_p(1)$.

\subsection{Main theorem}\label{sec:main-th}
In addition to Assumptions A1-A5, our main theorem uses two important sets related to the active set of $\hat \bbeta$ and $\hat \bbeta_{/i}$. In the definition of one of these sets we require the notion of the subgradient vector $g(\bbeta)$ that we introduce below. Define $g(\bbeta)$ as 
\begin{align}
    g(\bbeta):=\frac{1}{\lambda(1-\eta)} \sum_{i=1}^n \dot{\ell}(\by_i|\bx_i^{\top}\bbeta) - \frac{2\eta}{1-\eta}\bbeta
    \label{eq:subgrad}
\end{align}

It can be directly verified that 
$g(\hbbeta) \in \partial \norm{\hbbeta}_1$ belongs to the subgradient of $\ell_1$-term in \eqref{eq: elastic net}.
In fact, the first order derivative of \eqref{eq: elastic net} gives:
$$0 \in \sum_{i=1}^n \dot{\ell}(\by_i|\bx_i^{\top}\hbbeta) + \lambda (1-\eta)\partial \norm{\hbbeta}_1 + 2\lambda \eta \hbbeta$$
where $\partial \norm{\hbbeta}_1$ denotes the subgradient of $\norm{\bbeta}_1$ evaluated at $\hbbeta$. We then obtain \eqref{eq:subgrad} by rearranging the terms. Hence we hereafter refer to $g(\hbbeta)$ as `the' subgradient of $\norm{\hbbeta}_1$. Based on the subgradient, define the following subsets of $[p] = \{1,2, 3, \ldots, p \}$:
\begin{align*}\label{eq:def-S-sets}
	\calS^{(1)}&~:=\{k\in [p]:|\hat{\beta}_k|>\kappa_1(n)\},\\
	\calS^{(0)}&~:=\{k\in [p]:|g(\hat{\bbeta})_k|\le 1-\kappa_0(n)\}, \numberthis
\end{align*}
where $\kappa_1(n)$ and $\kappa_0(n)$ both are $o(1)$ as $n \rightarrow \infty$.\footnote{eg.  $\kappa_0(n)=(\log p)^{1/6}p^{-\delta}$ where $\delta\in (0,\frac{1}{6})$, and $\kappa_1(n)=p^{-1/12}$ as we will see in later sections.} We will clarify our choice of these parameters later. 
Likewise, we can define $\calS_{/i}^{(1)}, \calS_{/i}^{(0)}$  for the leave-$i$-out problems. 

Note that $\calS^{(1)}$ is a subset of the active set of $\hat{\bbeta}$ by only including active elements that are not too close to zero. On the one hand, the condition $\kappa_1(n) \to 0$ implies that $\calS^{(1)}$ is close to the active set. On the other hand, when the gap $\kappa_1(n)$ is selected to be sufficiently large (the choice will be clarified later), it is intuitively expected that only a very small fraction of the indices in $\calS^{(1)}$ will move out of the active set (i.e. the corresponding coefficient becomes zero) in the leave-$i$-out problem. These two points makes $\calS^{(1)}$ a good substitution of the active set discussed in Section \ref{ssec:ALO_non_diff}.

To understand $\calS^{(0)}$, consider the non-active set of $\hbbeta$, i.e. the indices of zero coefficients of $\hbbeta$. $\calS^{(0)}$ is actually a subset of the non-active set, which only includes elements with sub-gradients bounded away from 1, as indicated by its definition $\{k\in [p]:|g(\hat{\bbeta})_k|\le 1-\kappa_0(n)\}$. Again $\kappa_0(n)\to 0$ makes $\calS^{(0)}$ close to the non-active set, and by setting the right rate of the convergence of $\kappa_0(n)$, it should be expected that only a very small number of regression coefficients corresponding to indices in $\calS^{(0)}$ will become nonzero in the leave-$i$-out problem. 

To sum up, our choice of $\calS^{(1)}$ and $\calS^{(0)}$ serve as proxies of the active and non-active set of $\hbbeta$ that are more resilient to the leave-$i$-out procedure.  In other words, we expect the size of $(\calS^{(1)} \cup \calS^{(0)})^c$ to be small compared to $p$. As is clear from the above discussion, the speed at which  $\kappa_1(n)$ and $\kappa_0(n)$ go to zero must be carefully selected to satisfy the two contrasting objectives. On the one hand, we want them both to go to zero as slowly as possible so that $\calS^{(1)}$ and $\calS^{(0)}$ only include elements from which we have strong evidence to be active and non-active, respectively. On the other hand, we want $\kappa_1(n)$ and $\kappa_0(n)$ to go to zero as fast as possible, so that $(\calS^{(1)} \cup \calS^{(0)})^c$ is as small as possible. The details are presented in Theorem \ref{thm:size-diff-S}.

To describe our first theoretical result, consider the following sets:
    \begin{align*}\label{eq:def-B-sets}
 &\calB_{1,i}:=\calS^{(1)}\cap\calS^{(1)}_{/i},
 \calB_{0,i}:=\calS^{(0)}\cap\calS^{(0)}_{/i}\\
        &\calB_{1,i,+}:=\calB_{1,i}\cap \{k: \hbeta_k\cdot\hbeta_{/i,k}>0 \}.
        \numberthis
    \end{align*}
Heuristically, $\calB_{1,i}$ contains the indices of large coefficients that remain large after leaving the $i^{\rm th}$ observation out, and $\calB_{0,i}$ contains the indices of zero coefficients with small subderivative (of $\ell_1$-norm) that remain so after leaving the $i^{\rm th}$ observation out. Note that  $\calB_{1,i}$ and $\calB_{0,i}$ are mutually exclusive, and $\calB_{1,i,+}$ is the subset of $\calB_{1,i}$ that rules out the coefficients with flipped signs. Ideally we wish leaving-i-out would not change the fact of each coefficient being zero or non-zero, i.e., $\calB_{1,i}\cup\calB_{0,i} = [p]$. This is clearly not true according to Figure \ref{fig:sim_change_active_set}, but the violation is not substantial. Indeed, Figure~\ref{fig:change-size} shows that the size of the leave-$i$-out perturbations, i.e., $|(\calB_{1,i,+}\cup\calB_{0,i})^c|$ scale at a rate which is slower than $p$. We rigorously prove that $(\calB_{1,i,+}\cup\calB_{0,i})^c$ is indeed a small set under reasonable assumptions (see Theorem \ref{thm:size-diff-S}). Our main theoretical result proves that the difference between LO and the ALO formula of \eqref{eq:ALO:LASSO} is proportional to the size of $(\calB_{1,i}\cup\calB_{0,i})^c$.

\begin{theorem}\label{th:alo-main}
Under Assumptions A1-A5, for a sufficiently large constant $C>0$, let $1\le d_n\le p/C$ be such that
    $$
    \max_{1\le i \le n}
    |(\calB_{0,i}\cup
    \calB_{1,i,+})^c
    |
    \le d_n
    $$
    with probability at least $\tilde{q}_n$. Then we have
	\begin{align*}
	    &~|{\rm ALO}-{\rm LO}|\\
	\le&~ \dfrac{\polylog(n)}{\lambda^3\eta^3(1\wedge\lambda\eta)^3}
    \sqrt{\dfrac{d_n}{n\lambda\eta}}
    +\dfrac{d_n\polylog(n)}{n\lambda^2\eta^2}
    +
    \sqrt{\dfrac{\polylog(n)}{n\lambda\eta}}
\end{align*}
with probability at least $1-(n+1){\rm e}^{-p}-(n+2)p^{-d_n}-2q_n-2\check{q}_n-2\bar{q}_n-2\tilde{q}_n$. 
\end{theorem}

Before we prove the result let us clarify the statement of the theorem. First note that while the bound we have obtained for the difference between ALO and LO is a finite sample bound, one way to interpret and understand the result is through the asymptotic setting we described in Section \ref{ssec:scaling}, i.e. $n, p \rightarrow \infty$ while $n/p = \gamma_0$. Under this asymptotic regime, as argued in \cite{rad2018scalable} for twice differentiable losses and regularizers, LO offers a consistent estimate of out-of-sample risk estimate, while any $K$-fold cross validation for $K$ fixed does not. According to Theorem \ref{th:alo-main}, if $d_n = o(p^{\zeta})$ with $\zeta<1$, the the upper bound of Theorem \ref{th:alo-main} goes to zero as $n,p \rightarrow \infty$. The growth rate of $d_n$ will be discussed in Section \ref{sec:linearregression}.

\begin{proof}[Proof of Theorem~\ref{th:alo-main}]

As we mentioned in the introduction, one of ingredients of the proof is the smoothing idea that we would like to describe first. Define
\begin{equation}
	h(\bbeta):=
 \left\{
 \sum_{i=1}^n\ell(y_i|\bx_i^\top\bbeta)+\lambda (1- \eta)  \|\bbeta\|_1 + \lambda \eta\|\bbeta\|_2^2
 \right\}.
\end{equation}
If $h(\bbeta)$ were a differentiable function we could use one step of the Newton method to obtain an approximate leave-one-out estimate. However, the main issue here is the non-differentiability of $\|\cdot\|_1$. For that reason we start with approximating the $\|\cdot\|_1$ with a smooth function. 
Let 
\begin{equation}\label{def:smooth:ell1}
r_\alpha^{(1)}(z) = \frac{1}{\alpha}\left( \log(1+e^{\alpha z}) + \log(1+e^{-\alpha z}) \right)
\end{equation}
denote the $\alpha$-smoothed $\ell_1$ regularizer. The following lemma proved in \cite{rad2018scalable} shows the accuracy of this approximation:
\begin{lemma}[Lemma 13 in \cite{rad2018scalable}]\label{lem:acc:r1}
    If $r_\alpha^{(1)}(z)$ denotes the $\alpha$-smoothed $\ell_1$ regularizer. Then we have
    $$ r_\alpha^{(1)}(z) \geq \abs{z}, $$
    and 
    $$ \sup_z \abs{ r_\alpha^{(1)}(z) - \abs{z} } \leq \frac{2\log2}{\alpha}. $$
\end{lemma}
This lemma suggests that for large values of $\alpha$, $r_\alpha^{(1)}(z)$ can be a good approximation of $|z|$. Hence, based on this approximation we now introduce the smoothed cost function
\begin{equation}
	h_{\alpha}(\bbeta)=
	\sum_{j=1}^n\ell(y_i;\bx_i^{\top}\bbeta)
	+\lambda r_{\alpha} (\bbeta),
 \end{equation}
 where 
 \begin{equation}\label{eq:r_alpha:def}
 r_\alpha(\bbeta) := (1-\eta)\sum_{i=1}^pr^{(1)}_{\alpha}(\beta_i)
	+\eta\bbeta^{\top}\bbeta
 \end{equation}
denotes the smoothed regularizer. If we define
\begin{equation}
\hat{\bbeta}^\alpha = \arg\min_{\bbeta} h_{\alpha}(\bbeta),
\end{equation}
and $\hat{\bbeta}^{\alpha}_{/i}$ as its leave-one-out estimate, then we can use Theorem 3 of \cite{rad2018scalable} to prove that 
\begin{eqnarray}\label{eq:ALO:smoothed}
\lefteqn{\max_{1\leq i\leq n} \abs*{\bx_i^\top \hbbeta^{\alpha}_{/i} - \bx_i^\top \hbbeta^{\alpha} - \left( \frac{\dot{\ell}_i(\hbbeta^{\alpha})}{\ddot{\ell}_i(\hbbeta^{\alpha})} \right) \left( \frac{H^{\alpha}_{ii}}{1-H^{\alpha}_{ii}} \right)}} \nonumber \\
&\leq& \frac{C_0(\alpha)\polylog(n)}{\sqrt{p}}, \hspace{4cm}
\end{eqnarray}
where $ \bH^{\alpha}$ is defined as 
\begin{equation}\label{eq:Hmatalpha}
\small{\bX
	\left(\lambda{\rm diag}[\ddot{r}_{\alpha}(\hat{\bbeta}^{\alpha})]+\bX^\top[{\rm diag}(\ddot{\ell}(\bbeta^{\alpha}))]\bX\right)^{-1}
	\bX^\top
	{\rm diag}[\ddot{\ell}(\hat{\bbeta}^{\alpha})].}
\end{equation}
For completeness we have mentioned Theorem 3 of \cite{rad2018scalable} in Section \ref{sec:proofs} (Theorem \ref{thm:kamiarandI}). 
The main issue in the approximation of \eqref{eq:ALO:smoothed} is that $C_0(\alpha) \rightarrow \infty$ as $\alpha \rightarrow \infty$. This creates a dilemma. On one hand we would like $\alpha$ to be large to make $|r^{(1)}_\alpha(z)-|z||$ small. But on the other hand, the upper bound in Theorem 3 of \cite{rad2018scalable} goes to infinity as $\alpha \rightarrow \infty$. In addition to these two problems, as will be discussed later, some of the elements of ${\rm diag}[\ddot{r}_{\alpha}(\hat{\bbeta}^{\alpha})]$ go to infinity as $\alpha \rightarrow \infty$ that may cause inaccuracies and instabilities if this procedure is used in practice. Despite the fact that smoothing idea is not useful for approximating the leave-one-out risk of the elastic net,  as will be shown in this proof, it still serves as a good theoretical tool for proving the accuracy of \eqref{eq:ALO:LASSO}. Hence, we pursue two goals here:
\begin{itemize}
\item Use a different strategy than the one pursued in \cite{rad2018scalable} to find an upper bound on $\underset{1\leq i\leq n}{\max}\abs*{\bx_i^\top \hbbeta^{\alpha}_{/i} - \bx_i^\top \hbbeta^{\alpha} - \left( \frac{\dot{\ell}_i(\hbbeta^{\alpha})}{\ddot{\ell}_i(\hbbeta^{\alpha})} \right) \left( \frac{H^{\alpha}_{ii}}{1-H^{\alpha}_{ii}} \right)}$. Our new bounds will not go off to infinity as $\alpha \rightarrow \infty$.

\item We will then prove that for large values of $\alpha$, $\bH^{\alpha}_{ii}$ is close to $\bH_{ii}$ used in \eqref{eq:ALO:LASSO}. 
\end{itemize}
To understand the challenge for achieving the above two goals, let us start with the following lemma:

\begin{lemma}[Lemma 14 in \cite{rad2018scalable}]\label{lem:secondderivativesmoothr}
    $r_\alpha^{(1)}(z)$ is infinitely many times differentiable, and
    $$\dot{r}_\alpha^{(1)}(z) = \frac{e^{\alpha z} - e^{-\alpha z}}{e^{\alpha z} + e^{-\alpha z} + 2},
    $$
    $$\ddot{r}_\alpha^{(1)}(z) =  \frac{2\alpha}{(e^{\alpha z} + e^{-\alpha z} + 2)^2}.
    $$
\end{lemma}

Suppose that $z_{\alpha} = O(\frac{1}{\alpha})$. Then, as $\alpha \rightarrow \infty$, $\ddot{r}_\alpha^{(1)}(z) \rightarrow \infty$. It may seem to the reader that $z_{\alpha} = O(\frac{1}{\alpha})$ is a condition that may not happen and hence it won't cause any issues. However, this is not the case. In fact, as will be shown in the next lemma, many of the regression coefficients satify the condition $|\hat{\bbeta}_{/i,k}^{\alpha}|= O(\frac{1}{\alpha}\log p)$. For these elements as $\alpha \rightarrow \infty$, $\ddot{r} (\hat{\bbeta}_{/i,k}^{\alpha}) \rightarrow \infty$.

\begin{lemma}\label{lem:beta-size}
	Suppose Assumptions A1-A5 hold. Let $\calS^{(1)}$ and $\calS^{(0)}$ be as defined in (\ref{eq:def-S-sets}).
    Then we have:
	\begin{enumerate}
		\item $\max\limits_{1\le i\le n}\|\hat{\bbeta}_{/i}-\hat{\bbeta}_{/i}^{\alpha}\|\le \sqrt{\dfrac{4\log(2)p}{\alpha\eta}}$.
            \item $\|\hbbeta^{\alpha} - \hbbeta^{\alpha}_{/i} \| \le \frac{|\dot{\ell}(\hbbeta^{\alpha})|\|\bx_i\|}{2\lambda\eta}$
            \item $\|\hbbeta - \hbbeta_{/i} \| \le \frac{|\dot{\ell}(\hbbeta)|\|\bx_i\|}{2\lambda\eta}$.

            \item$\max\limits_{1\le i\le n} \norm{g(\hbbeta) - g(\hbbeta_{/i})}
                \le \frac{\polylog(n)}{\lambda^2\eta(1-\eta)}$ with probability at least $1-q_n-{\rm e}^{-p}-\check{q}_n-ne^{-p/2}$.
		\item Suppose $\alpha = \omega\left(\frac{p}{\kappa_1^2\eta}\right)$, then for large enough $p$, 
            \[
                \min\limits_{0\le i\le n}
                \min\limits_{k\in \calS^{(1)}_{/i}}
                |\hat{\beta}^{\alpha}_{/i,k}|\ge \dfrac{\kappa_1}{2}.
            \]
		\item Suppose $\alpha = \omega\left(\frac{n\polylog(n)}{\kappa_0^2\lambda^2(1-\eta)\eta}\right)$, then for large enough $p$, with probability at least $1-q_n-e^{-p}$:
        \[
            \max_{0\le i\le n}
            \max\limits_{k\in \calS^{(0)}_{/i}}
                |\hat{\beta}^{\alpha}_{/i,k}|
                \leq \frac{1}{\alpha}\log\left(\frac{4}{\kappa_0}\right).
        \]

        \end{enumerate}
\end{lemma}

The proof of this lemma is presented in Section \ref{ssec:proof:lemma3}.

Part (1) of this lemma confirms what we expected. In fact the difference $\|\hat{\bbeta}_{/i}-\hat{\bbeta}_{/i}^{\alpha}\|$ shrinks as $\alpha \rightarrow \infty$. We will use this bound later in our proof. 

Part (6) of Lemma \ref{lem:beta-size} confirms the claim we made before the proof. If $k \in \calS_{/ i}^{(0)}$ then $|\hat{\beta}^{\alpha}_{/i,k}|$ is proportional to $\alpha^{-1}$. Hence, combined with  \ref{lem:secondderivativesmoothr} we expect the second derivative of the smoothed regularizer at $|\hat{\beta}^{\alpha}_{/i,k}|$ to go to infinity as $\alpha \rightarrow \infty$ fast enough. Note that this issue does not happen for $k \in \calS_{/ i}^{(1)}$ as confirmed by part (5) of this lemma; In fact, for $k \in \calS_{/ i}^{(1)}$ since $|\hat{\beta}^{\alpha}_{/i,k}|$ is bounded away from zero, we can use Lemma \ref{lem:secondderivativesmoothr}  to show that as $\alpha \rightarrow \infty$, $\ddot{r} (\hat{\beta}^{\alpha}_{/i,k})$ goes to zero. 

Another source of difficulty in handling the smoothed regularizer is that we do not have much control over the curvature, $\ddot{r}_\alpha^{(1)} (\hat{\beta}^{\alpha}_{/i,k})$ when $k \in (\calS_{/ i}^{(1)}\cup \calS_{/ i}^{(0)})^c$. Keeping this issue in mind, let us first simplify the error between the leave-one-out cross validation risk and the ALO for the smoothed problem. To simplify the calculations we introduce the following notations. Let $\dot{r}_\alpha(\btheta)$ denote the vector $[\dot{r}_\alpha(\theta_1), \dot{r}_\alpha(\theta_2), \ldots, \dot{r}_\alpha (\theta_p) ]^{\top}$. Similarly, 
\[
\dot{\ell}({\btheta}) := [\dot{\ell}(y_1; {\bx}_1^\top \btheta), \dot{\ell}(y_2; \bx_2^\top \btheta), \ldots,  \dot{\ell}(y_n; \bx_n^\top \btheta)]^T
\]

and $\dot{\ell}_{/ i}({\btheta})$ is the same vector as $\dot{\ell}({\btheta})$ except that its $i^{\rm th}$ element is removed. Furthermore, define $\bff_{/i}(\btheta)$ as the gradient of $h_{\alpha}(\cdot)$ at $\btheta$, i.e.,
\begin{equation}
\bff_{/i}(\btheta):=\lambda\dot{r}_\alpha(\btheta)+\bX_{/i}^\top\dot{\ell}_{/i}(\btheta). 
\end{equation}
Notice that $\bff_{/i}(\hat{\bbeta}^{\alpha}_{/i})=0$, where $\hat{\bbeta}^{\alpha}_{/i}$ is the true LO estimate.
Similarly, define the Hessian of $h_\alpha(\cdot)$ and its leave-one-out coounter part as
\begin{equation}\label{eq:jaco}
	\bJ(\btheta)=\lambda{\rm diag}(\ddot{r}_{\alpha}(\btheta))
	+\bX^\top{\rm diag}[\ddot{\ell}(\btheta)]\bX,
\end{equation}
and
\begin{equation}\label{eq:jaco_lo}
	\bJ_{/i}(\btheta)=\lambda{\rm diag}(\ddot{r}_{\alpha}(\btheta))
	+\bX_{/i}^\top{\rm diag}[\ddot{\ell}_{/i}(\btheta)]\bX_{/i}.
\end{equation}
By using the first order optimality conditions, we have
\begin{eqnarray}
\bff(\hat{\bbeta}^\alpha)&=&0, \nonumber  \\
\bff_{/i}(\hat{\bbeta}_{/i}^\alpha) &=&0.
\end{eqnarray}
Define $\Delta^{\alpha}_{/i} := \hat{\bbeta}_{/i}^\alpha- \hat{\bbeta}^\alpha$. Using the multivariate mean-value theorem we have
\begin{eqnarray}\label{eq:grad1}
0 &=& \bff_{/i}(\hat{\bbeta}_{/i}^\alpha) \nonumber \\
&=& \bff_{/i} (\hat{\bbeta}^\alpha + {\Delta^*}_{/i})  \nonumber \\
&=& \bff_{/i} (\hat{\bbeta}^\alpha) + \left(\int_0^1 \bJ_{/i} (\hat{\bbeta}^\alpha + t \Delta^{\alpha}_{/i}) dt \right)\Delta^{\alpha}_{/i}.
\end{eqnarray}
Moreover,
\begin{eqnarray}\label{eq:grad2}
0 &=& \lambda {  \dot{r}}(\hat{\bbeta}^\alpha)+ {X}^\top \dot{\ell} (\hat{\bbeta}^\alpha)  = \bff_{/i} (\hat{\bbeta}^\alpha) + \dot{\ell}_i(\hat{\bbeta}^\alpha) {\bx_i}.
\end{eqnarray}
Combining \eqref{eq:grad1} and \eqref{eq:grad2} we have
\begin{eqnarray}\label{eq:deltaLO_1}
	\lefteqn{\Delta_{/i}^\alpha=\hat{\bbeta}_{/i}^\alpha-\hat{\bbeta}^\alpha} \nonumber \\
	&=&-\dot{\ell}_i(\hat{\bbeta}^\alpha)
	\left(\int_0^1\bJ_{/i}(t\hat{\bbeta}^\alpha
	+(1-t)\hat{\bbeta}_{/i}^\alpha)dt\right)^{-1}
	\bx_i. \hspace{.5cm}
\end{eqnarray}
As is clear from \eqref{eq:beta_nwtn}, the ALO approximation of $\Delta_{/i}^\alpha$ is  given by 
\begin{equation}\label{eq:deltaALO_1}
	\hat{\Delta}_{/i}^\alpha=\dot{\ell}_i(\hat{\bbeta}^\alpha)
	\left(\bJ_{/i}(\hat{\bbeta}_{/i}^\alpha-\Delta_{/i}^\alpha)\right)^{-1}\bx_i.
\end{equation}
Also, it is straightforward to see that
\begin{align}\label{eq:ALO-LO1}
&\abs*{{\rm ALO}^\alpha-{\rm LO}^\alpha} \nonumber \\
&= \frac{1}{n}\sum_{i=1}^n (\phi(y_i, \bx_i^{\top} \hat{\bbeta}^{\alpha}_{/ i}) - \phi(y_i, \bx_i^{\top} (\hat{\bbeta}^{\alpha} + \hat{\Delta}_{/i}^\alpha)) \nonumber \\
&\le \max_{1\le i\le n}\abs*{\dot{\phi}(y_i,\bx_i^\top\tilde{\bbeta}_i^\alpha)}
\cdot \dfrac{1}{n}
\cdot \sum_{i=1}^n\abs*{\bx_i^\top\Delta_i^\alpha-\bx_i^\top\hat{\Delta}_{/i}^\alpha},
\end{align}
where $\tilde{\bbeta}^{\alpha}_i$ is a point on the line that connects $\hat{\bbeta}^{\alpha}$ and $\hat{\bbeta}^{\alpha}_{/ i}$. Similar to the proof of Proposition \ref{prop:assumptionA4}, we can see that for many reasonable models, $\max_{1\le i\le n}\abs*{\dot{\phi}(y_i,\bx_i^\top\tilde{\bbeta}_i^\alpha)}=O_p(\polylog(n))$. Hence, it is enough to obtain an upper bound for $\bx_i^\top\Delta_i^\alpha-\bx_i^\top\hat{\Delta}_{/i}^\alpha$.
From equations \eqref{eq:deltaLO_1} and \eqref{eq:deltaALO_1}, we have
\begin{align*}\label{eq:xdeldiff1}
&~\abs*{\bx_i^\top\Delta_{/i}^\alpha-\bx_i^\top\hat{\Delta}_{/i}^\alpha}\\
&\le  \,|\dot{\ell}_i(\hat{\bbeta}^\alpha)|\times\\
&~\bx_i^\top
\bigg[ \left(\int_0^1\bJ_{/i}(t\hat{\bbeta}^\alpha
+(1-t)\hat{\bbeta}_{/i}^\alpha)dt\right)^{-1}\\
&\hspace{1cm}
-
\left(\bJ_{/i}(\hat{\bbeta}_{/i}^\alpha-\Delta_{/i}^\alpha)\right)^{-1}
\bigg]\bx_i\\
&\le  \,|\dot{\ell}_i(\hat{\bbeta}^\alpha)|\times\\
&~ \bx_i^\top
\bigg[ \left(\int_0^1\bJ_{/i}(t\hat{\bbeta}^\alpha
+(1-t)\hat{\bbeta}_{/i}^\alpha)dt\right)^{-1} -
\left(\bJ_{/i}(\hat{\bbeta}_{/i}^\alpha)\right)^{-1}
\bigg]\bx_i\\
&\,+ |\dot{\ell}_i(\hat{\bbeta}^\alpha)|
\bx_i^\top
\left[ \left(\bJ_{/i}(\hat{\bbeta}_{/i}^\alpha)\right)^{-1}
-\left(\bJ_{/i}(\hat{\bbeta}_{/i}^\alpha-\Delta_{/i}^\alpha)\right)^{-1}
\right]\bx_i. \numberthis
\end{align*}
We again emphasize that we cannot let $\alpha \rightarrow \infty$ in these expressions, since some of the elements of $\bJ$ matrix go to infinity. Hence we have to find proper ways for obtaining an upper bound for \eqref{eq:xdeldiff1} for large values of $\alpha$. As is clear from this discussion, we have to be careful about $\ddot{r}_{\alpha}(\hat{\beta}^{\alpha}_{/i,k})$. The next lemma provides some information about these quantities:

\begin{lemma}\label{lem:rdderiv} Suppose the assumptions of Lemma~\ref{lem:beta-size} hold, and assume 
$\alpha = \omega\left(\frac{n\polylog(n)}{\kappa_0^2\lambda^2(1-\eta)\eta}
\vee 
\frac{p}{\kappa_1^2\eta}
\right)
$. 
Then the following statements hold with probability at least $1-q_n-e^{-p}$:
\begin{enumerate}
    \item $\forall i, \forall k\in \calB_{0,i}$:
    $$
    \int_0^1\ddot{r}_{\alpha}(\hat{\beta}^{\alpha}_{/i,k}-t\Delta^{\alpha}_{/i,k})dt\ge 2\eta + \frac18 \alpha (1-\eta) \kappa_0 
    $$
    \item $\forall i, \forall k\in \calB_{0,i}$:
    \begin{align*}
        \ddot{r}_{\alpha}(\hat{\beta}_k^{\alpha}),\; \ddot{r}_{\alpha}(\hat{\beta}_{/i, k}^{\alpha}) \geq 2\eta + \frac18 \alpha (1-\eta) \kappa_0.
    \end{align*}
    \item $\forall i, \forall k\in \calB_{1,i,+}$:
    $$\abs*{\int_0^1\ddot{r}_{\alpha}(\hat{\beta}^{\alpha}_{/i,k}-t\Delta^{\alpha}_{/i,k})dt
		-2\eta}
		\le 2\alpha e^{-\frac12\alpha\kappa_1}.$$
    \item $\forall i, \forall k\in \calB_{1,i,+}$:
	$$\abs*{
		\ddot{r}_{\alpha}(\hat{\beta}^{\alpha}_{k})
		-2\eta}
	\le 2\alpha e^{-\frac12\alpha\kappa_1},$$	
	$$\abs*{
		\ddot{r}_{\alpha}(\hat{\beta}_{/i,k}^{\alpha})
		-2\eta}
	\le 2\alpha e^{-\frac12\alpha\kappa_1}.$$	
\end{enumerate}
\end{lemma}

Note that the rate assumption on $\alpha$ ensures that both rates hold in both Part 5 and Part 6 of Lemma \ref{lem:beta-size}. The proof of Lemma \ref{lem:rdderiv} is presented in Section \ref{ssec:proof:lemma4}. This lemma confirms that we have some control over the curvature of the regularizer on sets $\calB_{0,i}$ and $\calB_{1,i,+}$. The following lemma enables us to obtain an upper bound for the error between the leave-one-out estimate and ALO by finding an upper bound on the error over the set $\calB_0^{c}$. In other words, the next lemma enables us to remove the indices $k$ for which we are certain $\ddot{r}_{\alpha}(\hat{\beta}^{\alpha}_{k})$ converge to infinity from our analysis.


\begin{theorem}\label{th:rduc2supp} Suppose Assumptions A1-A5 hold, the conclusions of Lemma~\ref{lem:beta-size} are true, and $\alpha\ge \dfrac{4p\log(2)}{(\starepsilon)^2\eta (1- \eta)}$. Then we have
	\begin{align*}\label{eq:rduc_int}
		\bigg\vert
			\bx_i^\top
   \bigg(\int_0^1&
   \bJ_{/i}(t\hat{\bbeta}^\alpha
		+(1-  t)\hat{\bbeta}_{/i}^\alpha)dt\bigg)^{-1}\bx_i\\
           -\bx_{i,\calB_{0,i}^c}^\top
			\big(\lambda&{\rm diag}(\ddot{\underline{r}}_{\calB_{0,i}^c}^{\alpha/i})
			+\bX^\top_{/i,\calB_{0,i}^c}{\rm diag}
			(\ddot{\underline{\ell}}^{\alpha/i})\bX_{/i, \calB_{0,i}^c}\big)^{-1}\bx_{i,\calB_{0,i}^c}\bigg\vert\\
		\le&~ 
            \dfrac{16 \|\bx_{i}\|^2}{\lambda\alpha(1-\eta)\kappa_0} \Big(\frac{\polylog(n)\|\bX^{\top}\bX\|}{2\lambda\eta}+1\Big)^2,\numberthis
	\end{align*}
	and similarly
	\begin{align*}\label{eq:rduc_beta_del_i}
			&~\bigg\vert
			\bx_i^\top\left(\bJ_{/i}(\hat{\bbeta}^\alpha)\right)^{-1}\bx_i\\
		&~-\bx_{i,\calB_{0,i}^c}^\top 
			(\lambda{\rm diag}(\ddot{\utilde{r}}_{\calB_{0,i}^c}^{\alpha/i})
			+\bX^\top_{\calB_{0,i}^c}{\rm diag}
			(\ddot{\utilde{\ell}}^{\alpha/i})\bX_{\calB_{0,i}^c})^{-1}\bx_{i,\calB_{0,i}^c} \bigg\vert\\
		&~\le \dfrac{16 \|\bx_{i}\|^2}{\lambda\alpha(1-\eta)\kappa_0} \Big(\frac{\polylog(n)\|\bX^{\top}\bX\|}{2\lambda\eta}+1\Big)^2\numberthis
	\end{align*}
 Here 
	$
	\ddot{\underline{r}}_{k}^{\alpha/i}:=\int_{-1}^0(\ddot{r}_{\alpha}(\hat{\bbeta}^\alpha_{/i}+t\Delta_{/i}^\alpha))_kdt
	$, 
	$\,\,
	\ddot{\underline{\ell}}^{\alpha/i}_k
	:=\int_{-1}^0
	(\ddot{\ell}_{/i}(\hat{\bbeta}^\alpha_{/i}+t\Delta_{/i}^\alpha))_kdt,
	$
	$\,\, \ddot{\utilde{r}}^{\alpha/i}:= \ddot{r}_{\alpha}(\hat{\bbeta}_{/i}^\alpha-\Delta_{/i}^\alpha)
	$
	and
	$\ddot{\utilde{\ell}}^{\alpha/i}:=\ddot{\ell}_{/i}(\hat{\bbeta}_{/i}^\alpha-\Delta_{/i}^\alpha)
	$.
	
\end{theorem}
The complete proof of this theorem is presented in Section \ref{ssec:proof:thm2}.  

Applying \eqref{eq:xdeldiff1} and Theorem \ref{th:rduc2supp}, we have reduced the problem to the quadratic forms on the subset $\calB_0^c$. The main remaining difficulty is that for the indices outside $\calB_{0,i}^c\setminus \calB_{1,i,+}$ we do not have much control over $\ddot{r}_{\alpha}(\hat{\beta}^{\alpha}_{k})$. So the question is whether those terms can cause any issue in our approximations or not. Our next theorem will show that these elements will not cause any issue if there are not too many of them. More specifically, in the asymptotic regime we are interested in, i.e. the asymptotic regime in which $n,p$ grow at the same rate, if the size of the set $|\calB_{0,i}\setminus \calB_{1,i,+}|$ is sublinear in $n$, then the difference between ALO and LO converges to zero.

\begin{theorem}\label{thm:FINAL_STEP}
Suppose the assumptions of Lemma~\ref{lem:beta-size} hold, and assume $\alpha=\omega\left(\frac{n\polylog(n)}{\kappa_0^2\kappa_1^2\lambda(1-\eta))\eta}\right)$. Moreover, for a sufficiently large constant $C>0$, let  $1\le d_n\le p/C$ be such that
    $$
    \max_{1\le i \le n}
    |\calB_{0,i}^c\setminus
    \calB_{1,i,+}
    |
    \le d_n
    $$
    with probability at least $1-\tilde{q}_n$. Let $\calF$ denote a set such that $\calB_{1,i,+}\subset\calF \subset \calB_{0,i}^c$. Then we have
    \begin{align*}
    &\bigg\vert
			\bx_{i,\calF}^\top 
			(\lambda{\rm diag}(\ddot{\underline{r}}_{\calF}^{\alpha/i})
			+\bX^\top_{\calF}{\rm diag}	(\ddot{\underline{\ell}}^{\alpha/i})\bX_{\calF})^{-1}\bx_{i,\calF}\\
    &~-
    \bx_{i,\calF}^\top 
			(\lambda{\rm diag}(\ddot{\utilde{r}}_{\calF}^{\alpha/i})
			+\bX^\top_{\calF}{\rm diag}	(\ddot{\utilde{\ell}}^{\alpha/i})\bX_{\calF})^{-1}\bx_{i,\calF}
   \bigg\vert\\
   \le&~ \dfrac{\polylog(n)}{\lambda^3\eta^3(1\wedge\lambda\eta)^3}
    \sqrt{\dfrac{d_n}{n\lambda\eta}}
    +\dfrac{Cd_n}{n\lambda^2\eta^2}
    +
    \sqrt{\dfrac{C\log p}{n\lambda\eta}}
\end{align*}
with probability at least $1-(n+1){\rm e}^{-\frac{p}{2}}-(n+2)p^{-d_n}-2q_n-2\check{q}_n-2\bar{q}_n-2\tilde{q}_n$, for sufficiently large $p$.
\end{theorem}
The proof of this claim is long and will be presented in Section \ref{ssec:proof:thm:3}. We will now use this theorem to complete the proof of Theorem \ref{th:alo-main}. 

First note that the leave-one-out cross validation risk of elastic net and smoothed elastic-net are 
\begin{align}\label{eq:LO:sLO:l1}
{\rm LO}(\lambda) &= \frac{1}{n} \sum_{i=1}^n \phi(y_i, \bx_i^{\top} \hat{\bbeta}_{/ i}). \nonumber \\
{\rm LO}^{\alpha}(\lambda) &= \frac{1}{n} \sum_{i=1}^n \phi(y_i, \bx_i^{\top} \hat{\bbeta}^{\alpha}_{/ i}).
\end{align}
Hence the difference of the two is
\begin{align}\label{eq:ALOvsLO:smooth1}
|{\rm LO}(\lambda) -{\rm LO}^{\alpha}(\lambda)| &\leq \max_{i,z_i} |\dot{\phi} (y_i,z_i) | |\bx_i^{\top}(\hat{\bbeta}_{/ i}-\hat{\bbeta}^{\alpha}_{/ i} )|  \nonumber \\
&\leq\max_{i,z_i} |\dot{\phi} (y_i,z_i) | \|\bx_i\|\|\hat{\bbeta}_{/ i}-\hat{\bbeta}^{\alpha}_{/ i} )\|  \nonumber \\
&\leq \max_{i,z_i} |\dot{\phi} (y_i,z_i) | \|\bx_i\| \sqrt{\frac{4 \log 2 p}{\alpha \eta}}. 
\end{align}
According to \eqref{eq:ALO-LO1} we have
\begin{align}\label{eq:ALO-LO2}
&\abs*{{\rm ALO}^\alpha-{\rm LO}^\alpha} \nonumber \\
&\le \max_{1\le i\le n}\abs*{\dot{\phi}(y_i,\bx_i^\top\tilde{\bbeta}_i^\alpha)}
\cdot \dfrac{1}{n}
\cdot \sum_{i=1}^n\abs*{\bx_i^\top\Delta_i^\alpha-\bx_i^\top\hat{\Delta}_{/i}^\alpha},
\end{align}
Combining \eqref{eq:LO:sLO:l1}, \eqref{eq:ALOvsLO:smooth1}, and \eqref{eq:ALO-LO2}, we have 
\begin{align*}
&~\abs*{{\rm ALO}^\alpha-{\rm LO}}\\
\leq&~  \max_{i,z_i} |\dot{\phi} (y_i,z_i) | \|\bx_i\| \sqrt{\frac{4 \log 2 p}{\alpha \eta}} 
+
\max_{1\le i\le n}\abs*{\dot{\phi}(y_i,\bx_i^\top\tilde{\bbeta}_i^\alpha)}
\dfrac{1}{n}
\sum_{i=1}^n\abs*{\bx_i^\top\Delta_i^\alpha-\bx_i^\top\hat{\Delta}_{/i}^\alpha}. \ \ \ \ 
\end{align*}

Furthermore, according to \eqref{eq:xdeldiff1}, Theorem \ref{th:rduc2supp}, and Theorem \ref{thm:FINAL_STEP} we have that with probability at least $1-(n+1){\rm e}^{-\frac{p}{2}}-(n+2)p^{-d_n}-2q_n-2\check{q}_n-2\tilde{q}_n-2\bar{q}_n$: 
\begin{align}\label{eq:SLO:LO:L2}
&~\abs*{\bx_i^\top\Delta_{/i}^\alpha-\bx_i^\top\hat{\Delta}_{/i}^\alpha}\\
\le&  \,|\dot{\ell}_i(\hat{\bbeta}^\alpha)|\times
\nonumber\\
&~ \bx_i^\top
\bigg[ \left(\int_0^1\bJ_{/i}(t\hat{\bbeta}^\alpha
+(1-t)\hat{\bbeta}_{/i}^\alpha)dt\right)^{-1} -
\left(\bJ_{/i}(\hat{\bbeta}_{/i}^\alpha)\right)^{-1}
\bigg]\bx_i
\nonumber\\
&\,+ |\dot{\ell}_i(\hat{\bbeta}^\alpha)|
\bx_i^\top
\left[ \left(\bJ_{/i}(\hat{\bbeta}_{/i}^\alpha)\right)^{-1}
-\left(\bJ_{/i}(\hat{\bbeta}_{/i}^\alpha-\Delta_{/i}^\alpha)\right)^{-1}
\right]\bx_i, \nonumber\\
\leq& |\dot{\ell}_i(\hat{\bbeta}^\alpha)|\times \nonumber\\
&\Big| \bx_{i,\calB_{0,i}^c}^\top 
			(\lambda{\rm diag}(\ddot{\underline{r}}_{\calB_{0,i}^c}^{\alpha/i})
			+\bX^\top_{/i,\calB_{0,i}^c}{\rm diag}
			(\ddot{\underline{\ell}}^{\alpha/i})\bX_{/i, \calB_{0,i}^c})^{-1}\bx_{i,\calB_{0,i}^c} \nonumber \\
   &- \bx_{i,\calB_{0,i}^c}^\top 
			(\lambda{\rm diag}(\ddot{\utilde{r}}_{\calB_{0,i}^c}^{\alpha/i})
			+\bX^\top_{\calB_{0,i}^c}{\rm diag}
			(\ddot{\utilde{\ell}}^{\alpha/i})\bX_{\calB_{0,i}^c})^{-1}\bx_{i,\calB_{0,i}^c}  \Big| \nonumber \\
   &+ \dfrac{32 |\dot{\ell}_i(\hat{\bbeta}^\alpha)| \|\bx_{i}\|^2}{\lambda\alpha(1-\eta)\kappa_0} \Big(\frac{\polylog(n)\|\bX^{\top}\bX\|}{2\lambda\eta}+1\Big)^2 \nonumber\\
\leq&~ |\dot{\ell}_i(\hat{\bbeta}^\alpha)|\times\nonumber \\
& \Bigg( \dfrac{\polylog(n)}{\lambda^3\eta^3(1\wedge\lambda\eta)^3}
    \sqrt{\dfrac{d_n\log^2p}{n\lambda\eta}}
    +\dfrac{Cd_n}{n\lambda^2\eta^2}
    +
    \sqrt{\dfrac{C\log p}{n\lambda\eta}}\Bigg) \nonumber \\
    &+ \dfrac{32 |\dot{\ell}_i(\hat{\bbeta}^\alpha)| \|\bx_{i}\|^2}{\lambda\alpha(1-\eta)\kappa_0} \Big(\frac{\polylog(n)\|\bX^{\top}\bX\|}{2\lambda\eta}+1\Big)^2
    \numberthis.
\end{align}
As is clear from this equation, as $\alpha \rightarrow \infty$, and for large values of $n,p$, if $d_n$ grows slowly enough in $n$ (or equivalently in $p$) the difference $\abs*{{\rm ALO}^\alpha-{\rm LO}}$ will be negligible. 
The last step of the proof, is to show that the difference between $\abs*{{\rm ALO}^\alpha-{\rm ALO}}$ is also negligible. Note that our approximate ALO formula for elastic net can be written as
\begin{align}
{\rm ALO} = \frac{1}{n}\sum_{i=1}^n \phi(y_i, \bx_{i, \mathcal{S}}^{\top}\hat{\bbeta}_{\mathcal{S}} + \bx_{i, \mathcal{S}}^{\top} \hat{\Delta}_{/ i}),
\end{align}
where
\[
\hat{\Delta}_{/ i} = \dot{\ell} (\hat{\bbeta})  (2 \lambda \eta \mathbb{I} + \bX^{\top}_{\mathcal{S}} {\rm diag} (\ddot{\ell} (\hat{\bbeta})) \bX_{\mathcal{S}})^{-1} \bx_{i,\mathcal{S}}. 
\]

Hence, we have
\begin{align}\label{eq:alo-alpha-diff}
&~|{\rm ALO}- {\rm ALO}^{\alpha}| \nonumber \\ 
=&~ \frac{1}{n}\sum_{i=1}^n | \phi(y_i, \bx_{i, \mathcal{S}}^{\top}(\hat{\bbeta}_{\mathcal{S}} +  \hat{\Delta}_{/i})) - \phi(y_i, (\bx_{i}^{\top}\hat{\bbeta}^{\alpha} +\hat{\Delta}^{\alpha}_{/i})) | \nonumber \\
\leq&~ \max_{i,z_i} |\dot{\phi} (y_i,z_i) | |(\bx_{i}^{\top}\hat{\bbeta}^{\alpha}- \bx_{i, \mathcal{S}}^{\top}\hat{\bbeta}_{\mathcal{S}}| \nonumber \\
&+ \max_{i,z_i} |\dot{\phi} (y_i,z_i) | | \bx_{i, \mathcal{S}}^{\top}  \hat{\Delta}_{/ i} -\bx_{i}^{\top}\hat{\Delta}^{\alpha}_{/ i}| \nonumber \\
\leq&~ \max_{i,z_i} |\dot{\phi} (y_i,z_i) | \|\bx_i\| \sqrt{\frac{4 \log 2p}{\alpha \eta}} \nonumber \\
&+ \max_{i,z_i} |\dot{\phi} (y_i,z_i) | | \bx_{i, \mathcal{S}}^{\top}  \hat{\Delta}_{/ i} -\bx_{i}^{\top}\hat{\Delta}^{\alpha}_{/ i}|,
\end{align}
where to obtain the last inequality we have used the first part of Lemma \ref{lem:beta-size}.  

Similar to the proof of Theorem \ref{thm:FINAL_STEP} we can prove that 
\begin{align}
|&\bx_{i, \mathcal{S}}^{\top}  \hat{\Delta}_{/i} - \dot{\ell} (\hat{\bbeta})\bx_{i, \mathcal{B}_{1,i,+}}^{\top} 
(2 \lambda \eta \mathbb{I} + \bX^{\top}_{\mathcal{B}_{1,i,+}} {\rm diag} (\ddot{\ell} (\hat{\bbeta})) \bX_{\mathcal{B}_{1,i,+}})^{-1} \bx_{i,\mathcal{B}_{1, i,+}}| \nonumber \\
\leq&~ \dfrac{\polylog(n)}{\lambda^3\eta^3(1\wedge\lambda\eta)^3}
    \sqrt{\dfrac{d_n\log^2p}{n\lambda\eta}}
    +\dfrac{Cd_n}{n\lambda^2\eta^2}
    +
    \sqrt{\dfrac{C\log p}{n\lambda\eta}}
\end{align}
with probability at least $1-(n+1){\rm e}^{-\frac{p}{2}}-(n+2)p^{-d_n}-2q_n-2\check{q}_n-2\tilde{q}_n-2\bar{q}_n$, for sufficiently large $p$. Similarly,
\begin{align}
|&\bx_{i}^{\top}  \hat{\Delta}^{\alpha}_{/i} - \dot{\ell} (\hat{\bbeta}^{\alpha})\bx_{i, \mathcal{B}_{1,i,+}}^{\top} 
(\lambda {\rm diag} (\ddot{r}^{\alpha}(\hat{\bbeta}^{\alpha})) + \bX^{\top}_{\mathcal{B}_{1,i,+}} {\rm diag} (\ddot{\ell} (\hat{\bbeta}^{\alpha})) \bX_{\mathcal{B}_{1,i,+}})^{-1} \bx_{i,\mathcal{B}_{1, i,+}}| \nonumber \\
\leq&~ \dfrac{\polylog(n)}{\lambda^3\eta^3(1\wedge\lambda\eta)^3}
    \sqrt{\dfrac{d_n\log^2p}{n\lambda\eta}}
    +\dfrac{Cd_n}{n\lambda^2\eta^2}
    +
    \sqrt{\dfrac{C\log p}{n\lambda\eta}}
\end{align}
with probability at least $1-(n+1){\rm e}^{-\frac{p}{2}}-(n+2)p^{-d_n}-2q_n-2\check{q}_n-2\tilde{q}_n-2\bar{q}_n$, for sufficiently large $p$. In the rest of the proof, for the notational simplicity, we use the notation $\mathcal{B}_{+}$ instead of $\mathcal{B}_{1,i,+}$.

\begin{align}\label{eq:xdelta-diff}
&| \bx_{i, \mathcal{S}}^{\top}  \hat{\Delta}_{/i} -\bx_{i}^{\top}\hat{\Delta}^{\alpha}_{/i}| \nonumber \\
\leq&~ |\dot{\ell} (\hat{\bbeta})\bx_{i, \mathcal{B}_{+}}^{\top} (2 \lambda \eta \mathbb{I} + \bX^{\top}_{\mathcal{B}_{+}} {\rm diag} (\ddot{\ell} (\hat{\bbeta})) \bX_{\mathcal{B}_{+}})^{-1} \bx_{i,\mathcal{B}_{+}} \nonumber\\
&~-\dot{\ell} (\hat{\bbeta}^{\alpha})\bx_{i, \mathcal{B}_{+}}^{\top}   (\lambda {\rm diag} (\ddot{r}^{\alpha}(\hat{\bbeta}^{\alpha})) + \bX^{\top}_{\mathcal{B}_{+}} {\rm diag} (\ddot{\ell} (\hat{\bbeta}^{\alpha})) \bX_{\mathcal{B}_{+}})^{-1} \bx_{i,\mathcal{B}_{+}}| \nonumber \\
&+ \dfrac{2 \polylog(n)}{\lambda^3\eta^3(1\wedge\lambda\eta)^3}
    \sqrt{\dfrac{d_n\log^2p}{n\lambda\eta}}
    +\dfrac{Cd_n}{n\lambda^2\eta^2}
    +
    \sqrt{\dfrac{C\log p}{n\lambda\eta}} \nonumber \\
\leq&~ 
    | \dot{\ell} (\hat{\bbeta})|\times\nonumber\\
    &\times\big\vert
   \bx_{i, \mathcal{B}_{+}}^{\top} (2 \lambda \eta \mathbb{I} + \bX^{\top}_{\mathcal{B}_{+}} {\rm diag} (\ddot{\ell} (\hat{\bbeta})) \bX_{\mathcal{B}_{+}})^{-1} \bx_{i,\mathcal{B}_{+}}  \nonumber\\
   &~-\bx_{i, \mathcal{B}_{+}}^{\top}   (\lambda {\rm diag} (\ddot{r}^{\alpha}(\hat{\bbeta}^{\alpha})) + \bX^{\top}_{\mathcal{B}_{+}} {\rm diag} (\ddot{\ell} (\hat{\bbeta}^{\alpha})) \bX_{\mathcal{B}_{+}})^{-1} \bx_{i,\mathcal{B}_{+}}
\big\vert \nonumber \\
&+ | \dot{\ell} (\hat{\bbeta})
    -
    \dot{\ell} (\hat{\bbeta}^{\alpha})|\times\nonumber\\
&\times\bx_{i, \mathcal{B}_{+}}^{\top}   (\lambda {\rm diag} (\ddot{r}^{\alpha}(\hat{\bbeta}^{\alpha})) + \bX^{\top}_{\mathcal{B}_{+}} {\rm diag} (\ddot{\ell} (\hat{\bbeta}^{\alpha})) \bX_{\mathcal{B}_{+}})^{-1} \bx_{i,\mathcal{B}_{+}}
\nonumber \\
&+ \dfrac{2 \polylog(n)}{\lambda^3\eta^3(1\wedge\lambda\eta)^3}
    \sqrt{\dfrac{d_n\log^2p}{n\lambda\eta}}
    +\dfrac{Cd_n}{n\lambda^2\eta^2}
    +
    \sqrt{\dfrac{C\log p}{n\lambda\eta}} .
\end{align}
Since the minimum eigenvalue of $(\lambda {\rm diag} (\ddot{r}^{\alpha}(\hat{\bbeta}^{\alpha})) + \bX^{\top}_{\mathcal{B}_{+}} {\rm diag} (\ddot{\ell} (\hat{\bbeta}^{\alpha})) \bX_{\mathcal{B}_{+}})$ is larger than $2 \lambda \eta$, we can conclude that 
\begin{align}
&| \dot{\ell} (\hat{\bbeta})
    -
    \dot{\ell} (\hat{\bbeta}^{\alpha})|\times\nonumber\\
&\times\bx_{i, \mathcal{B}_{+}}^{\top}   (\lambda {\rm diag} (\ddot{r}^{\alpha}(\hat{\bbeta}^{\alpha})) + \bX^{\top}_{\mathcal{B}_{+}} {\rm diag} (\ddot{\ell} (\hat{\bbeta}^{\alpha})) \bX_{\mathcal{B}_{+}})^{-1} \bx_{i,\mathcal{B}_{+}}\nonumber\\
&\leq \frac{\|\bx_i\|_2^2 }{2 \lambda \eta} |\dot{\ell}(\hat{\bbeta}^{\alpha}) - \dot{\ell} (\hat{\bbeta})| \overset{(a)}{=}  \frac{\|\bx_i\|_2^2 }{2 \lambda \eta} |\ddot{\ell}(\theta) (\bx_i^\top(\hat{\bbeta}-\hat{\bbeta}^{\alpha} ))| \nonumber \\
&\leq \frac{\|\bx_i\|_2^3 }{2 \lambda \eta} |\ddot{\ell}(\btheta)| \| \hat{\bbeta}-\hat{\bbeta}^{\alpha} \|\nonumber\\
&\overset{(b)}{\leq}  
\frac{\|\bx_i\|_2^3 }{2 \lambda \eta} |\ddot{\ell}(\btheta)|
\sqrt{\dfrac{4\log(2)p}{\alpha\eta}}
\nonumber \\
&\overset{(c)}{\leq} \frac{\polylog(n) }{\lambda \eta} 
\sqrt{\dfrac{p}{\alpha\eta}},
\end{align}
with probability larger than $1-n{\rm e}^{-p}-q_n - \check{q}_n$. 
To obtain Equality (a) we have used the mean value theorem and $\btheta =  t \hat{\bbeta}+(1-t) \hat{\bbeta}^{\alpha}$ for some $t \in [0,1]$. To obtain inequality (b) we have used Part 1 of Lemma \ref{lem:beta-size}. Inequality (c) is based on Assumption A.4 along with Lemma~\ref{lem:xi-row-conc}.

Also, using Lemma \ref{lem:inverseinflation} we have that 
\begin{align}
&|\bx_{i, \mathcal{B}_{+}}^{\top} (2 \lambda \eta \mathbb{I} + \bX^{\top}_{\mathcal{B}_{+}} {\rm diag} (\ddot{\ell} (\hat{\bbeta})) \bX_{\mathcal{B}_{+}})^{-1} \bx_{i,\mathcal{B}_{+}}  \nonumber \\
&-\bx_{i, \mathcal{B}_{+}}^{\top}   (\lambda {\rm diag} (\ddot{r}^{\alpha}(\hat{\bbeta}^{\alpha})) + \bX^{\top}_{\mathcal{B}_{+}} {\rm diag} (\ddot{\ell} (\hat{\bbeta}^{\alpha})) \bX_{\mathcal{B}_{+}})^{-1} \bx_{i,\mathcal{B}_{+}}| \nonumber \\
&\leq \frac{\|\bx_i\|^2 \lambda_{\max}(\bGamma)}{(2 \lambda \eta)^2} + \frac{\|\bx_i\|^2 \lambda^2_{\max}(\bGamma)}{(2 \eta \lambda)^2 (2 \lambda \eta - \lambda_{\max}(\bGamma))}, 
\end{align}
where 
\begin{align}
\bGamma :=
&~\lambda {\rm diag} (\ddot{r}^{\alpha}(\hat{\bbeta}^{\alpha})) + \bX^{\top}_{\mathcal{B}_{+}} {\rm diag} (\ddot{\ell} (\hat{\bbeta}^{\alpha})) \bX_{\mathcal{B}_{+}} \nonumber \\
&- 2 \lambda \eta \mathbb{I} + \bX^{\top}_{\mathcal{B}_{+}} {\rm diag} (\ddot{\ell} (\hat{\bbeta})) \bX_{\mathcal{B}_{+}}).  
\end{align}
Therefore, by using Weyl's theorem 
\begin{align}
\lambda_{\max}(\bGamma) \leq&~ \lambda_{\max} (\lambda {\rm diag} (\ddot{r}_{\mathcal{B}_{ +}}^{\alpha}(\hat{\bbeta}^{\alpha})) - 2 \lambda \eta \mathbb{I}) \nonumber \\
&+\lambda_{\max } ( {\rm diag} (\ddot{\ell} (\hat{\bbeta}))- {\rm diag} (\ddot{\ell} (\hat{\bbeta}^{\alpha})) \|\bX^{\top} \bX\| \nonumber \\
{\leq}&~ 2 \lambda \alpha {\rm e}^{-\frac{1}{2} \alpha \kappa_1} +  \|\hat{\bbeta}^{\alpha}- \hat{\bbeta}\| \|\bX^\top \bX\| \polylog(n) \nonumber \\
{\leq}&~ 2 \lambda \alpha {\rm e}^{-\frac{1}{2} \alpha \kappa_1} + \sqrt{\frac{4 \log 2 p}{\alpha \eta}}\|\bX^\top \bX\| \polylog(n)\nonumber\\
\leq&~
\dfrac{1}{\sqrt{p}}
, 
\end{align}
with probability larger than $1- {\rm e}^{-p}-\tilde{q}_n$. 
To obtain the two penultimate two inqualities we have used Lemma \ref{lem:rdderiv} and \eqref{eq:assn_loss}. The last inequality uses $\alpha=\omega\left(
\frac{p\polylog(n)}{\eta}\wedge 
\frac{\polylog(n)}{\kappa_1} 
\right)$, where we use the fact Lemma~\ref{lem:maxsingularvalue} to conclude that $\|\bX^{\top}\bX\|\le (\sqrt{\gamma_0}+3)^2C_X$ with probability at least $1-{\rm e}^{-p}$.

Plugging in these bounds into \eqref{eq:xdelta-diff} we obtain
\begin{align*}
    &~
    | \bx_{i, \mathcal{S}}^{\top}  \hat{\Delta}_{/i} -\bx_{i}^{\top}\hat{\Delta}^{\alpha}_{/i}| \\
    \le&~
    |\dot{\ell}(\hat{\beta})|
    \left(
    \frac{\|\bx_i\|^2}{p(2 \lambda \eta)^2}
    + \frac{\|\bx_i\|^2}{(2 p\eta \lambda)^2 (\lambda \eta)}
    \right)\\
    \le&~
    \dfrac{\polylog(n)}{\sqrt{p}(\lambda\eta)(1\vee \lambda\eta)}
    \left(
    1+\dfrac{1}{p\lambda\eta}
    \right)
    +
    \dfrac{\polylog(n)}{\lambda^3\eta^3(1\wedge\lambda\eta)^3}
    \sqrt{\dfrac{d_n\log^2p}{n\lambda\eta}}
    +\dfrac{Cd_n}{n\lambda^2\eta^2}
    +
    \sqrt{\dfrac{C\log p}{n\lambda\eta}}
\end{align*}
where we use $\alpha=\omega\left(
\frac{p\polylog(n)}{\eta}\wedge 
\frac{\polylog(n)}{\kappa_1} 
\right)$.

Returning to \eqref{eq:alo-alpha-diff} we thus write
\begin{align*}
    &~|{\rm ALO}- {\rm ALO}^{\alpha}|\\
    \le&~
    \dfrac{\polylog(n)}{\lambda^3\eta^3(1\wedge\lambda\eta)^3}
    \sqrt{\dfrac{d_n}{n\lambda\eta}}
    +\dfrac{d_n\polylog(n)}{n\lambda^2\eta^2}
    +
    \sqrt{\dfrac{\polylog(n)}{n\lambda\eta}},
\end{align*}
by the assumption $\dot{\phi}(y_i,z_i)=O_p(\polylog(n))$. To conclude the proof of the theorem, note that
\begin{align*}
    &~|{\rm ALO}- {\rm LO}|\\
    \le &~|{\rm ALO}- {\rm ALO}^{\alpha}|
    +|{\rm ALO}^{\alpha}- {\rm LO}|\\
    \le &    \dfrac{\polylog(n)}{\lambda^3\eta^3(1\wedge\lambda\eta)^3}
    \sqrt{\dfrac{d_n}{n\lambda\eta}}
    +\dfrac{d_n\polylog(n)}{n\lambda^2\eta^2}
    +
    \sqrt{\dfrac{\polylog(n)}{n\lambda\eta}}
\end{align*}
provided $\alpha=\omega\left(\frac{n\polylog(n)}{\kappa_0^2\kappa_1^2\lambda(1-\eta))\eta}\wedge 
\frac{\polylog(n)}{\kappa_1} \right)$, with probability at least $1-(n+1){\rm e}^{-p}-(n+2)p^{-d_n}-2q_n-2\check{q}_n-2\tilde{q}_n-2\bar{q}_n$. This finishes the proof of Theorem~\ref{th:alo-main}.
\end{proof}

\section{The example of linear regression} \label{sec:linearregression}
Theorem \ref{th:alo-main} shows that if $d_n$ grows slowly enough (e.g. $d_n= p^{\alpha}$ for $\alpha<1$), then the difference betweent ALO and LO will go to zero in proability. To see what the growth rate of $d_n$ in terms of $p$ is, in this section we focus on the concrete example of linear regression and obtain an upper bound for $|\calB_{0,i}^c\setminus
    \calB_{1,i,+}|$. Even though the result of this section is given for the linear models, it is expected that a similar conclusion holds for more general models and under more general assumptions. However, given the length of the current paper, the complete investigation of the size of $|\calB_{0,i}^c\setminus
    \calB_{1,i,+}|$ under the generalized linear model is left for a future research. Let us start with our modelling assumptions:
    
    \noindent\textbf{Assumption group B}
\begin{enumerate}
    \item[B1] $\by = \bX \bbeta^* +\bw$, where $\bw \sim N(0, \sigma_w^2 I)$ is the noise or error in the observations. 
    \item[B2] The loss function $l(y|\bx^\top\bbeta) = \frac12 (y-\bx^\top\bbeta)^2$
    \item[B3] The true coefficients $\bbeta^*$ satisfy 
    $ \frac1p \norm{\bbeta^*}_2^2 \leq \xi$
    for some constant $\xi>0$.
    \item[B4] $\bX$ has iid entries $X_{ij}\sim N(0,\frac{1}{n})$.
    \item[B5] $\lambda^2(1-\eta)^2=\omega(p^{-\frac{7}{12}})$.
\end{enumerate}

Note that Assumptions B1-B3 are standard in the literature of linear regression. Assumption B4 is also frequently encountered in the high-dimensional asymptotic analysis of estimators \cite{miolane2021distribution, bradic2015robustness, DoMaMo09, bayati2012lasso, weng2016overcoming, dobriban2018high, WaWeMa20, AnMaBa12, pmlr-v40-Thrampoulidis15, rangan2011generalized, li2021minimum}. However, it is expected that this assumption can be relaxed as well given the more recent results in the literature \cite{celentano2020lasso}. Finally, Assumption B5 is a technical assumption on the rate of $\lambda$. We note here that it shows that for large values of $p$, one can choose $\lambda$ to be quite small. The following theorem uses Assumptions B1-B5 to find an upper bound for $ (\calB_{1,i,+}\cup\calB_{0,i})^c|$.

\begin{theorem}\label{thm:size-diff-S} Under Assumptions A1-A5 and B1-B5, in \eqref{eq:def-S-sets} set $\kappa_0=(\frac{8\log p}{c p})^{1/6}$ and $\kappa_1=p^{-1/12}(\log p)^{1/4}$ where $c$ is the constant $c_2$ in Lemma \ref{lem:sparsity-bd-MM18}. Then there exist constants $C, C'>0$ such that
	$$
	\max_{1\le i \le n}
	|(\calB_{1,i,+}\cup\calB_{0,i})^c|
	\le Cp^{11/12}(\log p)^{1/4}
	$$
with probability at least $1-C'p^{-6}-q_n-{\rm e}^{-cp}$.
\end{theorem}
\begin{proof}[Sketch of the proof]
One of the main components of this proof uses concentration results on the empirical distribution of $\hbbeta$, the subgradient of $\ell_1$-regularizer (see below for precise definition), and the sparsity of $\hbbeta$, where $\hbbeta$ is the minimizer of
\[
\frac{1}{2} \|\by -\bX \bbeta\|_2^2 + \lambda \eta \|\bbeta\|_2^2 + \lambda (1- \eta) \|\bbeta\|_1. 
\]
In order to state the results, we first introduce the following terms:
\begin{itemize}
    \item Let $\hmu$ denotes the empirical distribution of $\hbbeta$.
    \item Let $\Theta$ be a random variable with its value uniformly distributed among the elements of $\bbeta^*$, and let $Z\sim N(0,1)$ and independent of $\Theta$.
    \item Let $\soft(x, r)$ denote the soft thresholding function
    \[
        \soft(x,r) = (|x|-r)_+ \sign(x). 
    \]
    \item For a couple $(\tau,b)$, define 
    \[
        \hat{w}^f(\tau, b)= \frac{b}{b + 2\lambda\eta\tau}\soft\Big(\tau Z + \Theta, \frac{\lambda(1-\eta)\tau}{b}\Big) - \Theta.
    \]
    \item Let the couple $(\tau_*,b_*)$ be the unique solution of the following equations:\footnote{The uniqueness of the solution is proved in Lemma~\ref{lem: fix point equations}.}
    \begin{align}
        \tau^2 = &~\sigma^2 + \frac{1}{\gamma_0} \EE [\hat{w}^f(\beta,\tau)]^2
        \label{eq:tau_*},\\
        \beta = &~\tau - \frac{1}{\gamma_0} \EE Z\cdot \hat{w}^f(\beta,\tau).
        \label{eq:b_*}
    \end{align}
    \item Let $\mu^*$ denote the law of the random variable $\hat{w}^f(\tau_*, b_*)+\Theta$.
    \item Let $s_*$ be defined as
    \[
        s_* = \PP\left(|\Theta+\tau_*Z|\ge \frac{\lambda\tau_*}{b_*}\right).
    \]
\end{itemize}

\begin{lemma}[restated and proved later as Theorem~\ref{thm:W2 of elasticnet}]
\label{lem:wasserstein-bd-MM18}
Under Assumptions A1-A5 and B1-B5, there exist constants $C,c>0$ such that for all $\eps\in(0,0.5]$:
$$\PP\left( W_2(\hmu, \mu^*)^2>\eps\right) \leq C\eps^{-2}e^{-cp\eps^3 (\log\eps)^{-2}}, $$
where $W_2$ denotes the Wasserstein 2-distance.
\end{lemma}

We can now start the proof of the theorem. By the definition of $\calB_{0,i}$  and $\calB_{1,i}$ we have, for each fixed $i$ that:
	 \begin{align*}\label{eq:small-sets}
	 	&~(\calB_{1,i} \cup \calB_{0,i})^c \\
	 	=&~\bigg\{k: \left( \abs{\hbeta_k}\leq \kappa_1 \text{ or } \abs{\hbeta_{/i,k}}\leq \kappa_1 \right) \text{ and } \\
            &~\left( \abs{g(\hbeta_k)} >1-\kappa_0 \text{ or }\abs{g(\hbeta_{/i,k})} >1-\kappa_0  \right) \bigg\}\\
	 	\subset&~
	 	\left\{k:0<|\hbeta_k|\le\kappa_1\right\}\\
	 	&~\cup
	 	\left\{k:0<|\hbeta_{/i,k}|\le\kappa_1\right\}\\
	 	&~\cup 
	 	\left\{k:1-\kappa_0\le\abs{g(\hbeta_k)} < 1\right\}\\
	 	&~\cup 
	 	\left\{k:1-\kappa_0\le\abs{g(\hbeta_{/i,k})} < 1\right\}\\
	 	&~\cup 
	 	\left\{k:|\hat{\beta}_k|>\kappa_1;\abs{g(\hbeta_{/i,k})}\le 1-\kappa_0\right\}\\
	 	&~\cup
	 	\left\{k:|\hat{\beta}_{/i,k}|>\kappa_1;\abs{g(\hbeta_{k})}\le 1-\kappa_0\right\}	 	\\
	 :=&~\calK_1\cup \calK_1'\cup \calK_2\cup \calK_2'\cup \calK_3\cup \calK_3'.\numberthis
	 \end{align*}
 	We can now bound the sizes of each of the above sets. Since the full model and the leave-one-out models are the same in nature, we only bound the sets ($\cK_1, \cK_2, \cK_3$) related to the full model, since the same proofs apply also to the leave-one-out models ($\cK_1', \cK_2', \cK_3'$).
	 \begin{enumerate}
	\item Bounding $|\calK_1|$: The following lemma helps us bound the size of this set:

 \begin{lemma}\label{lem:small-nzro-coef}
	Suppose $\kappa_1 = o(p^{-\frac{1}{12}}(\log p)^{\frac14})$. Then, then there exists constants $C, C'$ such that for all $0\leq k\leq n$,
        \begin{align*}
	\abs*{\{ k: 0<\abs{\hbeta_{/i,k}}\leq \kappa_1  \}}
	\le&~
	Cp^{\frac{11}{12}}(\log p)^{\frac14}\\
	\end{align*}
        with probability at least $1-C'p^{-7}$. It then follows from a union bound over $i$   that
        \begin{align*}
	\max_{0\leq i \leq n}\abs*{\{ k: 0<\abs{\hbeta_{/i,k}}\leq \kappa_1  \}} 
	\le&~
	Cp^{\frac{11}{12}}(\log p)^{\frac14}
	\end{align*}
	with probability at least $1-C'p^{-6}$.
\end{lemma}
Note that the size of the set $\abs*{\{ k: 0<\abs{\hbeta_{k}}\leq \kappa_1  \}}$ can be calculated from the empirical distribution $\hat{\mu}$. Also, Lemma \ref{lem:wasserstein-bd-MM18} connects $\hat{\mu}$ with $\mu^*$. Hence, it is expected that we should be able to find a concentration result for $\abs*{\{ k: 0<\abs{\hbeta_{k}}\leq \kappa_1  \}}$. However, there are several technical issues that need to be addressed in order to prove Lemma  \ref{lem:small-nzro-coef}. Hence, the complete proof of this lemma will appear in Section \ref{ssec:proof:lemma5}. 
 
 Using Lemma~\ref{lem:small-nzro-coef}, it is straightforward to confirm that	 	
        $$
        |\calK_1|\le Cp^{11/12}(\log p)^{1/4}
        $$
        with probability at least $1-C'p^{-7}$ for some $C,C'>0$.
    
        \item Bounding $\cK_2$: To find an upper bound for the size of the set $\cK_2$,  consider the following two sets:
        $$
            \calT_1 = \{k: \hat{\beta}_k \neq 0\}
        $$
        and
        $$\calT_2(\kappa_0) =  \{k : |g     
            (\hbbeta)_k| \in [1-\kappa_0, 1]\}. 
        $$
        First note that  
        \[
            \calT_1 \subset \calT_2
            \text{ and }
            \calK_2 = \calT_2 / \calT_1. 
        \]
    Our first goal is to show that $\frac{1}{p}|\calT_1|$ and $\frac{1}{p}|\calT_2|$ are close to each other. The following two lemmas enable us to compre the two sets. 

    \begin{lemma}[restated and proved later as Theorem~\ref{thm:conc_sparsity_subgrad}]
    \label{lem:num-large-subgrad_MM18}
    Under Assumptions A1-A5 and B1-B5, there exist constants $C,C',c>0$ such that for all $\eps \in (0,1]$,
    $$\PP\left( \frac1p \sum_{k=1}^p \mathbbm{1}_{\{ \abs{g(\hbbeta)_k}\geq 1-\eps \}}\geq s_* + C\eps\right)\leq C'\eps^{-3}e^{-cn\eps^6}.$$
\end{lemma}

\begin{lemma}[restated and proved later as Theorem~\ref{thm:betahat-l0}]
    \label{lem:sparsity-bd-MM18}
    Under the same assumptions as the above theorem, there exist constants $C,c>0$ such that for all $\eps \in (0,1]$ we have
    $$
    \PP\left(\abs*{\frac1p\norm{\hbbeta}_0 - s_*}\geq\eps\right) \leq C\eps^{-6}e^{-cn\eps^6}.
    $$
\end{lemma}

We should mention that the above two lemmas were originally proved in \cite{miolane2021distribution} for the LASSO problem. Theorems A.4 and A.5 are the extensions of the results of \cite{miolane2021distribution} and their proof strategies are the same.

        From Lemma \ref{lem:sparsity-bd-MM18} and Lemma \ref{lem:num-large-subgrad_MM18}, it follows that, for some constant $C,c_1, c_2$,
        $$
            \abs*{\dfrac{1}{p}|\calT_1|-s_*}
            \le\kappa_0 
            \text{ and }
            \dfrac{1}{p}|\calT_2(\kappa_0)|-s_*
            \le C\kappa_0 
        $$
        with probability at least $1-\dfrac{2c_1}{\kappa_0^6}\exp(-c_2p\kappa_0^6)$. Here $s_*\in [0,1]$ is the constant in Lemma \ref{lem:num-large-subgrad_MM18}. Setting $\kappa_0= (\frac{8\log p}{c_2p})^{1/6}$,  using the above concentration, we obtain
	 	 $$
	 	 |\calK_2|
	 	 \le |\calT_1\setminus\calT_2|
	 	 \le Cp^{5/6}(\log p)^{1/6}
	 	 $$
	 	 with probability at least $1-C' p^{-7}$.
	 	 
	 	 \item To obtain an upper bound for $|\cK_3|$, let $g(\hat{\bbeta}),g(\hat{\bbeta}_{/i})$ denote the sub-gradients of the LASSO penalty defined in (\ref{eq:subgrad}). Note that $|g(\hat{\bbeta})_k|=1$ and $|g(\hat{\bbeta}_{/i,k})|\le 1-\kappa_0$ for $k\in\calK_3$, and hence
         \[
             |g(\hat{\bbeta})_k-g(\hat{\bbeta}_{/i})_{k}|\ge \kappa_0
             \quad
             \text{for }
             k\in\calK_3.
         \]
	Thus 
        \begin{align*}
            \sqrt{\kappa_0^2|\calK_3|}
            &\leq
            \sqrt{\sum_{k\in\calK_3}|
            g(\hbbeta)_k-g(\hbbeta_{/i})_{k}
            |^2}\\
            &\le
            \norm*{g(\hbbeta)-g(\hbbeta_{/i})}\\
            &\le \frac{2|\dot{\ell}_i(\hbbeta)| \|\bx_i \|}{\lambda(1-\eta)}\\
            &\le \frac{C\polylog(p)}{\lambda(1-\eta)}
        \end{align*}
        with probability at least $1-q_n - {\rm e}^{-cp}$. The penultimate line follows from Part 4 of Lemma \ref{lem:beta-size}, and the last line uses Assumption A4 and Lemma \ref{lem:xi-row-conc}. Therefore
        $$
            |\calK_3|\le \frac{C\polylog(p)}{\kappa_0^2\lambda^2(1-\eta)^2}= \frac{C\polylog(p)}{\lambda^2(1-\eta)^2}p^{\frac13}\le Cp^{\frac{11}{12}}
        $$
	with probability at least $1-q_n - {\rm e}^{-cp}$, provided $\lambda^2(1-\eta)^2=\omega(p^{-\frac{7}{12}})$
	 	 
\end{enumerate}
 	Note that the same arguments hold for all $1\le i\le n$. By cases 1-3 above, along with \eqref{eq:small-sets}, we have proved that
 	
 	\begin{equation}\label{eq:b19c-card}
 		\max_{1\le i \le n}
 	|(\calB_{1,i}\cup\calB_{0,i})^c|
 	\le Cp^{11/12}(\log p)^{1/4}
 	\end{equation}
 	with probability at least $1-Cp^{-6}-q_n-{\rm e}^{-cp}$, for sufficiently large $p$.

    The last step of the proof is to bound $|\calB_{1,i} / \calB_{1,i,+}|:=|\calB_{1,i,-}|$. 
    The proof follows the arguments made for $\calK_3$ above. More precisely, note that
    $$
        |g(\hbbeta)_k-g(\hbbeta_{/i})_k|=2
        \text{ for }
        k\in\calB_{1,i,-}.
    $$
    Thus
    \begin{align*}
        \max_{1\le i \le n}
        \sqrt{4|\calB_{1,i,-}|}
        =&~
        \max_{1\le i \le n}
        \sqrt{\sum_{k\in\calB_{1,i,-}}|g(\hbbeta)_k-g(\hbbeta_{/i})_k|^2}\\
        \le&~
        \max_{1\le i \le n}
        \|g(\hbbeta)-g(\hbbeta_{/i})\|\\
        \le& ~ \frac{2\norm{\bx_i}|\dot{\ell}_i(\hbbeta)|}{\lambda(1-\eta)}\\
        \leq &~ \frac{C\polylog(p)}{\lambda(1-\eta)}
    \end{align*}
    with probability at least $1-q_n-{\rm e}^{-cp}$, for sufficiently large $p$. The last inequality again follows from Part 4 of Lemma \ref{lem:beta-size}. 
    Hence with probability at least $1-q_n-{\rm e}^{-cp}$ we have
    $$
        \max_{1\le i\le n}|\calB_{1,i,-}|
        \le \frac{C\polylog(p)}{\lambda^2(1-\eta)^2}
        \leq C p^{\frac{7}{12}}\polylog(p)
    $$
    provided $\lambda^2(1-\eta)^2 = \omega(p^{-\frac{7}{12}})$. Since $\calB_{1,i,+}=\calB_{1,i}\setminus\calB_{1,i,-}$, we therefore have from \eqref{eq:b19c-card} that
    $$
        \max_{1\le i \le n}
        |(\calB_{1,i,+}\cup\calB_{0,i})^c|
        \le Cp^{11/12}(\log p)^{1/4}
    $$
    with probability at least $1-Cp^{-6}-q_n-{\rm e}^{-cp}$, for sufficiently large $p$. This finishes the proof.
    \end{proof}

	       


\section{Detailed proofs}\label{sec:proofs}

\subsection{Preliminaries}

\subsubsection{Basic linear algebra results}

\begin{lemma}[Weyl's Theorem][Theorem 4.3.1 in \cite{horn1994}]
\label{lem:Weyls}
Let $\bA, \bB \in \RR^{n\times n}$ be symmetric, and let the eigenvalues of $\bA, \bB$ and $\bA+\bB$ be $\{\lambda_i(\bA)\}_{i=1}^n$, $\{\lambda_i(\bB)\}_{i=1}^n$ and $\{\lambda_i(\bA + \bB)\}_{i=1}^n$, in increasing order. Then 
$$
|\lambda_i(\bA+\bB)-\lambda_i(\bA)|\le \lambda_1(\bB)
$$
for $i=1,\dots,n$.
\end{lemma}

\begin{lemma}[Woodbury Inversion Formula]
\label{lem:woodberry}
    Suppose $\bA\in \RR^{n\times n}$ is nonsingular, and $\bM = \bA + \bU \bB \bV$, then 
    \[
        \bM^{-1} = \bA^{-1} - \bA^{-1}\bU(\bB^{-1} + \bV\bA^{-1}\bU)^{-1}\bV\bA^{-1}
    \]
    provided that all relevant inverse matrices exist.
\end{lemma}

\begin{lemma}
\label{lem:BMIL}
    Let $\bM\in \RR^{n\times n}$ be non-singular, and partitioned as a 2-by-2 block matrix
    \[
    \bM = \begin{pmatrix}
        \bA & \bB\\
        \bB^\top & \bC
    \end{pmatrix}
    \]
    where $\bA\in \RR^{n_1\times n_1}$, $\bC\in\RR^{n_2\times n_2}$ with $n_1 + n_2 = n$. Then
    \[
    \bM^{-1} = \begin{pmatrix}
        \bA^{-1} + \bA^{-1}\bB\bD\bB^\top\bA^{-1} & -\bA^{-1}\bB \bD\\
        -\bD\bB^\top\bA^{-1} & \bD
    \end{pmatrix}
    \]
    where $\bD=(\bC - \bB^\top\bA\bB)^{-1}$, provided that all relevant inverse matrices exist.
\end{lemma}

\begin{lemma}\label{lem:woodtwice}
Suppose that $\bA \in \mathbb{R}^{p \times p}$ is an invertible matrix. Furthermore, assume that $\bC \in \mathbb{R}^{n \times n}$ is a diagonal matrix. Finally, $\bB \in\mathbb{R}^{p \times n}$. If $\bA + \bB \bC \bB^\top$ is invertible, then its inverse is:
\begin{align*}
&(\bA + \bB \bC \bB^\top)^{-1} = \bA^{-1} - \bA^{-1}\bB^{\top} \bC \bB \bA^{-1} \nonumber \\
 &+ \bA^{-1}\bB \bC\bB^\top (\bA + \bB\bC\bB^\top)^{-1}\bB\bC\bB^\top\bA^{-1}
\end{align*}
\end{lemma}
\begin{proof}
 First assume $\bC$ is invertible. 
 Applying Woodbury formula twice yields
 \begin{align*}
     &(\bA + \bB\bC\bB^\top)^{-1} \\
     =& \bA^{-1} 
     - \bA^{-1}\bB( 
 \bC^{-1} + \bB^\top\bA^{-1}\bB)^{-1}\bB^\top\bA^{-1}\\
     =& \bA^{-1} - \bA^{-1}\bB 
     \bC\bB^\top\bA^{-1}\\
     +& \bA^{-1}\bB \bC\bB^\top (\bA + \bB\bC\bB^\top)^{-1}\bB\bC\bB^\top\bA^{-1}
 \end{align*}
 If $\bC$ is not invertible, WLOG assume it has non-zero diagonal elements, i.e. by rearranging its rows and columns there is a diagonal matrix $\bC_1 \in \RR^{k\times k}$ such that 
 $$\bC = \diag({\bC_1, 0 ,..,0})$$
 Split $\bB$ in the same way:
 $$\bB = (\bB_1, \bB_2)$$
 where $\bB_1\in\RR^{p\times k}$ and $\bB_2\in\RR^{p\times(n-k)}$, and we have
 $$\bA + \bB\bC\bB^\top = \bA + \bB_1\bC_1\bB_1^\top$$
 so we can still use the above formula by replacing $\bB, \bC$ by $\bB_1,\bC_1$ respectively.
\end{proof}

\begin{lemma}\label{lem:inverseinflation}
Suppose that $\bA, \bGamma  \in \mathbb{R}^{n \times n}$ and that both $\bA+ \bGamma$ and $\bGamma$ are invertible. Then for any $\bv \in \mathbb{R}^n$ we have
\[
|\bv^{\top} (\bA + \bGamma)^{-1} \bv - \bv^{\top} (\bA )^{-1} \bv | \leq 
\frac{\lambda_{\max} (\Gamma) \bv^{\top} \bv}{\lambda_{\min}^2 (A)} + \frac{\lambda^2_{\max} (\Gamma) \bv^{\top} \bv}{\lambda_{\min}^2 (A) (\lambda_{\min}(A) - \lambda_{\max} (\Gamma))}
\]
\end{lemma}
\begin{proof}
Using Lemma \ref{lem:BMIL} we obtain
\begin{align*}
&|\bv^{\top} (\bA + \bGamma)^{-1} \bv - \bv^{\top} (\bA )^{-1} \bv | \nonumber \\
&\leq  | \bv^{\top} \bA^{-1} \bGamma \bA^{-1} \bv | + | \bv^{\top} \bA^{-1} \bGamma (\bA+ \bGamma)^{-1} \bGamma \bA^{-1} \bv | \nonumber \\
&\leq \frac{\lambda_{\max} (\Gamma) \bv^{\top} \bv}{\lambda_{\min}^2 (A)} + \frac{\lambda^2_{\max} (\Gamma) \bv^{\top} \bv}{\lambda_{\min}^2 (A) (\lambda_{\min}(A) - \lambda_{\max} (\Gamma))}. 
\end{align*}
In the last inequality, we have used Weyl's theorem for bounding the maximum eigenvalue of $(\bA+\bGamma)^{-1}$. 
\end{proof}

\subsubsection{Basic probability and statistics}

\begin{lemma}[Stirling's approximation]
\label{lem:stirling}
    For $1< s\leq p \in \ZZ$ and $p> 2$, we have 
    \[
    {p\choose s}\le 
    e^{s\log\frac{ep}{s}}
    \]
\end{lemma}

\begin{proof}
    By \cite{robbins1955}, 
    \[
        n! = \sqrt{2\pi}n^{n+\frac12}e^{-n}e^{r_n}
    \]
    with 
    \[
        \frac{1}{12n+1}\leq r_n \leq \frac{1}{12n}
    \]
    So 
    \begin{align}\label{eq:stir0}
        {p\choose s} &= \frac{\sqrt{2\pi}p^{p+\frac12}e^{-p}e^{r_p}}{\sqrt{2\pi}s^{s+\frac12}e^{-s}e^{r_s}\sqrt{2\pi}(p-s)^{(p-s)+\frac12}e^{-(p-s)}e^{r_(p-s)}} \nonumber \\
        &=\frac{1}{\sqrt{2\pi}}\sqrt{\frac{p}{s(p-s)}} \left(\frac{p}{s}\right)^s \left(\frac{p}{p-s}\right)^{p-s} e^{r_n-r_s - r_{n-s}}
    \end{align}
    The last term satisfies
    \begin{equation}\label{eq:stir1}
        e^{r_n-r_s-r_{n-s}} = e^{\frac{1}{12n} - \frac{1}{12s+1} - \frac{1}{12(n-s)+1}} \leq 1
    \end{equation}
    Furthermore, if we asuume that $1<s<p$, and $p>2$ then
    \begin{equation*}
\frac{p}{p(p-s)}<2. 
    \end{equation*}
    Hence,
     \begin{equation}\label{eq:stir2}
\frac{1}{\sqrt{2 \pi}} \sqrt{\frac{p}{p-s}} \leq \sqrt{\frac{1}{\pi}} \leq 1. 
     \end{equation}
     Finally,
     \begin{eqnarray}\label{eq:stir3}
\left(\frac{p-s}{s} \right)^{p-s} = {\rm e}^{(p-s) \log \frac{p}{p-s}} ={\rm e}^{(p-s) \log \left( 1+\frac{s}{p-s} \right)}  
\leq {\rm e}^{s }. 
     \end{eqnarray}
 Combining \eqref{eq:stir0}, \eqref{eq:stir1}, \eqref{eq:stir2}, and \eqref{eq:stir3} we conclude the result:
 \[
 {p \choose s} \leq {\rm e}^{s \log \frac{{\rm e} p}{s}}. 
 \]
\end{proof}

\begin{lemma}{Theorem 1.1 in \cite{rudelson2013hanson}}
\label{lem:HansonWrightIn}
    Let $\bx \in \RR^p$ be a random vector with independent sub-Gaussian entries that satisfy $\EE x_i =0$ and $\norm{x_i}_{\psi_2}\leq K$. Let $\bA \in \RR^{p\times p}$. Then for every $t>0$,
    \[
        \PP(\abs{\bx^\top\bA\bx - \EE \bx^\top\bA\bx}>t)
        \leq 2e^{-c\min\{\frac{t}{K^2\norm{\bA}}, \frac{t^2}{K^4\norm{\bA}_{HS}^2}\}},
    \] 
where $\norm{\bA}$, and $\norm{\bA}_{HS}$ denote, respectively, the spectral norm and the Hilbert-Schmidt norm of matrix $\bA$.
\end{lemma}

\begin{lemma} [Lemma 6 of \cite{jalali2016}] \label{lem:chi:sq:ind}
    Let $\bx \sim N(0,\bI_p)$, then
    \[
    \PP (\bx^{\top } \bx\geq p + pt ) \leq {\rm e}^{-    \frac{p}{2} (t- \log (1+t)) }
    \]
\end{lemma}

\begin{lemma}\label{lem:xi-row-conc}
    Let $\bx_1,\dots,\bx_n\stackrel{iid}{\sim} N(0,\bSigma)\in \RR^{p\times p }$ and suppose $\rho_{\max}(\bSigma)\leq p^{-1}C_X$ for some constant $C_X>0$, then 
    \[
        \PP (\max_{1\le i \le n}\|\bx_i\|\geq 2\sqrt{C_X} ) \leq ne^{-p/2}
    \]
\end{lemma}
\begin{proof}
    Let $\bz = \Sigma^{-\frac12}\bx$, then $\bz\sim \calN(0,\bI_p)$ and
    \begin{align*}
        \PP(\|\bx\|\geq 2\sqrt{C_X})&=\PP(\bz^\top\bSigma\bz\geq 4C_X)\\
        &\leq \PP(\bz^\top\bz \geq 4p)\\
        &\leq e^{-\frac{p}{2}(3-\log(4)) } \le e^{-p/2}
    \end{align*}
    The last line uses Lemma \ref{lem:chi:sq:ind}. A union bound over all $1\le i\le n$ finishes the proof.
\end{proof}

\begin{lemma}[Lemma 12 in \cite{rad2018scalable}]\label{lem:maxsingularvalue_0}
    $\bX \in \mathbb{R}^{p \times p}$ is composed of independently distributed $N(0, \bSigma)$ rows, with $\rho_{\max} = \sigma_{\max} (\Sigma)$, where $\Sigma \in \mathbb{R}^{p \times p}$. Then
    \[
    {\PP} (\|\bX^{\top} \bX\| \geq (\sqrt{n} + 3 \sqrt{p})^2 \rho_{\max}) \leq {\rm e}^{-p}. 
    \]
\end{lemma}

\begin{lemma}\label{lem:maxsingularvalue}
   If $\rho_{\max}\leq p^{-1}C_X$ and $n/p=\gamma_0$, then
    \[
    {\PP} (\|\bX^{\top} \bX\| \geq (\sqrt{\gamma_0} + 3 )^2 C_X) \leq {\rm e}^{-p}. 
    \]
\end{lemma}
This is a straightforward application of Lemma \ref{lem:maxsingularvalue_0}.

\begin{lemma} [Lemma 4.10 of \cite{chatterjee2014}]
Let $V_1, V_2, \ldots, V_p$ denote dependent zero mean Gaussian random variables with mean zero and variances $\sigma_1^2, \sigma_2^2, \ldots, \sigma_p^2$. We then have
\[
\mathbb{E} (\max_i V_i) \leq  \sqrt{2 \log 2p} \left( \max_i \sigma_{i} \right). 
\]
\end{lemma}

\begin{lemma}[Borell-TIS inequality]
Let $V_1, V_2, \ldots, V_p$ denote dependent zero mean Gaussian random variables with mean zero and variances $\sigma_1^2, \sigma_2^2, \ldots, \sigma_p^2$. Then,
\[
{\PP} (|\max_i V_i - \mathbb{E} (\max_i V_i)|>t)\leq 2 {\rm e}^{- \frac{t^2}{2 \sigma^2}},
\]
where $\sigma = \max (\sigma_1, \sigma_2, \ldots, \sigma_p)$. 
\end{lemma}
\begin{proof}
    It is a special case of the original Borell-TIS inequality, see Theorem 2.1.1 of \cite{adler2009}.
\end{proof}

\begin{corollary}\label{cor:maxGauss}
    Let $\bv \in \mathbb{R}^p$ be an $N(0, \bSigma)$ random vector.  Let $\rho_{\max}$ denotet the maximum eigenvalue of $\bSigma$. We then have
    \[
    {\PP} (\|\bv\|_\infty > \sqrt{2\rho_{\max} \log 2p} + t ) \leq 2 {\rm e}^{-\frac{t^2}{2 \rho_{\max}}}. 
    \]
\end{corollary}
\begin{proof}
    This corollary is a direct application of the previous two lemmas and using the fact that $\max_i \sigma_i^2 < \rho_{\max}$. 
\end{proof}

\subsubsection{Accuracy of ALO for smooth loss functions and regularizers}

\begin{theorem}[Theorem 3 in \cite{rad2018scalable}]\label{thm:kamiarandI}
    Under assumptions A1-A5 and B1-B5, with probability at least $1-4n{\rm e}^{-p}-\frac{8n}{p^3}-\frac{8n}{(n-1)^3} - q_n$ the following bound is valid:
    \begin{align*}
    &\max_{1\leq i\leq n} \abs*{
    \bx_i^\top \hbbeta_{/i} - \bx_i^\top \hbbeta - \left( \frac{\dot{\ell}_i(\hbbeta)}{\ddot{\ell}_i(\hbbeta)} \right) \left( \frac{H_{ii}}{1-H_{ii}} \right)} 
    \\
    \leq&~ \frac{C_0\polylog(n)}{\sqrt{p}}
    \end{align*}
    for some constant $C_0>0$.
\end{theorem}

\subsection{Proof of Lemma \ref{lem:beta-size}}\label{ssec:proof:lemma3}

The elastic net objective function is
$$
h(\bbeta)=\sum_{j=1}^n\ell(y_i;\bx_i^{\top}\bbeta)
+\lambda(1-\eta)\sum_{i=1}^p|\beta_i|
+\lambda\eta\bbeta^{\top}\bbeta
$$
and its surrogate smoothed objective function is
$$
h_{\alpha}(\bbeta)=
\sum_{j=1}^n\ell(y_i;\bx_i^{\top}\bbeta)
+\lambda(1-\eta)\sum_{i=1}^pr^{(1)}_{\alpha}(\beta_i)
+\lambda\eta\bbeta^{\top}\bbeta.
$$
where $r^{(1)}_{\alpha}$ is defined in \eqref{def:smooth:ell1}.

\begin{enumerate}
    \item According to Lemma \ref{lem:acc:r1},
$$
\sup_{\bbeta}|h(\bbeta)-h_{\alpha}(\bbeta)|\le \dfrac{2\lambda(1-\eta)p(\log 2)}{\alpha}
$$
Hence,
\begin{align*}
\label{eq:h-alfa-diff}
    0\le&~
    h_{\alpha}(\hat{\bbeta})-h_{\alpha}(\hat{\bbeta}^{\alpha})\\
    =& \,h_{\alpha}(\hat{\bbeta})-h(\hat{\bbeta}^{\alpha})+h(\hat{\bbeta}^{\alpha})-h_{\alpha}(\hat{\bbeta}^{\alpha})\\
    \le &~ h_{\alpha}(\hat{\bbeta})-h(\hat{\bbeta})+\dfrac{2\lambda (1-\eta)p(\log 2)}{\alpha}\\
    \le& \,\dfrac{4p\lambda (1-\eta)(\log 2)}{\alpha}.\numberthis
\end{align*}
In the first and second inequalities above, we have used the facts that $\hat{\bbeta}$ and $\hat{\bbeta}^{\alpha}$ are the optimizers of $h(\bbeta)$ and $h_{\alpha}(\bbeta)$ respectively. By the Taylor series expansion at $\bz=\hat{\bbeta}_{\alpha}$, we obtain
\begin{align*}\label{eq:h-alfa-diff2}
    &~h_{\alpha}(\hat{\bbeta})-h_{\alpha}(\hat{\bbeta}^{\alpha})\\
    =&~\nabla h_{\alpha}(\hat{\bbeta}^{\alpha})^{\top}(\hat{\bbeta}-\hat{\bbeta}^{\alpha})
    +\frac{1}{2}(\hat{\bbeta}-\hat{\bbeta}^{\alpha})^{\top}
    \nabla^2h_{\alpha}(\bxi)
    (\hat{\bbeta}-\hat{\bbeta}^{\alpha})\\
    =&~(\hat{\bbeta}-\hat{\bbeta}^{\alpha})^{\top}
    \nabla^2h_{\alpha}(\bxi)
    (\hat{\bbeta}-\hat{\bbeta}^{\alpha})/2\\
    \ge&~\lambda\eta\|\hat{\bbeta}-\hat{\bbeta}^{\alpha}\|^2.\numberthis
\end{align*}
Here $\bxi=t\hat{\bbeta}^{\alpha}+(1-t)\hat{\bbeta}$ for some $t\in [0,1]$. Note that to obtain the second equality, we have used the fact that $\nabla h_{\alpha}(\hat{\bbeta}^{\alpha}) =0$ due to the optimality of $\hat{\bbeta}^{\alpha}$. The last line of \eqref{eq:h-alfa-diff2} is due to the existence of the ridge penalty term $\lambda\eta\|\bbeta\|^2$ in $h_{\alpha}$. Comparing \eqref{eq:h-alfa-diff2} and \eqref{eq:h-alfa-diff}, one has that
\begin{equation}
    \label{eq:beta-l2-diff}
    \|\hat{\bbeta}-\hat{\bbeta}^{\alpha}\|\le \sqrt{\dfrac{4p(1-\eta)(\log 2)}{\alpha\eta}}\le \sqrt{\dfrac{4p(\log 2)}{\alpha\eta}}
\end{equation}
Using a similar approach we can also prove that
\[
	\|\hat{\bbeta}_{/i}-\hat{\bbeta}^{\alpha}_{/i}\|\le \sqrt{\dfrac{4p(\log 2)}{\alpha\eta}}
	.
\]
This finishes the proof of Part 1. \\

\item Consider the first-order optimality equations of $\hbbeta^\alpha$ and $\hbbeta^\alpha_{/i}$:
	\begin{align*}
		\sum_j \bx_j \dot{\ell}_j(\hbbeta^\alpha) + \lambda(1-\eta) \dot{r}_\alpha^{(1)}(\hbbeta^\alpha) +\lambda\eta \hbbeta^\alpha =&~0\\
		\sum_{j\neq i} \bx_j \dot{\ell}_j(\hbbeta^\alpha_{/i}) 
			+ \lambda(1-\eta) \dot{r}_\alpha^{(1)}(\hbbeta^\alpha_{/i}) 
			+ \lambda \eta \hbbeta^\alpha_{/i} =&~0.
	\end{align*}
	By subtracting one from the other we have
	\begin{align*}
	&0=\sum_{j\neq i} \bx_j [ \dot{\ell}_j(\hbbeta^\alpha) -\dot{\ell}_j(\hbbeta^\alpha_{/i})  ] 
	+ \bx_i \dot{\ell}_i(\hbbeta^\alpha)\\ 
	&~+ \lambda(1-\eta) [\dot{r}^{(1)}_\alpha(\hbbeta^\alpha) -\dot{r}^{(1)}_\alpha(\hbbeta^\alpha_{/i}) ] 
	+ 2 \lambda \eta \left( \hbbeta^\alpha - \hbbeta^\alpha_{/i} \right).
	\end{align*}
	It is straightforward to simplify this expression by using the mean value theorem for $\dot{\ell}_j(\hbbeta^\alpha) -\dot{\ell}_j(\hbbeta^\alpha_{/i})$ and $\dot{r}^{(1)}_\alpha(\hbbeta^\alpha) -\dot{r}^{(1)}_\alpha(\hbbeta^\alpha_{/i})$: 
  \begin{align}\label{eq:bound_diff_1}
 \left[\bX_{/i}^\top\diag{(\underline{\ddot{\ell}}^{\alpha/i})}\bX_{/i} 
	+ \lambda(1-\eta) \diag{\underline{\ddot{r}}^{(1)}
 _{\alpha/i})} 
	+2 \lambda\eta \II_p \right] \nonumber \\
	\left( \hbbeta^\alpha - \hbbeta^\alpha_{/i} \right) = - \bx_i \dot{\ell}_i(\hbbeta^\alpha).
    \end{align}
	In this equation, $\bX_{/i}$ is identical to $\bX$, except for the removal of the $i^{\rm th}$ row. As expected from the mean value theorem, in each diagonal element of the two terms $\diag{(\underline{\ddot{\ell}}^{\alpha/i})}$ and ${\rm diag} (\underline{\ddot{r}}^{(1)}
 _{\alpha/i})$ the second derivative of $\ell$ and $r$ are calculated at a point  $\xi=t\hat{\bbeta}^{\alpha}+(1-t)\hat{\bbeta}^{\alpha}_{/i}$ for some $t\in[0,1]$. The choice of $t$ can be different for different diagonal elements and is dictated by the mean value theorem. Defining $\diag(\underline{\ddot{r}}^{\alpha}) := (1-\eta) \diag({\underline{\ddot{r}}^{(1)}
 _{\alpha/i})} +2 \eta \II_p$, it is straightforward to use \eqref{eq:bound_diff_1} and obtain
	\begin{align*}
		&~\norm{\hbbeta^\alpha - \hbbeta^\alpha_{/i}}\\
		\le&~ 
		\norm*{ \left[\bX_{/i}^\top\diag{(\underline{\ddot{\ell}}^{\alpha/i})}\bX_{/i} + \lambda \diag(\underline{\ddot{r}}^{\alpha}) \right]^{-1} }
		 \norm{\bx_i}
		|\dot{\ell}_i(\hbbeta^\alpha)|\\
		\le&~ \frac{|\dot{\ell}_i(\hbbeta^{\alpha})|\norm{\bx_i}}{2\lambda\eta},
	\end{align*}
 where the last inequality is because of the existence of $2 \eta \II_p$ in $\diag(\underline{\ddot{r}}^{\alpha})$.  \\

\item According to Part 1, $\lim_{\alpha\to\infty}\hbbeta_{/i}^\alpha = \hbbeta_{/i}$, $\forall 0\leq i\leq n$. The result then follows by letting $\alpha\to \infty$ in Part 2, and using the fact that $\dot{\ell}$ is continuous.\\

\item  By Part 3,
    \begin{align}\label{eq:bd:hbeta-hebtaminusi:2}
        \norm{\hbbeta - \hbbeta_{/i}}
        \le&~ \frac{|\dot{\ell}(\hbbeta)|\norm{\bx_i}}{2\lambda\eta}
    \end{align}
    Note that the first order optimaility equations of $\hbbeta$ and $\hbbeta_{/i}$ are:
    \begin{align*}
        &\sum_j \bx_j \dot{\ell}_j(\hbbeta) + \lambda(1-\eta)g(\hbbeta) +2\lambda\eta \hbbeta =0\\
        &\sum_{j\neq i} \bx_j \dot{\ell}_j(\hbbeta_{/i}) 
        + \lambda(1-\eta) g(\hbbeta_{/i}) 
        + 2\lambda\eta \hbbeta_{/i} =0.
    \end{align*}
    By subtracting one from another and applying the mean value theorem we have 
    \begin{align*}
        &\lambda(1-\eta)
        [ g(\hbbeta) - g(\hbbeta_{/i}) ] \\
        =& -\sum_{j\neq i} \bx_j \left( \dot{\ell}_j(\hbbeta) - \dot{\ell}_j(\hbbeta_{/i}) \right) - \bx_i \dot{\ell}_i(\hbbeta) - 2\lambda\eta (\hbbeta-\hbbeta_{/i})\\
        =& - \left[\bX_{/i}^\top \diag [\underline{\ddot{\ell}}]\bX_{/i} + 2\lambda\eta \II_p \right](\hbbeta - \hbbeta_{/i}) - \bx_i \dot{\ell}_i(\hbbeta).
    \end{align*}
    Here, we have defined $\underline{\ddot{\ell}}$ as it was in the proof of Part 2. Under the event that
    \begin{align*}
        \{
        \sup_i\sup_{j\neq i}\sup_{t\in [0,1]}\ddot{\ell}_i(t\hbbeta + (1-t)\hbbeta_{/i})\leq&~  \polylog(n),\\
        \|\bX^\top\bX\|\leq&~ (\sqrt{\gamma_0}+3)^2C_X,\\
        \sup_i\|\bx_i\|\leq&~ 2\sqrt{C_X},\\
        \sup_i |\dot{\ell}_i(\hbbeta)|\leq&~  \polylog(n)\}
    \end{align*}
    with probability at least $1-q_n-{\rm e}^{-p}-\check{q}_n-n{\rm e}^{-p/2}$ according to Assumption A4, Lemma \ref{lem:xi-row-conc} and Lemma \ref{lem:maxsingularvalue}, we have
    \begin{align*}
        &\lambda(1-\eta)\norm{g(\hbbeta) - g(\hbbeta_{/i}) }\\
        \leq&~  (   \polylog(n) \|\bX^\top\bX\| +2\lambda\eta)  \norm{\hbbeta - \hbbeta_{/i}} + \norm{\bx_i}\abs{\dot{\ell}_i(\hbbeta)} \\
        \leq&~ \left(\frac{(\sqrt{\gamma_0}+3)^2C_X\polylog(n)}{2\lambda\eta}+2\right)\|\bx_i\||\dot{\ell}_i(\hbbeta)|\\
        \leq&~ \frac{\polylog(n)}{\lambda\eta}
    \end{align*}
    where in the second inequality we also used \eqref{eq:bd:hbeta-hebtaminusi:2}. 
    Combining the above results we have
    \begin{align*}
       \max_{1\le i\le n} \norm{g(\hbbeta) - g(\hbbeta_{/i}) }
        &\leq \frac{\polylog(n)}{\lambda^2\eta(1-\eta)}
    \end{align*}
    with probability at least $1-q_n-{\rm e}^{-p}-\check{q}_n-ne^{-p/2}$.
\item $(k\in \calS^{(1)})$ We only provide a proof for $\hbbeta$ and  $\hbbeta^\alpha$, since the arguments are exactly the same for the leave-one-out estimators $\hbbeta_{/i}$ and $\hbbeta^{\alpha}_{/i}$.

For $ k\in \calS^{(1)}$, we have $|\hbeta_k|>\kappa_1$ so
$$\abs{\hbeta_k^\alpha} \geq \abs{\hbeta_k}- \norm{\hbbeta - \hbbeta^\alpha}\geq \kappa_1 - \sqrt{\frac{4p(1-\eta)\log2}{\alpha\eta}}\geq \frac{\kappa_1}{2} $$
provided that $\alpha\eta\kappa_1^2\geq 16(1-\eta)(\log2)p $. The second inequality uses \eqref{eq:beta-l2-diff} to bound $\norm{\hbbeta - \hbbeta^\alpha}$. \\

\item $(k\in \calS^{(0)})$ From the first order optimality conditions on $\hat{\bbeta}_{/i}$ and $\hat{\bbeta}^{\alpha}_{/i}$, we have
\begin{align*}
	&\sum_{j\neq i}\bx_j\dot{\ell}_j(\hbbeta_{/i})
	+\lambda (1-\eta)g(\hbbeta_{/i})+2\lambda\eta\hbbeta_{/i}
	=0\\
	&\sum_{j\neq i}\bx_j\dot{\ell}_j(\hbbeta^{\alpha}_{/i})
	+\lambda(1-\eta)\nabla r_{\alpha}^{(1)}(\hbbeta^{\alpha}_{/i})
	+2\lambda\eta\hbbeta^{\alpha}_{/i}=0.
\end{align*}
By subtracting the two equalities we obtain
\begin{align}\label{eq:fodifference1}
    & \nabla r_{\alpha}^{(1)}(\hat{\bbeta}^{\alpha}_{/i})
    -g(\hat{\bbeta}_{/i})\nonumber \\
    =& \dfrac{1}{\lambda(1-\eta)}\left(\sum_{j\neq i}\bx_j(\dot{\ell}_j(\hat{\bbeta}_{/i})-\dot{\ell}_j(\hat{\bbeta}^{\alpha}_{/i}))+
    2\lambda\eta(\hat{\bbeta}_{/i}-\hat{\bbeta}^{\alpha}_{/i})\right)
    \nonumber \\
    =& \dfrac{1}{\lambda(1-\eta)}\left( \bX_{/i}^\top \diag[\ddot{\ell}_j(\bxi_i)]_{j\neq i}\bX_{/i} + 2\lambda\eta \II_p \right) (\hbbeta_{/i}-\hbbeta^\alpha_{/i})
\end{align}

The last line follows from the mean value theorem applied to $\dot{\ell}(\cdot)$, and  $\bxi_i^\alpha=t\hat{\bbeta}_{/i}+(1-t)\hat{\bbeta}^{\alpha}_{/i}$ for some $t\in[0,1]$, where $t$ can be different for different $i,\;j$. By Parts 1-3, we have $\forall i$, $\bxi_i$ lies in set $\cD$ in Assumption A4 for large enough $p$. To see this, let $\bxi_i:=t\hat{\bbeta}_{/i}+(1-t)\hat{\bbeta}_{/i}$. By definition of $\cD$, $\bxi_i\in\cD$. Now since
\[
\|\bxi_i^\alpha-\bxi_i\|_2 = (1-t)\| \hbbeta_{/i}^\alpha - \hbbeta_{/i} \|_2\leq \sqrt{\frac{4p(\log 2)}{\alpha\eta}},
\]
the difference can be arbitrarily small for large $p$, when we assume $\alpha\eta = \omega(p)$. So there exists $p_0$ such that $\forall p \geq p_0$, $\|\bxi_i^\alpha-\bxi_i\|_2\leq \starepsilon$, i.e. $\bxi_i^\alpha \in \cD$.

So $\max_{0\leq i\leq n, j\neq i} \ddot{\ell}_j(\bxi_i)\leq \polylog(n)$ with probabilty at least $1-q_n$.\footnote{Please note that for notational simplicity we use the same notation for all the terms that are polynomial functions of $\log(n)$, and have dropped the subscript $2$ of the term $\polylog(n)$ that appeared in Assumption A4.} Then we have
\begin{align}
    &\max_{0\le i\le n}\norm{\nabla r_{\alpha}^{(1)}(\hat{\bbeta}^{\alpha}_{/i})
        -g(\hat{\bbeta}_{/i}) }\nonumber\\
    \le & \frac{1}{\lambda(1-\eta)} ( \max_{0\le i\le n, j\neq i }|\ddot{\ell}_j(\bxi_i)|\|\bX^{\top}\bX\|+2\lambda\eta)
    \|\hat{\bbeta}_{/i}-\hat{\bbeta}^{\alpha}_{/i}\|\nonumber \\
    \le & \frac{1}{\lambda(1-\eta)}(\polylog(n)(\sqrt{\gamma_0}+3)^2 C_X+2\lambda\eta )\sqrt{\frac{4p(1-\eta)(\log 2)}{\alpha\eta }}\nonumber\\
    \leq & \frac{\polylog(n)}{\lambda(1-\eta)}\sqrt{\frac{p (1-\eta)}{\alpha\eta}}\label{eq:fodifference2}
\end{align}
with probability at least $1-q_n-e^{-p}$.
The second inequality uses Lemma \ref{lem:maxsingularvalue} to bound $\|\bX^{\top}\bX\|$ and Part 1 to bound $\|\hat{\bbeta}-\hat{\bbeta}^{\alpha}\|$. The last line uses boundedness of $\lambda$ and $\eta$ (Assumption A5) to absorb $2\lambda\eta$ into the constant $C$.
\end{enumerate}

Without loss of generality we assume that $0<g(\hat{\bbeta})_k\leq 1-\kappa_0$ (negative subgradients can be handled similarly). 
We first obtain
\begin{align*}
    \nabla r_{\alpha}^{(1)}(\hat{\bbeta}^{\alpha}_{/i})_k-1
    =&\, \dfrac{e^{\alpha\hat{\beta}^{\alpha}_{/i,k}}-e^{-\alpha\hat{\beta}^{\alpha}_{/i,k}}}
    {e^{\alpha\hat{\beta}^{\alpha}_{/i,k}}+e^{-\alpha\hat{\beta}^{\alpha}_{/i,k}}+2}-1 \nonumber \\
    =& -\frac{2}{1+e^{\alpha\hbeta_{/i,k}^\alpha}}.
\end{align*}
It follows from (\ref{eq:fodifference2}) that, with probability at least $1-q_n-{\rm e}^{-p}$, $\forall i\ge 0$, $k\in [p]$:
\begin{align*}
\label{eq:grad-diff-2}
    -\frac{\polylog(n)}{\lambda(1-\eta)}\sqrt{\frac{p (1-\eta)}{\alpha\eta}}
    \le&~1-g(\hat{\bbeta}_{/i})_{k}-\frac{2}{1+e^{\alpha\hbeta_{/i,k}^\alpha}}\\
    \le&~\frac{\polylog(n)}{\lambda(1-\eta)}\sqrt{\frac{p(1-\eta)}{\alpha\eta}}
\numberthis
\end{align*}

By rearranging the terms in the second inequality of \eqref{eq:grad-diff-2} and using the fact that $1-g(\hbbeta_{/i})_{k}\geq \kappa_0$, we obtain
\begin{align*}
    \kappa_0-\frac{\polylog(n)}{\lambda(1-\eta)}\sqrt{\frac{(1-\eta)}{\alpha\eta}}
    \le&~ \dfrac{2}{1+e^{\alpha\hat{\beta}_{/i,k}^{\alpha}}}\leq 2 e^{-\alpha\hat{\beta}_{/i,k}^{\alpha}}
\end{align*}
Therefore 
\begin{align*}
    \hbeta_{/i,k}^{\alpha} 
    &\leq \frac{1}{\alpha}
    \left(
        \log 2 - \log 
        \left[
            \kappa_0-\frac{\polylog(n)}{\lambda(1-\eta)}\sqrt{\frac{p(1-\eta)}{\alpha\eta}}
        \right]
    \right)
\end{align*}
By our assumption that $\alpha = \omega\left(\frac{n\polylog(n)}{\kappa_0^2\lambda^2(1-\eta)\eta}\right)$, 
we have 
\[
    \kappa_0 - \frac{\polylog(n)}{\lambda(1-\eta)}\sqrt{\frac{p(1-\eta)}{\alpha\eta}} \geq \frac12 \kappa_0,
\]
therefore with probability at least $1-q_n-{\rm e}^{-p}$, $\forall i\ge 0$, $k\in \calS_{/i}^{(0)}$:
\begin{align*}
    \hbeta_{/i,k}^{\alpha} 
    & \leq \frac{1}{\alpha}
    \left(
        \log 2 - \log \left(\frac12 \kappa_0\right)
    \right)= \frac{1}{\alpha}\log \left(\frac{4}{\kappa_0}\right).
\end{align*}
Using a symmetric argument on the case of negative subgradients, we conclude that
\[
    \max_{0\le i \le n}\max_{k\in \calS_{/i}^{(0)}}|\hbeta_{/i,k}^{\alpha}| 
    \le  \frac{1}{\alpha}\log \left(\frac{4}{\kappa_0}\right)
\]
with probability at least $1-q_n-{\rm e}^{-p}$.
\qed

\subsection{Proof of Lemma \ref{lem:rdderiv}}\label{ssec:proof:lemma4}

Recall that the penalty function is
$$
r_{\alpha}(z)=\eta z^2+\frac{1-\eta}{\alpha}\cdot\left(\log(1+e^{\alpha z})+
\log(1+e^{-\alpha z})\right),
$$
It can be verified that
\begin{equation}\label{eq:rdderiv}
    \ddot{r}_{\alpha}(z)=2\eta+(1-\eta)\cdot \dfrac{2\alpha}{e^{\alpha z}+e^{-\alpha z}+2}.
\end{equation}
We have 
$$
\int_0^1 \ddot{r}_\alpha(\hbeta_{/i,k}^\alpha-t\Delta_{/i,k}^\alpha)dt
=\frac{\dot{r}_\alpha(\hbeta_k^\alpha) - \dot{r}_\alpha(\hbeta_{/i,k}^\alpha)}{\hbeta_k^\alpha -\hbeta_{/i,k}^\alpha }.
$$
Note that $\dot{r}_\alpha(z)=2\eta z+(1-\eta)\left( 1-\frac{2}{1+e^{\alpha z}} \right)$ is an increasing odd function, concave on $[0,+\infty)$ and convex on $(-\infty,0]$. Furthermore, using the fact that $\frac14 e^{-x}\leq \frac{e^x}{(1+e^x)^2}\leq e^{-x}$ for $x\geq 0$, we have, for $z\geq 0$:
\begin{align}\label{eq:rddot:lowerbd}
    \ddot{r}_\alpha(z)\geq 2\eta + \frac12\alpha(1-\eta)e^{-\alpha z},
\end{align}
and 
\begin{align}\label{eq:rddot:upperbd}
    \ddot{r}_\alpha(z)
    \leq 2\eta + 2\alpha(1-\eta)e^{-\alpha z}.
\end{align}
With this background, we can now state the proof of each part.  
\begin{enumerate}
    \item \textbf{($k\in\calB_{0,i}$):} 
    By Part 6 of Lemma~\ref{lem:beta-size}, for large enough $p$, with probability at least $1-q_n-e^{-p}$ we have $\forall i, \forall k\in\calB_{0,i}$:
    \[
    \max\left\{
      |\hbeta_k^\alpha|,\; |\hbeta_{/i,k}^\alpha|
            \right\}
      \leq \frac{1}{\alpha}
      \log\left(\frac{4}{\kappa_0}\right).
    \]
    With the same probability we then have
    \begin{align*}
        \int_0^1 \ddot{r}_\alpha(\hbeta_{/i,k}^\alpha-t\Delta_{/i,k}^\alpha)dt &\overset{(a)}{\geq} \ddot{r}_\alpha(\frac{1}{\alpha}\log(\frac{4}{\kappa_0})) \\
        &\overset{(b)}{\ge}
        2\eta  + \frac12 \alpha (1-\eta) \frac{\kappa_0}{4}\\
        &=2\eta + \frac18 \alpha (1-\eta) \kappa_0.
    \end{align*}
    Inequality (a) is because $\ddot{r}_\alpha(z)$ is decreasing in $\abs{z}$. Inequality (b) uses (\ref{eq:rddot:upperbd}).

    \item \textbf{($k\in\calB_{0,i}$):} An argument similar to the one presented for part (1) proves
    \begin{align*}
        \ddot{r}_{\alpha}(\hat{\beta}_k^{\alpha})
        \ge&~ 2\eta + \frac18 \alpha (1-\eta) \kappa_0.
    \end{align*}
    with probability at least $1-q_n-e^{-p}$.

    \item \textbf{($k\in\calB_{1,i,+}$):} 

    \begin{align*}
        2\eta&\leq \int_0^1 \ddot{r}_\alpha(\hbeta_{/i,k}^\alpha-t\Delta_{/i,k}^\alpha)dt\\
        &\leq \ddot{r}_\alpha\left(\frac{\kappa_1}{2}\right)\leq 2\eta + 2\alpha e^{-\frac12\alpha\kappa_1}.
    \end{align*}
    
    \item\textbf{($k\in\calB_{1,i,+}$):} Similarly we have
    \begin{align*}
        2\eta 
        \le\ddot{r}_{\alpha}(\hat{\beta}^{\alpha}_k)
        \le2\eta+2\alpha e^{-\frac12\alpha\kappa_1}.
    \end{align*}
    \item\textbf{($k\in\calB_{1,i,+}$):} The proof is identical to 4) by substituting $\hbbeta^{\alpha}$ with $\hbbeta^{\alpha}_{/i}$. \qed
	\end{enumerate}

\subsection{Proof of Theorem \ref{th:rduc2supp}}\label{ssec:proof:thm2}

 We begin with the proof of \eqref{eq:rduc_int}. To simplify notation we will use the compact notation $\Ldd_{/i}$ and $\Rdd_{/i}$ to denote the diagonal matrices ${\rm diag}[\int_0^1\ddot{\ell}_{/i}(\btheta(t))dt]$ and ${\rm diag}[\int_0^1\ddot{r}_{\alpha}(\btheta(t)) dt]$ respectively, where $\btheta=t\hat{\bbeta}^\alpha+(1-t)\hat{\bbeta}_{/i}^\alpha$. We also fix an index $i$ and write $\calB_{1,+}$ and $\calB_0$ to denote $\calB_{1,i,+}$ and $\calB_{0,i}$ respectively.

	Plugging in $\btheta=t\hat{\bbeta}^\alpha
	+(1-t)\hat{\bbeta}_{/i}^\alpha$, with a possible permutation of the rows and columns of $\bJ_{/i}(\btheta)$, by \eqref{eq:jaco}, we have that
	\begin{align*}
		&\int_0^1\bJ_{/i}(\btheta(t))dt
		=:\begin{pmatrix}
			\bA &\bB\\
			\bB^\top &\bC
		\end{pmatrix}\\
		=&\begin{pmatrix}
			\lambda\Rdd_{/i,\calB_0^c}
			+\bX_{/i,\calB_0^c}^\top
			\Ldd_{/i}\bX_{/i,\calB_0^c}
			&\bX_{/i,\calB_0^c}^\top
			\Ldd_{/i}
			\bX_{/i,\calB_0}\\
			\bX_{/i,\calB_0}^\top
			\Ldd_{/i}
			\bX_{/i,\calB_0^c}
			&
			\lambda 
			\Rdd_{/i,\calB_0}+
			\bX_{/i,\calB_0}^\top
			\Ldd_{/i}
			\bX_{/i,\calB_0}
		\end{pmatrix}.
	\end{align*}
	By the block matrix inversion lemma, i.e. Lemma \ref{lem:BMIL}, we have
	\begin{align*}
		&~\left(\int_0^1\bJ_{/i}(\btheta(t))dt\right)^{-1}\\
		=&~ \begin{pmatrix}
			\bA^{-1}+\bA^{-1}\bB\bD\bB^\top\bA^{-1} & -\bA^{-1}\bB\bD\\
			-\bD\bB^\top\bA^{-1} &\bD
		\end{pmatrix}
	\end{align*}
	where $\bD:=(\bC-\bB^\top\bA^{-1}\bB)^{-1}$. We will now estimate the norms of each of these terms separately, using Lemma \ref{lem:rdderiv}.
	
	\bigskip
	\noindent
	\textbf{Bounding $\|\bA^{-1}\|$:} Note that for $\btheta(t)=t\hat{\bbeta}^\alpha
	+(1-t)\hat{\bbeta}_{/i}^\alpha$, we have
	\begin{align*}
				\sigma_{\min}(\bA) = &~\sigma_{\min}\left(
		\lambda\Rdd_{/i,\calB_0^c}
		+\bX_{/i,\calB_0^c}^\top
		\Ldd_{/i}
		\bX_{/i,\calB_0^c}
		\right)\\
		\overset{(a)}{\ge}&~
	\sigma_{\min}\left(\lambda\Rdd_{/i,\calB_0^c}	\right)
		\ge 2\lambda\eta,
	\end{align*}
 where for inequality (a) we have used the fact that the matrix $\bX_{/i,\calB_0^c}^\top
		\Ldd_{/i}
		\bX_{/i,\calB_0^c}$ is positive semidefinite because of the convexity of the loss function.
	Hence,
\begin{equation}\label{eq:Abdlow}
		\|\bA^{-1}\|=\dfrac{1}{\sigma_{\min}(\bA)}\le (2\lambda\eta)^{-1}.
	\end{equation}
	
	\bigskip
	
	\noindent
	\textbf{Bounding $\|\bB\|$:} If $\SS^{p-1}$ denotes the unit sphere in $\mathbb{R}^p$, then we can write the cross term
	\begin{align*}\label{eq:Bbdup}
		\|\bB\|
		=&~\sup_{\bu\in \SS^{|\calB_0^c|},\bv\in\SS^{|\calB_0|}}\bu^\top\bB\bv\\
		\le&
		\sup_{\bu\in \SS^{|\calB_0^c|},\bv\in\SS^{|\calB_0|}}{\bu}^\top
		\left(\bX_{/i,\calB_0^c}^\top
		\Ldd_{/i}
		\bX_{/i,\calB_0}
		\right){\bv}\\
		\le&~\sup_{\tilde{\bu},\tilde{\bv}\in \SS^{p}}\tilde{\bu}^\top
		\left(\bX_{/i}^\top
		\Ldd_{/i}
		\bX_{/i}
		\right)\tilde{\bv}\\
		\le&~\max_{j\neq i}\int\limits_0^1\ddot{\ell}_j(\btheta(t))dt\|\bX_{/i}^\top\bX_{/i}\|\\
		\le&~%
  \polylog(n)\|\bX^{\top}\bX\|
  \numberthis
	\end{align*}

 In the third line above we define $\tilde{\bu}$ and $\tilde{\bv}$ based on $\bu$ and $\bv$ as follows. Let $\calB_0^c=\{j_1,\dots,j_{|\calB_0^c|}\}$. Then we define $\tilde{\bu}_{j_k}=\bu_k$ for $k\in [|\calB_0^c|]$, and $\tilde{\bu}_j=0$ for all $j\in \calB_0$. $\tilde{\bv}$ is defined similarly. Note that $\tilde{\bu}$, $\tilde{\bv}\in \SS^{p-1}$ so that the supremum in the next line is justified. In the last line we use the second part of \eqref{eq:assn_loss}, in conjunction with part 1) of Lemma~\ref{lem:beta-size}.\footnote{Please note that as mentioned before, for notational simplicity, we have dropped the subscript $24$ of the term $\polylog(n)$ that appeared in the second part of \eqref{eq:assn_loss}, use the same notation for all the terms that are polynomial functions of $\log(n)$. }
	
	\bigskip
	
	\noindent
	\textbf{Bounding $\|\bD\|$:} By the repeated use of Weyl's theorem, i.e.,  Lemma \ref{lem:Weyls}, we have
	\begin{align}\label{eq:lmin-lbdrddot}
		&~\sigma_{\min}(\bC-\bB^{\top}\bA^{-1}\bB)\nonumber \\
		&\ge~
		\lambda
		\min\left\{
		\int_0^1
		\ddot{r}_\alpha(\btheta(t))_kdt:
		k\in\calB_0
		\right\}
        -\|\bB\|^2\|\bA^{-1}\|
		\nonumber \\
	    &\ge~
	    \lambda
		\min\left\{
		\int_0^1
		\ddot{r}_\alpha(\btheta(t))_kdt:
		k\in\calB_0
		\right\}\nonumber \\
        &~- 
        \frac{(\polylog(n)\|\bX^{\top}\bX\|)^2}{2 \lambda \eta},
	\end{align}
 where in the last line we use the upper bounds on $\|\bB\|$ and $\|\bA^{-1}\|$ we obtained above. From Lemma \ref{lem:rdderiv} we have
\begin{align}\label{eq:lbd:rddot2}
        \min_{k\in \calB_{0}}
        \int_0^1
        \ddot{r}_\alpha(\btheta(t))_kdt 
        \ge 2\eta + \frac18 \alpha (1-\eta)\kappa_0,
    \end{align}	
By combining~\eqref{eq:lmin-lbdrddot} and \eqref{eq:lbd:rddot2} we obtain
\begin{equation}\label{eq:lowerbd:D}
    \sigma_{\min}(\bC-\bB^{\top}\bA^{-1}\bB)
    \ge \frac{\lambda\alpha}{16} (1-\eta)\kappa_0.
\end{equation}
provided 
\[
    2\lambda \eta + \frac{\lambda\alpha}{16} (1-\eta)\kappa_0
    \ge 
    \frac{(\polylog(n)\|\bX^{\top}\bX\|)^2}{2\lambda\eta}
\]
 Note that according to Lemma \ref{lem:maxsingularvalue}, $\|\bX^{\top}\bX\|$ is bounded by a constant with probability at least $1-{\rm e}^{-p}$. By the assumptions of this theorem, we know that $\alpha$ grows fast enough, and thus the above event holds with probability at least $1-{\rm e}^{-p}$.
 
  Hence, under the event $\|\bX^{\top}\bX\|\le C$, which occurs with probability at least $1-{\rm e}^{-p}$, we expect $\|\bD\|$ to go to zero as $\alpha \rightarrow \infty$. 
 Therefore,
	\begin{align*}\label{eq:Dbdup}
		\|\bD\|=&~\|(\bC-\bB^\top\bA^{-1}\bB)^{-1}\|\\
		=&~(\sigma_{\min}(\bC-\bB^\top\bA^{-1}\bB))^{-1}\le 
        \dfrac{16}
        {\lambda\alpha(1-\eta)\kappa_0}. 
		\numberthis
	\end{align*}
 
	By the block matrix inversion lemma,
	\begin{align*}
		&\big\vert\bx_i^\top
			\left(\int_0^1\bJ_{/i}(t\hat{\bbeta}^\alpha
			+(1-t)\hat{\bbeta}_{/i}^\alpha)dt\right)^{-1}\bx_i\\
		&~-\bx_{i,\calB_0^c}^\top 
			(\lambda{\rm diag}(\ddot{\underline{r}}_{\calB_0^c}^{\alpha/i}))
			+\bX^\top_{/i,\calB_0^c}{\rm diag}
			(\ddot{\underline{\ell}}^{\alpha/i})\bX_{/i,\calB_0^c})^{-1}\bx_{i,\calB_0^c}\big\vert\\
		=&~\left\vert\bx_i^\top\begin{pmatrix}
				\bA^{-1}+\bA^{-1}\bB\bD\bB^\top\bA^{-1} & -\bA^{-1}\bB\bD\\
				-\bD\bB^\top\bA^{-1} &\bD
			\end{pmatrix}\bx_i-\bx_{i,\calB_0^c}^\top\bA^{-1}\bx_{i,\calB_0^c}
            \right\vert\\
		    \leq&~\bx_{i,\calB_0^c}^\top\bA^{-1}\bB\bD\bB^\top\bA^{-1}\bx_{i,\calB_0^c}
  +2|\bx_{i,\calB_0^c}^{\top}\bA^{-1}\bB\bD\bx_{i,\calB_0}|+
  \bx_{i,\calB_0}^\top\bD\bx_{i,\calB_0}\\
  \le&~ \|\bx_{i,\calB_0^c}\|^2 \|\bA^{-1}\|^2 \|\bB\|^2 \|\bD\|+ 2 \|\bx_{i,\calB_0^c}\| \|\bA^{-1}\bB\bD\bx_{i,\calB_0}\| +\|\bx_{i,\calB_0}\|^2 \|\bD\| \nonumber \\
    \le&~    \|\bx_{i}\|^2 \|\bA^{-1}\|^2 \|\bB\|^2 \|\bD\|+ 2 \|\bx_{i,\calB_0^c}\| \|\bA^{-1}\| \|\bB\|\|\bD\|\|\bx_{i,\calB_0}\| +\|\bx_{i}\|^2 \|\bD\| \nonumber \\
\le&~    \|\bx_{i}\|^2 \|\bD\|( \|\bA^{-1}\|^2 \|\bB\|^2 + 2  \|\bA^{-1}\| \|\bB\|+ 1) \nonumber \\ =&~    \|\bx_{i}\|^2 \|\bD\|( \|\bA^{-1}\|\|\bB\|+ 1)^2 \nonumber \\
\le&~ \frac{16\|\bx_i\|^2}{\lambda\alpha(1-\eta)\kappa_0}\left(\frac{\polylog(n)\|\bX^{\top}\bX\|}{2\lambda\eta}+1\right)^2.
	\end{align*}
 In the last step we use the previously derived bounds on the operator norms of $\bA^{-1}$, $\bB$ and $\bD$, along with the Cauchy-Schwarz inequality. As discussed before, this error will be small for large values of $\alpha$. 

To show \eqref{eq:rduc_beta_del_i}, we define $\btheta:=\hat{\bbeta}_{/i}^{\alpha}-\Delta_{/i}^\alpha$. The bounds then follow by the same steps with very minor modifications.  \qed
\medskip

\subsection{Proof of Theorem~\ref{thm:FINAL_STEP}}\label{ssec:proof:thm:3}

We will use the notation
\begin{align*}\label{eq:thm3-not}
\omega_s:= &~
\|\bX^{\top}\bX\|
\sup_{t\in[0,1]}
\max_{
1\le i\le n}
\ddot{\ell}_i(
t\hat{\bbeta}^{\alpha}
+(1-t)\hat{\bbeta}_{/i}^{\alpha}
)
\,;\\
\rho(\alpha)
:=&~\lambda\left(
2\eta+\dfrac{\alpha(1-\eta)\kappa_0}{8}
\right).\numberthis
\end{align*}
By the assumptions of the Theorem we define the following two events:
\begin{enumerate}
    \item $\calA_{L}:=\{
    \text{Assumption A4 holds}
    \}$
    \item $
    \calA_s:=\{
\max_{1\le i\le n}
|\calB_{0,i}^c\setminus\calB_{1,i,+}|
\le d_n\}
$
\end{enumerate}
which hold with probabilities at least $1-q_n$ and $1-\tilde{q}_n$ respectively. We will prove the theorem assuming that the two events above both hold, which by the union bound, happens with probability at least $1-q_n-\tilde{q}_n$.

Our strategy will be to bound the following quadratic forms:
  	\begin{align}\label{eq:rduc_int3}
   		&\bigg\vert
			\bx_{i,\calF}^\top 
			(\lambda{\rm diag}(\ddot{\underline{r}}_{\calF}^{\alpha/i})
			+\bX^\top_{\calF}{\rm diag}	(\ddot{\underline{\ell}}^{\alpha/i})\bX_{\calF})^{-1}\bx_{i,\calF}
            \nonumber\\
            &~- \bx_{i,\calB_{1,i,+}}^\top
			(\lambda{\rm diag}(\ddot{\utilde{r}}_{\calB_{1,i,+}}^{\alpha/i})
			+\bX^\top_{\calB_{1,i,+}}{\rm diag}
	(\ddot{\utilde{\ell}}^{\alpha/i})\bX_{\calB_{1,i,+}})^{-1}\bx_{i,\calB_{1,i,+}}\bigg\vert \nonumber \\
 &\leq f(d_n),
			\end{align}
	and similarly
 	\begin{align*}\label{eq:rduc_beta_del_i4}
 		&~\bigg\vert \bx_{i,\calF}^\top 
 			(\lambda{\rm diag}(\ddot{\utilde{r}}_{\calF}^{\alpha/i})
 			+\bX^\top_{\calF}{\rm diag} 	(\ddot{\utilde{\ell}}^{\alpha/i})\bX_{\calF})^{-1}\bx_{i,\calF} \\
            &~-\bx_{i,\calB_{1,i,+}}^\top
            (\lambda{\rm diag}(\ddot{\utilde{r}}_{\calB_{1,i,+}}^{\alpha/i})
 			+\bX^\top_{\calB_{1,i,+}}{\rm diag}
 (\ddot{\utilde{\ell}}^{\alpha/i})\bX_{\calB_{1,i,+}})^{-1}\bx_{i,\calB_{1,i,+}} \bigg\vert\\
   		&~\le f(d_n)\numberthis
 	\end{align*}
for a suitable function $f(\cdot)$ of $d_n$, i.e., the cardinality of set differences.
 

In the rest of the proof, we fix an index $i$ and write $\calB_{1,+}$ and $\calB_0$ to denote $\calB_{1,i,+}$ and $\calB_{0,i}$ respectively. Consider the following decomposition 
	\begin{equation}\label{eq:defA1B1C1}
		\bH_1:=\lambda\,{\rm diag}(\ddot{\underline{r}}_{\calB_0^c}^{\alpha/i})
		+\bX^\top_{/i,\calF}{\rm diag}
		(\ddot{\underline{\ell}}^{\alpha/i})\bX_{/i,\calF}
		=\begin{pmatrix}
			\bA_1 &\bB_1\\
			\bB_1^\top &\bC_1
		\end{pmatrix}
	\end{equation}
	that is obtained by a permutation of rows and columns such that the rows and columns of $\bA_1$ belong to $\calB_{1,+}$, and rows and columns of $\bC_1$ belong to $\calF \setminus\calB_{1,+}$. Define
 $\btheta(t)=t\hat{\bbeta}^\alpha
	+(1-t)\hat{\bbeta}_{/i}^\alpha$. Then, we have
	\begin{eqnarray*}
		\bA_1   &=&~{\rm diag}\left[
        \lambda
        \left\{
        \int_0^1
        \ddot{r}_{\alpha}(\btheta(t))dt\right\}_{\calB_{1,+}}\right] 
        +\bX_{/i,\calB_{1,+}}^\top
		{\rm diag}\left[
        \int_0^1 
        \ddot{\ell}_{/i}(\btheta(t))dt
        \right]\bX_{/i,\calB_{1,+}}\\
		\bB_1&=&~\bX_{/i,\calB_{1,+}}^\top
        {\rm diag}\left[
        \int_0^1 
        \ddot{\ell}_{/i}(\btheta(t))dt
        \right]
  \bX_{/i,\calF \setminus\calB_{1,+}}\\
		\bC_1&=&~{\rm diag}\left[
        \lambda\int_0^1
        \ddot{r}_{\alpha}(\btheta(t))_{\calF\setminus\calB_{1,+}})dt \right] 
        + \bX_{/i,\calF\setminus\calB_{1,+}}^\top{\rm diag}
  \left[
  \int_0^1 
  \ddot{\ell}_{/i}(\btheta(t))dt\right]
  \bX_{/i,\calF\setminus\calB_{1,+}}.
	\end{eqnarray*}
Similarly, we define $\bA_2, \bB_2, \bC_2$ in the following way:
\begin{equation}\label{eq:defA2B2C2}
		\bH_2:=\lambda\,{\rm diag}(\ddot{\utilde{r}}_{\calF}^{\alpha/i})
		+\bX^\top_{\calF}{\rm diag}
		(\ddot{\utilde{\ell}}^{\alpha/i})\bX_{\calF}
		=\begin{pmatrix}
			\bA_2 &\bB_2\\
			\bB_2^\top &\bC_2
		\end{pmatrix}
	\end{equation}
	where 
	\begin{eqnarray*}
		\bA_2  &=&~{\rm diag}\left[
        \lambda
        \left\{
        \ddot{r}_{\alpha}(\hat{\bbeta}^{\alpha})\right\}_{\calB_{1,+}}\right] \nonumber \\
		&&+\bX_{/i,\calB_{1,+}}^\top
		{\rm diag}\left[
        \ddot{\ell}_{/i}(\hat{\bbeta}^{\alpha})
        \right]\bX_{/i,\calB_{1,+}}\\
		\bB_2&=&~\bX_{/i,\calB_{1,+}}^\top
        {\rm diag}\left[
        \ddot{\ell}_{/i}(\hat{\bbeta}^{\alpha})
        \right]
  \bX_{/i,\calF\setminus\calB_{1,+}}\\
		\bC_2&=&~{\rm diag}\left[
        \lambda
        \left\{
        \ddot{r}_{\alpha}(\hat{\bbeta}^{\alpha})
        \right\}
        _{\calF\setminus\calB_{1,+}} \right] 
        + \bX_{/i,\calF\setminus\calB_{1,+}}^\top{\rm diag}
  \left[
  \ddot{\ell}_{/i}(\hat{\bbeta}^{\alpha})\right]
  \bX_{/i,\calF}.
	\end{eqnarray*}
We then obtain
\begin{align*}
&\bigg\vert\bx_{i,\calF}^\top 
			(\lambda{\rm diag}(\ddot{\underline{r}}_{\calF}^{\alpha/i})
			+\bX^\top_{\calF}{\rm diag}
			(\ddot{\underline{\ell}}^{\alpha/i})\bX_{\calF})^{-1}\bx_{i,\calF}\\
		&~-
		\bx_{i,\calF}^\top 
		(\lambda{\rm diag}(\ddot{\utilde{r}}_{\calF}^{\alpha/i})
		+\bX^\top_{\calF}{\rm diag}
		(\ddot{\utilde{\ell}}^{\alpha/i})\bX_{\calF})^{-1}\bx_{i,\calF}
		\bigg\vert 	\\
 =&~ \bx_{i,\calF}^{\top}
    \left[
    \begin{pmatrix}
			\bA_1 &\bB_1\\
			\bB_1^\top &\bC_1
		\end{pmatrix}^{-1}
    -
    \begin{pmatrix}
			\bA_2 &\bB_2\\
			\bB_2^\top &\bC_2
		\end{pmatrix}^{-1}
    \right]
  \bx_{i,\calF}.
  \end{align*}
By matrix inversion of block diagonal matrices, Lemma \ref{lem:BMIL}, we have
    	\begin{align}\label{eq:BDinv}
	\bH_k^{-1}
		= \begin{pmatrix}
			\bA_k^{-1}+\bA_k^{-1}\bB_k\bD\bB_k^\top\bA_k^{-1} & -\bA_k^{-1}\bB_k\bD_k \\
			-\bD_k \bB_k^\top\bA_k^{-1} &\bD_k
		\end{pmatrix}
	\end{align}
 where for $k=1,2$ we define
    \begin{align*}
    \bD_k:=&~(\bC_k-\bB_k^{\top}\bA_k^{-1}\bB_k)^{-1}. 
    \end{align*}
From \eqref{eq:BDinv} we have 
\begin{align}\label{eq:def-psi123def}
&\bigg\vert\bx_{i,\calF}^\top 
			(\lambda{\rm diag}(\ddot{\underline{r}}_{\calF}^{\alpha/i})
			+\bX^\top_{\calF}{\rm diag}
			(\ddot{\underline{\ell}}^{\alpha/i})\bX_{\calF})^{-1}\bx_{i,\calF}\nonumber \\
		&~-
		\bx_{i,\calF}^\top 
		(\lambda{\rm diag}(\ddot{\utilde{r}}_{\calF}^{\alpha/i})
		+\bX^\top_{\calF}{\rm diag}
		(\ddot{\utilde{\ell}}^{\alpha/i})\bX_{\calF})^{-1}\bx_{i,\calF}
		\bigg\vert \nonumber 	\\
 =&~ \bx_{i,\calF}^{\top}
    \left[
    \begin{pmatrix}
			\bA_1 &\bB_1\\
			\bB_1^\top &\bC_1
		\end{pmatrix}^{-1}
    -
    \begin{pmatrix}
			\bA_2 &\bB_2\\
			\bB_2^\top &\bC_2
		\end{pmatrix}^{-1}
    \right]
  \bx_{i,\calF} \nonumber \\
  \le &~  |\psi_{01} - \psi_{02}| +|\psi_{11} - \psi_{12}|+|\psi_{21}- \psi_{22}|+
  2|\psi_{31}-\psi_{32}|, 
  \end{align}
where
   \begin{align*}
        \psi_{0k}:=&~
        \bx_{i,\calB_{1,+}}^{\top} \bA_k^{-1}
  \bx_{i,\calB_{1,+}},
    \\
        \psi_{1k}:=&~\bx_{i, \calB_{1,+}}^{\top}
  \bA_k^{-1}\bB_k\bD_k\bB_k^{\top}\bA_k^{-1}
  \bx_{i,\calB_{1,+}},
    \\
    \psi_{2k}:=&~
    \bx_{i,\calF\setminus\calB_{1,+}}^{\top}
    \bD_k
    \bx_{i,\calF\setminus\calB_{1,+}},
    \\
    \psi_{3k}:=&~
    \bx_{i,\calB_{1,+}}^{\top}
    \bA_k^{-1}\bB_k\bD_k
    \bx_{i,\calF\setminus\calB_{1,+}}.
    \end{align*}
The last inequality of \eqref{eq:def-psi123def} is the result of the triangle inequality. Our goal is to prove that all the terms in \eqref{eq:def-psi123def} are small with high probability for large values of $n,p$. Towards this goal, we will first prove that $\psi_{11}$, $\psi_{12}$, $\psi_{21}$, $\psi_{22}$, $\psi_{31}$, and $\psi_{32}$ are all individually small. Then, we finally show that the difference $\psi_{01}-\psi_{02}$ is small too. \\
\begin{itemize}

\item \textbf{Finding upper bounds for $\psi_{21}$ and $\psi_{22}$:} \\
Note that
\begin{equation}\label{eq:lmaxD}
\sigma_{\max}(\bD_k) \leq \sigma_{\max} (\bH_k^{-1}) = \frac{1}{\sigma_{\min} (\bH_k) }.
\end{equation}
It is then straightforward to see that since $\bH_1$ and $\bH_2$ are summations of $\lambda{\rm diag}(\ddot{\underline{r}}_{\calB_0^c}^{\alpha/i})$ and $\lambda{\rm diag}(\ddot{\utilde{r}}_{\calB_0^c}^{\alpha/i})$ with a pair of positive semidefinite matrices, and that $r$ has a ridge component in it, we have
\[
\sigma_{\min} (\bH) \geq 2 \lambda \eta. 
\]
Using this fact and \eqref{eq:lmaxD} we obtain
\begin{equation}\label{eq:Dbound}
\|\bD_k\|\le(2\lambda\eta)^{-1}\,
\text{ for }
k=1,2.
\end{equation}
It then follows that 
\begin{align*}\label{eq:psi2-bd}
    \psi_{2k}
    \le&~ \|\bD_k\|\times \|\bx_{i,\calF\setminus\calB_{1,+}}\|^2\\
    \overset{(a)}{\le}&~  
    \dfrac{\max_i \sup_{\calT: |\calT| = d_n} \|\bx_{i,\calT}\|^2}{2\lambda\eta}
    .
    \numberthis
\end{align*}
Note that the reason we have used the maximum on the set $\calT$ in bounding $\|\bx_{i,\calF\setminus\calB_{1,+}}\|^2$  is that $\calF / \calB_{1,i}$ depends on $\bx_i$ and hence we cannot use standard concentration results for $\chi^2$ random variables (e.g., Lemma \ref{lem:chi:sq:ind}). Furthermore, when taking the supremum over sets $\calT$, we have to consider all the sets $\calT$ whose sizes are smaller than $d_n$. But as is clear, in \eqref{eq:psi2-bd} we have only considered $\calT$ with sizes equal to $d_n$. This is because the norm of $\|\bx_{i,\calT'}\|^2 \leq \|\bx_{i,\calT}\|^2$ if $\calT' \subset \calT$. We have
\begin{eqnarray}
\lefteqn{\PP (\max_i \sup_{\calT: |\calT| = d_n} \|\bx_{i,\calT}\|^2 > d_n \rho_{\max}(1+t))} \nonumber \\
&\leq& \sum_i \sum _{\calT: |\calT| = d_n}  {\PP} (\|\bx_{i,\calT}\|^2 > d_n\rho_{\max}(1+t))  \nonumber \\
&\leq&n {p \choose d_n} {\rm e}^{-\frac{d_n}{2} (t- \log (1+t)) } \nonumber \\
&\leq& n {\rm e}^{d_n \log \Big( \frac{ep}{d_n} \Big)}{\rm e}^{-\frac{d_n}{2} (t- \log (1+t)) },
\end{eqnarray}
where to obtain the last two inequalities we have used Lemma \ref{lem:chi:sq:ind} and Lemma~\ref{lem:stirling}. Setting $t = 8\log p$ in this equation, we conclude that 
\begin{eqnarray}\label{eq:conc:final:chi:union}
\lefteqn{\PP (\max_i \sup_{\calT: |\calT| = d_n} \|\bx_{i,\calT}\|^2 > d_n\rho_{max}(1+8\log p))} \nonumber \\
&\leq&n {\rm e}^{d_n \log \Big( \frac{ep}{d_n} \Big)}{\rm e}^{-\frac{d_n}{2} (8\log p - \log (1+8\log p)) }.
\end{eqnarray}
Combining \eqref{eq:psi2-bd} and \eqref{eq:conc:final:chi:union} we conclude that 
\begin{eqnarray}
\lefteqn{\PP\left(\psi_{2k} > \frac{\rho_{\max} d_n (1
+ 8\log p) }{2 \eta \lambda}  \right)} \nonumber \\
&\leq& n {\rm e}^{d_n \log \Big( \frac{ep}{d_n} \Big)}{\rm e}^{-\frac{d_n}{2} (8\log p - \log (1+8\log p)) }
\nonumber
\\
&\le &
{\rm e}^{
\log(n)+d_n\log\Big( \frac{ep}{d_n} \Big)
-\frac{d_n}{2} (8\log p - \log (1+8\log p))
}
\nonumber
\\
&\le& np^{-d_n},
\end{eqnarray}
provided $d_n\ge {\rm e}\sqrt{1+8\log p}$.\\

\item \textbf{Finding upper bounds for $\psi_{11}$ and $\psi_{12}$:} For better readability, we defer the proof of this bound to Lemma~\ref{lem:psi1-bds}:

\begin{lemma}\label{lem:psi1-bds} For $k=1,2$ we have a sufficienty large constant $C>0$ depending only on $\gamma_0$ and $C_X$ such that 
\begin{align*}
   &~\psi_{1k}
   :=\bx_{i,\calB_{1,+}}^{\top}
    \bA_2^{-1}\bB_k\bB_k^{\top}\bA_k^{-1}
    \bx_{i,\calB_{1,+}}\\
    \le&~
    \dfrac{\polylog(n)}{\lambda^3\eta^3(1\wedge\lambda\eta)^3}
    \sqrt{\dfrac{d_n}{n\lambda\eta}}
    +\dfrac{Cd_n}{n\lambda^2\eta^2}
    +\dfrac{d_n\log^2p}{p-d_n}
    +
    \dfrac{C}{p\lambda\eta}
\end{align*}
with probability at least $1-
(n+1){\rm e}^{-\frac{p}{2}}
-
{\rm e}^{-d_n\log p}-q_n-\check{q}_n-\bar{q}_n-\tilde{q}_n$, for sufficiently large $p$.
\end{lemma}

\medskip

\item \textbf{Finding upper bounds for $\psi_{31}$ and $\psi_{32}$:} \\
We have 
\begin{align*}\label{eq:psi3-bd}
    \psi_{3k}
    \le&~\sqrt{\psi_{2k}\times \bx_{i,\calB_{1,+}}^{\top}
\bA_k^{-1}\bB_k\bD_k\bB_k^{\top}\bA_k^{-1}
 \bx_{i,\calB_{1,+}}} \nonumber \\
 \le& \sqrt{\psi_{2k} \psi_{1k}}
 \le (\psi_{1k}+\psi_{2k})/2.\numberthis
\end{align*}

\item \textbf{Finding an upper bound for $|\psi_{01}- \psi_{02}|$:}
\end{itemize}

$$
    \bx_{i,\calB_{1,+}}^{\top}(\bA_1^{-1}-\bA_2^{-1})\bx_{i,\calB_{1,+}}.
$$
The proof of this part is similar to our proof technique for bounding $\psi_{11}$ and $\psi_{12}$, with a few important differences. As before, there are two sources of dependency between $\bA_1, \bA_2$ and $\bx_i$: (i) The dependency of the input arguments of $\ddot{\ell}$ and $\ddot{r}$ on $\bx_i$. (ii) the dependency between the set $\calB_{1,+}$ and $\bx_i$. The goal is to remove these dependencies. We start with removing the dependency of the input argument of $r$ on $\bx_i$. Define
\begin{eqnarray}  
\starbA_2 &:=& ~{\rm diag}\left[
        \lambda
        \left\{
        \ddot{r}_{\alpha}(\hat{\bbeta}^{\alpha}_{/ i})\right\}_{\calB_{1,+}}\right] \nonumber \\
		&+&\bX_{/i,\calB_{1,+}}^\top
		{\rm diag}\left[
        \ddot{\ell}_{/i}(\hat{\bbeta}^{\alpha})
        \right]\bX_{/i,\calB_{1,+}}
\end{eqnarray}
Define 
\begin{eqnarray}
\starbDelta_1 &=& \starbA^{-1}_1- \bA_1^{-1}, \nonumber \\
\starbDelta_2 &=&\starbA^{-1}_2- \bA_2^{-1}.
\end{eqnarray}
The goal is to obtain a bound on $\norm{\starbDelta_1}$ and $\norm{\starbDelta_2}$. Since the proofs are similar and for notational simplicity we show our claim for $\starbDelta_2$. Since $\starbDelta_2 =\bA_2^{-1} (\bA_2- \starbA_2)\starbA_2^{-1}$, and $\bA_2- \starbA_2 = {\rm diag}\left[\lambda
        \left\{
        \ddot{r}(\hat{\bbeta}^{\alpha}_{/ i})\right\}_{\calB_{1,+}}\right] - {\rm diag}\left[ \lambda
        \left\{
        \ddot{r}(\hat{\bbeta}^{\alpha})\right\}_{\calB_{1,+}}\right]$, we have
\begin{eqnarray}\label{eq:boundDeltastar}
\norm{\starbDelta_2} &\leq& \frac{\norm{\bA_2- \starbA_2}}{\sigma_{\min} (\starbA_2)\sigma_{\min}(\bA_2)} \leq \frac{2 \lambda \alpha (1- \eta) {\rm e}^{-\frac{1}{2} \alpha \kappa_1}}{ (2\lambda \eta)^2} \nonumber \\ 
&=&\frac{ \alpha (1- \eta) {\rm e}^{-\frac{1}{2} \alpha \kappa_1}}{ 2\lambda \eta^2}. 
\end{eqnarray}
To obtain \eqref{eq:boundDeltastar}, we used Part (4) and (5) of Lemma \ref{lem:rdderiv} to find a bound on $\ddot{r}_{\alpha}(\hat{\bbeta}^{\alpha}_{/ i, k})- \ddot{r}_{\alpha}(\hat{\bbeta}^{\alpha}_k)$. Furthermore, given the ridge part of the regularizer, the minimum eigenvalues of $\bA_2$ and $\starbA_2$ are $2 \lambda \eta$.   It is straight forward to see that
\begin{eqnarray}\label{eq:psi_0k:1}
\lefteqn{\Big|\bx_{i, \calB_{1,+}}^{\top} \Big(\bA_2^{-1} - \starbA_2^{-1}\Big)\bx_{i, \calB_{1,+}}\Big| \leq \norm{\bx_i}^2 \norm{\starbDelta_2}} \nonumber \\
&\leq& \frac{ \|\bx_i\|_2^2 \alpha (1- \eta) {\rm e}^{-\frac{1}{2} \alpha \kappa_1}}{ 2\lambda \eta^2}. \hspace{3cm}
\end{eqnarray}
It is clear that as $\alpha \rightarrow \infty$, the upper bound in \eqref{eq:psi_0k:1} goes to zero. By replacing $\bA_2^{-1}, \bA_1^{-1} - \starbA_2^{-1}$ with $\starbA_2^{-1}, \starbA_1^{-1}$ we have removed the dependency between the input argument of $\ddot{r}$ and $\bx_i$. As before to remove the dependency of the set $\calB_{1, +}$ we lift the problem to a higher dimensional space. Towards this goal, we define
\begin{align*}
		\tilde{\bA}_2  &=~ {\rm diag}\left[
        \lambda
        \left\{
        \ddot{r}_{\alpha}(\hat{\bbeta}^{\alpha}_{/i})\right\}_{\calB^+}\right] 
		+\bX_{/i,\calB^+}^\top
		{\rm diag}\left[
        \ddot{\ell}_{/i}(\hat{\bbeta}^{\alpha})
        \right]\bX_{/i,\calB^+},
	\end{align*}
where $\calB^+ := \calB_{1,+} \cup \calB_0$. We write $\tilde{\bA}_2$ as 
\begin{align*}
    \tilde{\bA}_2=
    \begin{pmatrix}
        \starbA_2 & \bB\\
        \bB^\top & \bC
    \end{pmatrix}
\end{align*}
where
\begin{align*}
    &\bB = \bX_{\cB_{1,+}}^\top\diag\{\underline{\ddot{\ell}}^{\alpha/i}\}\bX_{\cB_{1,+}},\\
    &\bC = \lambda \diag \left[ \{ \underline{\ddot{r}}_{\alpha} \}_{\cB_0} \right] + \bX_{\cB_0}^\top\diag\{\underline{\ddot{\ell}}^{\alpha/i}\}\bX_{\cB_0}.
\end{align*}
Using the block matrix inversion lemma, i.e. Lemma \ref{lem:BMIL},  we have
\begin{align*}
    \tilde{\bA}_2^{-1} = 
    \begin{pmatrix}
        \starbA_2^{-1} + \starbA_2^{-1}\bB\bH\bB^\top\starbA_2^{-1} & -\starbA_2^{-1}\bB\bH\\
        -\bH\bB^\top\starbA^{-1} & \bH
    \end{pmatrix}
\end{align*}
where $\bH = (\bC - \bB^\top \starbA_2^{-1}\bB)^{-1}$. Then we have
\begin{align*}
    &\abs{\bx_{i, \calB_{1,+}}^{\top} \starbA_2^{-1} \bx_{i, \calB_{1,+}}- \bx_{i, \calB^+}^{\top} \tilde{\bA}_2^{-1} \bx_{i, \calB^+}}\\
    &\leq \abs{\bx_{i, \calB_{1,+}}^{\top} \starbA_2^{-1}\bB\bH\bB^\top\starbA_2^{-1}\bx_{i, \calB_{1,+}}}
    +2\abs{\bx_{i, \calB_{1,+}}^{\top} \starbA_2^{-1}\bB\bH\bx_{i, \cB_0}} + \abs{\bx_{i, \cB_0}^\top\bH\bx_{i, \cB_0}}\\
    &\leq \norm{\bx_i}^2\norm{\starbA_2^{-1}}^2\norm{\bB}^2\norm{\bH}
    + 2\norm{\bx_i}^2\norm{\starbA_2^{-1}}\norm{\bB}\norm{\bH} + \norm{\bx_i}^2\norm{\bH}\\
    &\leq \norm{\bx_i}^2\norm{\bH} \left( \norm{\starbA_2^{-1}}\norm{\bB}+1 \right)^2
\end{align*}

To bound the matrix norms in the above equation, we have 
\begin{itemize}
    \item $\norm{\starbA_2^{-1}}$:
    \begin{align*}
        \sigma_{\min}(\starbA) \geq 2 \lambda\eta 
        \quad 
        \text{and hence}
        \quad
        \norm{\starbA_2^{-1}}\leq  \frac{1}{2\lambda\eta}.
    \end{align*}
    \item $\norm{\bB}$:
    It follows by definition of $\omega^s$ in \eqref{eq:thm3-not}  that
    $$\norm{\bB}\leq \omega^s.$$
    \item $\norm{\bH}$:
    Using a derivation similar to \eqref{eq:lowerbd:D}, we have
    \begin{eqnarray*}
       \sigma_{\min}(\bC - \bB^\top \starbA_2^{-1}\bB)
        \geq \dfrac{\lambda\alpha}{16}(1-\eta)\kappa_0.
    \end{eqnarray*}
   
\end{itemize}
Inserting the above bounds for matrix norms, we have
\begin{align}\label{eq:psi_0k:2}
    &\abs{\bx_{i, \calB_{1,+}}^{\top} \starbA_2^{-1} \bx_{i, \calB_{1,+}}- \bx_{i, \calB^+}^{\top} \tilde{\bA}_2^{-1} \bx_{i, \calB^+}}\nonumber \\
    &\leq \norm{\bx_i}^2\norm{\bH} \left( \norm{\starbA_2^{-1}}\norm{\bB}+1 \right)^2\nonumber \\
    &\leq \frac{16\|\bx_i\|^2}{\lambda\alpha(1-\eta)} 
    \left(1 + \frac{\omega^s}{2 \lambda \eta}\right)^2,
\end{align}
with probability larger than $1-q_n$. Again it is straightforward to check that as $\alpha \rightarrow \infty$, the upper bound in \eqref{eq:psi_0k:2} goes to zero. Hence, we can now focus on the term $\bx_{i, \calB^+}^{\top} \tilde{\bA}_2^{-1} \bx_{i, \calB^+}$. Note that in matrix $\tilde{\bA}_2^{-1}$, there are $\ddot{\ell} (\hat{\bbeta}^{\alpha})$. In the next step, we would like to show that the difference between this term and $\ddot{\ell} (\hat{\bbeta}_{/i}^{\alpha})$ is negligible for large values of $n,p$ and any $\alpha$. Note that once $\ddot{\ell} (\hat{\bbeta}_{/i}^{\alpha})$ the dependence between $\tilde{\bA}_2^{-1}$ and $\bx_i$ reduces to only the dependence of the two terms on $\calB_{1,+}$.  Define
\begin{align}
		\check{\bA}_2  &:={\rm diag}\left[
        \lambda
        \left\{
        {\ddot{r}}_{\alpha}(\hat{\bbeta}_{/i}^{\alpha})\right\}\right] +\bX_{/i,\calB^+}^\top
		{\rm diag}\left[
        \ddot{\ell}_{/i}(\hat{\bbeta}^{\alpha}_{/i})
        \right]\bX_{/i,\calB^+},
\end{align}
The goal is to show that 
\[
  \left|\bx_{i, \calB^{+}}^{\top} \left(\check{\bA}_2^{-1}  -\tilde{\bA}_2^{-1} \right)\bx_{i, \calB^+}\right|
\]
is small. Towards this goal, define 
\[
\bDelta_\ell := {\rm diag}\left[
        \ddot{\ell}_{/i}(\hat{\bbeta}^{\alpha})
        \right]- {\rm diag}\left[
        \ddot{\ell}_{/i}(\hat{\bbeta}_{/i}^{\alpha})
        \right] 
\]
We have
\begin{eqnarray*}
    \tilde{\bA}_2 = \check{\bA}_2 + \bX_{/i,\calB^+}^\top
		\bDelta_\ell \bX_{/i,\calB^+}
\end{eqnarray*}
According to Lemma \ref{lem:woodtwice}
\begin{align*}
    &~\tilde{\bA}_2^{-1}\\
    =&~\check{\bA}_2^{-1} - \check{\bA}_2^{-1} \bX_{/i,\calB^+}^\top \bDelta_\ell \bX_{/i,\calB^+}\check{\bA}_2^{-1}\\
    &~+ \check{\bA}_2^{-1} \bX_{/i,\calB^+}^\top \bDelta_\ell \bX_{/i,\calB^+} \tilde{\bA}_2^{-1} \bX_{/i,\calB^+}^\top \bDelta_\ell \bX_{/i,\calB^+}\check{\bA}_2^{-1}.
\end{align*}
Hence, if we define
\begin{eqnarray}
\check{\bv} &=& \bX_{/i,\calB^+}\check{\bA}_2^{-1} \bx_{i, \calB^+} \nonumber \\
\check{\bu}^{\top} &= &  \bx_{i, \calB^+} \check{\bA}_2^{-1} \bX_{/i,\calB^+}^\top \bDelta_\ell \bX_{/i,\calB^+} \tilde{\bA}_2^{-1} \bX_{/i,\calB^+}^\top ,
\end{eqnarray}
then we have
\begin{align}\label{eq:AcheckAtilde1}
&|\bx_{i, \calB^{+}}^{\top} (\check{\bA}_2^{-1}  -\tilde{\bA}_2^{-1} )\bx_{i, \calB^+}|  \leq |\check{\bv}^\top \bDelta_{\ell} \check{\bv}|+ |\check{\bu}^{\top} \bDelta_{\ell} \check{\bv}| \nonumber \\
 &{\leq} \sqrt{\sum_i \bDelta_{\ell, ii}^2} \sqrt{\sum_i \check{\bu}_i^2 \check{\bv}_i^2} + \sqrt{\sum_i \Delta_{\ell, ii}^2} \sqrt{\sum_i \check{\bv}_i^4} \nonumber \\
 &\overset{(c)}{\leq} \sqrt{\sum_i \bDelta_{\ell, ii}^2} (\sum_i \check{\bu}_i^4)^{\frac{1}{4}} (\sum_i  \check{\bv}_i^4)^{\frac{1}{4}} + \sqrt{\sum_i \Delta_{\ell, ii}^2} (\sum_i  \check{\bv}_i^4)^{\frac{1}{2}} \nonumber \\
 &\leq \sqrt{\sum_i \bDelta_{\ell, ii}^2} (\sum_i \check{\bu}_i^2)^{\frac{1}{2}} (\max_{i} |\check{\bv}_i|)^{\frac{1}{2}} (\sum_i  \check{\bv}_i^2)^{\frac{1}{4}} \nonumber \\
 &+ \sqrt{\sum_i \bDelta_{\ell, ii}^2}  (\max_{i} |\check{\bv}_i|) (\sum_i  \check{\bv}_i^2)^{\frac{1}{2}}.
\end{align}
By our assumptions (see \eqref{eq:DeltaEllBound} for a more detailed calculation), we have
\begin{align}\label{eq:AcheckAtilde2}
\sqrt{\sum_j \Delta_{\ell, ii}^2} 
~{\leq}~    \frac{\polylog(n)\|\bx_i\|_2}{2 \lambda \eta},
\end{align}
with probability larger than $1-\tilde{q}_n- \check{q}_n$. Furthermore,
\begin{align*}\label{eq:AcheckAtilde3}
   \|\check{\bv}\|_2 
   =&~
   \|
 \bX_{/i,\calB^+} \check{\bA}^{-1} \bx_{i,\calB^+}\| \\
 \le &~
 \frac{\|\bX\| \|\bx_i\|_2}{\sigma_{\min} (\check{\bA})} \leq \frac{\|\bX\| \|\bx_i\|_2}{ 2\lambda \eta}\numberthis  
\end{align*}
and 
\begin{align}\label{eq:AcheckAtilde4}
\|\check{\bu}\|\leq \frac{\|\bx_i\|_2^2 \|\bX\|^2\|\bDelta_\ell\|}{\sigma^2_{\min} (\check{\bA})} \leq 
\frac{\polylog(n)
\|\bx_i\|_2^3 \|\bX\|^2}{8 \lambda^3 \eta^3},
\end{align}
with probability larger than $1- \check{q}_n$.
Finally, 
\begin{align*}\label{eq:AcheckAtilde5}
   \PP\left( \max_i \check{\bv}_i >    \gamma_5(n) + t \right) 
  \le&~
   {\rm e}^{ 2d_n \log \frac{e p}{d_n}}  \left(
{\rm e}^{-p} + 2 {\rm e}^{-  \frac{n \lambda \eta t^2}{(\sqrt{n}+ 3\sqrt{p})^2 \rho_{\max} }}\right).\numberthis
\end{align*}

For the definition of $\gamma_5(n)$ and the detailed derivation of \eqref{eq:AcheckAtilde5}, please check \eqref{eq:def:gamma5} and \eqref{eq:uppboundABdelBA4}. Hence, if we set $t = \sqrt{\frac{d_n \log^2 p}{\lambda\eta n}}$, and define
\[
\gamma_7(n) = 
\gamma_5(n)
+
\sqrt{\frac{d_n \log^2 p}{\lambda\eta n}}
\]
by combining \eqref{eq:AcheckAtilde1}, \eqref{eq:AcheckAtilde2}, \eqref{eq:AcheckAtilde3}, \eqref{eq:AcheckAtilde4}, \eqref{eq:AcheckAtilde5}, we obtain
\begin{align}\label{eq:psi_0k:3}
&\PP(|\bx_{i, \calB^{+}}^{\top} (\check{\bA}_2^{-1}  -\tilde{\bA}_2^{-1} )\bx_{i, \calB^+}| > \gamma_7(n)) \nonumber \\
\leq&~ {\rm e}^{ 2d_n \log \frac{e p}{d_n}}  \left(
{\rm e}^{-p} + 2 {\rm e}^{-  \frac{d_n\log^2p}{(\sqrt{n}+ 3\sqrt{p})^2 \rho_{\max} }}\right) + q_n + \check{q}_n\nonumber\\
\le&~ {\rm e}^{-d_n\log p}+q_n+\check{q}_n
\end{align}
provided $d_n\le \dfrac{p}{C}$ for a sufficiently large constant $C>0$, where we use the fact that $\rho_{\max}=O(p^{-1})$.

Note that we were originally interested in bounding $|\bx_{1,\calB_{1,+}}^{\top}(\bA_1^{-1}-\bA_2^{-1})\bx_{1,\calB_{1,+}}|$. We have
\begin{eqnarray}
    &&|\bx_{1,\calB_{1,+}}^{\top}(\bA_1^{-1}-\bA_2^{-1})\bx_{1,\calB_{1,+}}| \nonumber \\
    &\leq&  |\bx_{1,\calB_{1,+}}^{\top}({\bA}_1^{-1}-{\starbA}_1^{-1})\bx_{1,\calB_{1,+}}| \nonumber \\
    &&+ |\bx_{1,\calB_{1,+}}^{\top}({\bA}_2^{-1}-{\starbA}_2^{-1})\bx_{1,\calB_{1,+}}| \nonumber \\
    &&+ |\bx_{1,\calB_{1,+}}^{\top}{\starbA}_1^{-1}\bx_{1,\calB_{1,+}} - \bx_{1, \calB^+}^{\top} \tilde{\bA}_1^{-1}\bx_{1,\calB_{1,+}}| \nonumber \\
    &&+ |\bx_{1,\calB_{1,+}}^{\top}{\starbA}_2^{-1}\bx_{1,\calB_{1,+}} - \bx_{1, \calB^+}^{\top} \tilde{\bA}_2^{-1}\bx_{1,\calB_{1,+}}| \nonumber \\
    &&+ |\bx_{1, \calB^+}^{\top} \left(\tilde{\bA}_2^{-1} - \check{\bA}^{-1}_2\right)\bx_{1,\calB^+}| \nonumber \\
    &&+ |\bx_{1, \calB^+}^{\top} \left(\tilde{\bA}_1^{-1} - \check{\bA}^{-1}_1\right)\bx_{1,\calB^+}|.
\end{eqnarray}
Hence, by combining \eqref{eq:psi_0k:1}, \eqref{eq:psi_0k:2}, \eqref{eq:psi_0k:3} we obtain that if
\begin{align}
\gamma_8(n, \alpha) 
:=&~ 2\gamma_7(n) +2 \frac{ \|\bx_i\|_2^2 \alpha (1- \eta) {\rm e}^{-\frac{1}{2} \alpha \kappa_1}}{ 2\lambda \eta^2} \nonumber \\ &+ 2\frac{\|\bx_i\|^2}{\gamma_1^s(\alpha)} 
\left(1 + \frac{\omega^s}{2 \lambda \eta}\right)^2,
\end{align}
then
\begin{align}
&{\PP(|\bx_{1,\calB_{1,+}}^{\top}(\bA_1^{-1}-\bA_2^{-1})\bx_{1,\calB_{1,+}}| > \gamma_8(n, \alpha) )} \nonumber \\ 
\leq&~ {\rm e}^{ 2d_n \log \frac{e p}{d_n}}  \left(
{\rm e}^{-p} + 2 {\rm e}^{-  \frac{d_n\log^2p}{(\sqrt{n}+ 3\sqrt{p})^2 \rho_{\max} }}\right) + q_n + \check{q}_n + \tilde{q}_n\nonumber\\
\le&~
{\rm e}^{-d_n\log p}+q_n+\check{q}_n+\tilde{q}_n,
\end{align}
provided $d_n\le \dfrac{p}{C}$ for a sufficiently large constant $C>0$, where we use the fact that $\rho_{\max}=O(p^{-1})$.

Equipped with the bounds on $|\psi_{01}-\psi_{02}|$, $\psi_{1k}$, $\psi_{2k}$ and $\psi_{3k}$, we now return to equation~\eqref{eq:def-psi123def} to obtain
\begin{align*}
    &\bigg\vert\bx_{i,\calF}^\top 
			(\lambda{\rm diag}(\ddot{\underline{r}}_{\calF}^{\alpha/i})
			+\bX^\top_{\calF}{\rm diag}
			(\ddot{\underline{\ell}}^{\alpha/i})\bX_{\calF})^{-1}\bx_{i,\calF}
   -
		\bx_{i,\calF}^\top 
		(\lambda{\rm diag}(\ddot{\utilde{r}}_{\calF}^{\alpha/i})
		+\bX^\top_{\calF}{\rm diag}
		(\ddot{\utilde{\ell}}^{\alpha/i})\bX_{\calF})^{-1}\bx_{i,\calF}
		\bigg\vert \nonumber 	\\
 =&~ \bx_{i,\calF}^{\top}
    \left[
    \begin{pmatrix}
			\bA_1 &\bB_1\\
			\bB_1^\top &\bC_1
		\end{pmatrix}^{-1}
    -
    \begin{pmatrix}
			\bA_2 &\bB_2\\
			\bB_2^\top &\bC_2
		\end{pmatrix}^{-1}
    \right]
  \bx_{i,\calF} \nonumber \\
  \le &~  |\psi_{01} - \psi_{02}| +|\psi_{11} - \psi_{12}|+|\psi_{21}- \psi_{22}|+
  2|\psi_{31}-\psi_{32}|\\
  \le &~
  \gamma_8(n,\alpha)
  +\dfrac{2\rho_{\max}d_n(1+8\log p)}{\eta\lambda}
  +2(\psi_{11}+\psi_{12})\\
  \le&~
  \dfrac{2\rho_{\max}d_n(1+8\log p)}{\eta\lambda}
  +2(\psi_{11}+\psi_{12})\\
  &~+2\sqrt{\dfrac{(\sqrt{n}+3\sqrt{p})^2\rho_{\max}}{n\lambda\eta}\log 2n}\\
  &~+2\sqrt{\dfrac{d_n\log^2p}{n\lambda\eta }}
  +\frac{16\|\bx_i\|^2}{\lambda\alpha(1-\eta)} \left(
  1+\dfrac{\omega^s}{2\lambda\eta}
  \right)^2+\dfrac{\alpha{\rm e}^{-\frac{1}{2}\alpha\kappa_1}}{\lambda\eta^2}\\
  \le&~
  \dfrac{\polylog(n)}{\lambda^3\eta^3(1\wedge\lambda\eta)^3}
    \sqrt{\dfrac{d_n}{n\lambda\eta}}
    +\dfrac{Cd_n\log^2p}{p\lambda\eta} +\dfrac{Cd_n}{n\lambda^2\eta^2}
    +
    \sqrt{\dfrac{C\log p}{p\lambda\eta}}
\end{align*}
with probability at least $1-(n+1){\rm e}^{-\frac{p}{2}}-(n+2)p^{-d_n}-2q_n-2\check{q}_n-2\tilde{q}_n$. This finishes the proof of Theorem~\ref{thm:FINAL_STEP}.\qed

\bigskip

\subsection{Proof of Lemma~\ref{lem:psi1-bds}}

First note that by using \eqref{eq:Dbound} we obtain
\begin{align*}
\bx_{i,\calB_{1,+}}^{\top}
\bA_k^{-1}\bB_k\bD_k\bB_k^{\top}\bA_k^{-1}
 \bx_{i,\calB_{1,+}}
\le~
(2\lambda\eta)^{-1}
\bx_{i,\calB_{1,+}}^{\top}
\bA_k^{-1}
\bB_k\bB_k^{\top}
\bA_k^{-1}\bx_{i,\calB_{1,+}}.
\end{align*}
We will use the notation
\begin{align*}\label{eq:thm3-not-2}
\omega_s:= &~
\|\bX^{\top}\bX\|
\sup_{t\in[0,1]}
\max_{
1\le i\le n}
\ddot{\ell}_i(
t\hat{\bbeta}^{\alpha}
+(1-t)\hat{\bbeta}_{/i}^{\alpha}
\,;\\
\rho(\alpha)
:=&~\lambda\left(
2\eta+\dfrac{\alpha(1-\eta)\kappa_0}{8}
\right).\numberthis
\end{align*}
In the rest of the proof, we focus on $k=2$ and find an upper bound for $|\psi_{2k}|$. The proof for $k=1$ is similar and will hence be skipped. Furthermore, for notational simplicity, we drop the subscripts of $k$ from matrices $\bA_k$ and $\bB_k$. 

Intuitively speaking, the matrix $\bG = \bA^{-1} \bB \bB^{\top} \bA^{-1}$ has rank $d_n$. Hence, if $\bG$ were independent of $\bx_i$, we could have used concentration of $\chi^2$ random variables to show that $\psi_{1k} = O_p(d_n \rho_{\max})$. However, there are multiple sources of dependency between $\bx_i$ and $\bG$: (1) The input arguments of $\ddot{r}$, (2) the input arguments of $\ddot{\ell}$, and (3) The dependence of $\calB_{1,+}$ and $\calF$ on $\bx_i$. In the rest of the proof, we aim to handle these three dependencies and show that $\psi_{1k}$ is still small. We remind the reader that
	\begin{align*}
		\bA  =&~{\rm diag}\left[
        \lambda
        \left\{
        \ddot{r}_{\alpha}(\hat{\bbeta}^{\alpha})\right\}_{\calB_{1,+}}\right] ~+\bX_{/i,\calB_{1,+}}^\top
		{\rm diag}\left[
        \ddot{\ell}_{/i}(\hat{\bbeta}^{\alpha})
        \right]\bX_{/i,\calB_{1,+}}\\
		\bB=&~\bX_{/i,\calB_{1,+}}^\top
        {\rm diag}\left[
        \ddot{\ell}_{/i}(\hat{\bbeta}^{\alpha})
        \right]
  \bX_{/i,\calF\setminus\calB_{1,+}}.
	\end{align*}
 We start by defining
 \begin{align}
 \starbA 
 =& ~{\rm diag}\left[
        \lambda
        \left\{
        \ddot{r}_{\alpha}(\hat{\bbeta}^{\alpha}_{/ i})\right\}_{\calB_{1,+}}\right] ~+\bX_{/i,\calB_{1,+}}^\top
		{\rm diag}\left[
        \ddot{\ell}_{/i}(\hat{\bbeta}^{\alpha})
        \right]\bX_{/i,\calB_{1,+}}.
 \end{align}
 Note that in $\starbA$ we have removed one source of dependence, i.e., the dependence of the input argument of $\ddot{r}$ on $\bx_i$. As the first step we would like to show that the difference
\[
\left\vert\bx_{i,\calB_{1,+}}^{\top}
 \Big( \bA^{-1}
\bB\bB^{\top}
\bA^{-1}-
\starbA^{-1}
\bB\bB^{\top}
\starbA^{-1} \Big) \bx_{i,\calB_{1,+}}
\right\vert
\]
is negligible for large values of $\alpha$. Define
\begin{align}
\starbDelta = (\bA)^{-1}- (\starbA)^{-1}  =\bA^{-1} (\starbA - \bA) \starbA^{-1}
\end{align}
Since 
$$
\starbA- \bA = {\rm diag}\left[ \lambda    
        \ddot{r}_{\alpha}(\hat{\bbeta}^{\alpha}_{/ i})_{\calB_{1,+}}\right] - {\rm diag}\left[\lambda      
        \ddot{r}_{\alpha}(\hat{\bbeta}^{\alpha})_{\calB_{1,+}}\right],
        $$
  according to Lemma \ref{lem:rdderiv}, $\norm{\starbA-\bA} \leq 4\lambda \alpha(1-\eta)e^{-\frac12\alpha\kappa_1}$. We conclude that
 \[
\norm{\starbDelta} \leq \frac{\norm{\bA - \starbA} }{\sigma_{\min} (\bA) \sigma_{\min} (\starbA)} \leq \frac{\alpha(1- \eta) {\rm e}^{-\frac{1}{2}\alpha \kappa_1}}{ \lambda \eta^2},
 \]
 where we have used the lower bound $2 \lambda \eta$ for $\sigma_{\min}(\bA_2)$ and $\sigma_{\min} (\starbA)$. This lower bound is obtained from the ridge part of the regularizer. 
 Define 
 \begin{equation}\label{eq:def-gamma0s}
 \gamma_0^s(\alpha) := \frac{\alpha (1- \eta) {\rm e}^{-\frac{1}{2}\alpha \kappa_1}}{ \lambda \eta^2}. 
 \end{equation}
 As we will discuss later, we will ensure that $\gamma_0^s(\alpha) \rightarrow 0$ as $\alpha \rightarrow \infty$. Hence, $\|\starbDelta\|$ will be small too. Using this result and the triangle inequality we have
\begin{eqnarray}\label{eq:psi1-term1}
\lefteqn{|\bx_{i,\calB_{1,+}}^{\top}
 \Big( \bA^{-1}
\bB\bB^{\top}
\bA^{-1}-
\starbA^{-1}
\bB\bB^{\top}
\starbA^{-1} \Big) \bx_{i,\calB_{1,+}}| } \nonumber \\
&\leq& |\bx_{i,\calB_{1,+}}^{\top}
  \starbDelta
\bB\bB^{\top}
\bA^{-1}\bx_{i,\calB_{1,+}}| 
+ |\bx_{i,\calB_{1,+}}^{\top}
  \starbDelta
\bB\bB^{\top}
\starbA^{-1}\bx_{i,\calB_{1,+}}| \nonumber \\
&\leq& \norm{\bx_i}^2 \norm{\bB}^2 \norm{\starbDelta} \left( \frac{1}{\sigma_{\min}(\bA) } + \frac{1}{\sigma_{\min}(\starbA) } \right) \nonumber \\
&\leq& \norm{\bx_i}^2 \norm{\bB}^2 \norm{\starbDelta} \left( \frac{1}{\sigma_{\min}(\bA) }  + \frac{1}{\sigma_{\min}(\starbA)} \right) \nonumber \\
&\leq& \norm{\bx_i}^2 \norm{\bB}^2 \gamma_0^s(\alpha) \left(\frac{1}{\lambda \eta} \right). 
\end{eqnarray}
Similar to the proof of \eqref{eq:Bbdup}, it is not hard to see that 
\[
\norm{\bB_2} \leq  \omega_s 
\]
under event $\calA_L$. Furthermore, given the scaling we have considered in our paper, as will be clarified later, $\|\bX^{\top} \bX\| = O_p(1)$. Hence, the term $\norm{\bx_i}^2 \norm{\bB}^2 \gamma_0^s(\alpha) (\frac{1}{ \lambda \eta})$ will go to zero with high probability as $\alpha \rightarrow \infty$.  

In the rest of our proof, we will work with $\bx_{i, \calB_{1,+}}^{\top}\starbA^{-1}
\bB\bB^{\top}
\starbA^{-1} \bx_{i,\calB_{1,+}}$. Still there are two sources of dependency between the middle matrix and $\bx_{i, \calB_{1,+}}$: (1) the input argument of $\ddot{\ell}$, and (2) The dependence of $\calB_{1,+}$ and $\calF$ on $\bx_i$. In order to remove the dependence of $\calF$ and $\calB_{1,+}$, we lift the problem to a higher dimensional problem for the reasons that will become clearer later. Define the two sets:
\begin{eqnarray}
\calB^+ &:=& \calB_{1,+} \cup \calB_0, \nonumber \\
\calB^- &:=&  \calF \setminus \calB_{1,+}.
\end{eqnarray}
Use these two sets to define
\begin{align*}
		\tilde{\bA}  &=~ {\rm diag}\left[
        \lambda
        \left\{
        \ddot{r}_{\alpha}(\hat{\bbeta}^{\alpha}_{/ i})\right\}_{\calB^+}\right]  +\bX_{/i,\calB^+}^\top
		{\rm diag}\left[
        \ddot{\ell}_{/i}(\hat{\bbeta}^{\alpha})
        \right]\bX_{/i,\calB^+}\\
		\tilde{\bB}&=~\bX_{/i,\calB^+}^\top
        {\rm diag}\left[
        \ddot{\ell}_{/i}(\hat{\bbeta}^{\alpha})
        \right]\bX_{/i,\calB^-},
	\end{align*}
As will be clarified later, working with $\calB^+$ and $\calB^-$ will be helpful when we would like to remove the dependancies that are caused by $\calB^+$ and $\calB^-$. Hence, as the first step our aim is to obtain an upper bound on the difference  
\begin{eqnarray*}
|\bx_{i,\calB_{1,+}}^{\top}
\starbA^{-1}
\bB\bB^{\top}
\starbA^{-1}\bx_{i,\calB_{1,+}} - \bx_{i,\calB^+}^{\top}
\tilde{\bA}^{-1}
\tilde{\bB}\tilde{\bB}^{\top}
\tilde{\bA}^{-1}\bx_{i,\calB^+}|. 
\end{eqnarray*}
We write $\tilde{\bA}$ as  
\begin{eqnarray}
	\tilde{\bA}	=\begin{pmatrix}
			\bE & \bF\\
			\bF^\top &\bG
		\end{pmatrix},
\end{eqnarray}
where 
\begin{eqnarray}
\bE &=&~ {\rm diag}\left[
        \lambda
        \left\{
        \ddot{r}_{\alpha}(\hat{\bbeta}^{\alpha}_{/i})
        \right\}_{\calB_{1,+}}\right] \nonumber \\
		&&+\bX_{/i,\calB_{1,+}}^\top
		{\rm diag}\left[
        \ddot{\ell}_{/i}(\hat{\bbeta}^{\alpha})
        \right]\bX_{/i,\calB_{1,+}}, \nonumber \\
\bF &=& \bX_{/i,\calB_{1,+} }^\top
        {\rm diag}\left[
        \ddot{\ell}_{/i}(\hat{\bbeta}^{\alpha})
        \right]
  \bX_{/i,\calB_0}, \nonumber \\
  \bG &=&  {\rm diag}\left[
        \lambda
        \left\{
        \ddot{r}_{\alpha}(\hat{\bbeta}^{\alpha}_{/i})
        \right\}_{\calB_0}\right]  +\bX_{/i,\calB_{0}}^\top
		{\rm diag}\left[
        \ddot{\ell}_{/i}(\hat{\bbeta}^{\alpha})
        \right]\bX_{/i,\calB_{0}}. \nonumber 
\end{eqnarray}
Similar to the proof of \eqref{eq:Bbdup} it can be checked that 
\begin{eqnarray}\label{eq:Fbd1}
\|\bF\|_2 \leq  \omega^s 
\end{eqnarray}
 with probability at least $1-q_n$. 
Using the matrix inversion of block matrices (Lemma \ref{lem:BMIL}) we have
   	\begin{align}
	\tilde{\bA}^{-1}
		= \begin{pmatrix}
			\bE^{-1}+\bE^{-1}\bF \bH \bF^\top\bE^{-1} & -\bE^{-1}\bF\bH \\
			-\bH^\top \bF^\top\bE^{-1} &\bH
		\end{pmatrix},
	\end{align}
where 
\[
\bH = (\bG- \bF^\top \bE^{-1} \bF)^{-1}. 
\]
Furthermore, similar to the derivation in \eqref{eq:lowerbd:D} we have	
	\begin{align}\label{eq:Hbdup}
		\|\bH\|
		=&~(\sigma_{\min}(\bG-\bF^\top\bE^{-1}\bF))^{-1}\nonumber\\
        \le&~ 
		\frac{1}{\rho(\alpha) - \frac{(\omega^s)^2}{2 \lambda \eta}},
	\end{align}
 with probability $1-q_n$. Both $\rho(\alpha)$ and $\omega^s$ are defined at the beginning of the proof. It follows from the definition of $\rho(\alpha)$ that $\frac{1}{\rho(\alpha) - \frac{(\omega^s)^2}{2 \lambda \eta}}$ goes to zero for large values of $\alpha$ and hence the norm of $\bH$ should be considered as a small number. If we define 
   	\begin{align}
	\bar{\bA}^{\dagger}
		:= \begin{pmatrix}
			\starbA^{-1} & \mathbf{0}_{|\calB_1 \times \calB_0|} \\
			\mathbf{0}_{|\calB_0 \times \calB_1|}  &\mathbf{0}_{|\calB_0 \times \calB_0|}
		\end{pmatrix},
	\end{align}
and 
\[
\bDelta := \bar{\bA}^{\dagger}- \tilde{\bA}^{-1},
\]
then it is straightforward to show that 
\begin{align}\label{eq:uplifting:diff1}
\lefteqn{|\bx_{i,\calB_{1,+}}^{\top}
\starbA^{-1}
\bB\bB^{\top}
\starbA^{-1}\bx_{i,\calB_{1,+}}} \nonumber \\
&- ~ \bx_{i,\calB^+}^{\top}
\tilde{\bA}^{-1}
\tilde{\bB}\tilde{\bB}^{\top}
\tilde{\bA}^{-1}\bx_{i,\calB^+}| \nonumber \\
=&~|\bx_{i,\calB^+}^{\top}
\bar{\bA}^{\dagger}
\tilde{\bB} \tilde{\bB}^{\top}
\bar{\bA}^{\dagger}\bx_{i,\calB^+ }  \nonumber \\
&-  \bx_{i,\calB^+}^{\top}
\tilde{\bA}^{-1}
\tilde{\bB}\tilde{\bB}^{\top}
\tilde{\bA}^{-1}\bx_{i,\calB^+}| \nonumber \\
\le&~  |\bx_{i,\calB^+}^{\top}
\bDelta
\tilde{\bB} \tilde{\bB}^{\top}
\bDelta\bx_{i,\calB^+ }| \nonumber \\
&+2|\bx_{i,\calB^+}^{\top}
\bDelta
\tilde{\bB} \tilde{\bB}^{\top}
\bar{\bA}^{\dagger}\bx_{i,\calB^+}| \nonumber \\
\le& ~\|\bx_{i}\|^2 \|\tilde{\bB}\|^2 \|\bDelta\|\left( \|\bDelta\|+2\|\bar{\bA}^{\dagger}\|\right). 
\end{align}
To find an upper bound for \eqref{eq:uplifting:diff1} we have to bound $\|\bDelta\|$, $\|\tilde{\bB}_2\|$ and $\|\bar{\bA}^{\dagger}\|$. Similar to our previous calculations, we have under event $\calA_L$ that

\begin{eqnarray}\label{eq:upbd:tildeB}
\|\tilde{\bB}_2\| &\leq& \omega^s, 
\nonumber 
\\
 \|\bar{\bA}_2^{\dagger}\| &=& \frac{1}{\sigma_{\min} (\starbA_2)} \leq  \frac{1}{ 2\lambda \eta}. 
 \end{eqnarray}
Finally, by using Weyl's theorem (Lemma \ref{lem:Weyls}),  we have
\begin{align}\label{eq:bounddelta1}
\|\bDelta\| 
\le&~ \left\|  \begin{pmatrix}
			\bE^{-1}\bF \bH \bF^{\top}\bE^{-1} & \mathbf{0}_{|\calB_1 \times \calB_0|} \\
			\mathbf{0}_{|\calB_0 \times \calB_1|}  &\mathbf{0}_{|\calB_0 \times \calB_0|} ) \nonumber \\
   		\end{pmatrix}\right\|  
     +\left\|\begin{pmatrix}
			 \mathbf{0}_{|\calB_1 \times \calB_1|} & \mathbf{0}_{|\calB_1 \times \calB_0|} \\
			\mathbf{0}_{|\calB_0 \times \calB_1|}  &\bH \nonumber \\
   		\end{pmatrix}\right\| \nonumber \\
      &+ \left\| \begin{pmatrix}
			\mathbf{0}_{|\calB_1 \times \calB_1|} & \bE^{-1}\bF \bH \\
			\bH^{\top} \bF^{\top} \bE^{-1}  &\mathbf{0}_{|\calB_0 \times \calB_0|} ) \nonumber \\
   		\end{pmatrix}\right\| \nonumber \\
    \le &~  \frac{\|\bH\| \|\bF\|^2 }{\sigma^2_{\min} (\bE)}  +\|\bH\| +2 \frac{\|\bF\|\|\bH\|}{\sigma_{\min} (\bE)} \nonumber \\
    \le &~ \frac{1}{\rho(\alpha) - \frac{(\omega^s)^2}{2 \lambda \eta}} 
    \left(
    1+ \frac{\omega^s}{2 \lambda \eta}  + \frac{(\omega^s)^2}{(2 \lambda \eta)^2}  \right),
\end{align}
under event $\calA_L$. To obtain the last inequality we have  used \eqref{eq:Fbd1}, and \eqref{eq:Hbdup}. Define
\begin{equation}\label{eq:def-gamma1s}
    \gamma_1^s(\alpha) : = \frac{1}{\rho(\alpha) - \frac{(\omega^s)^2}{2 \lambda \eta}} 
    \left(
    1+ \frac{\omega^s}{2 \lambda \eta}  + \frac{(\omega^s)^2}{(2 \lambda \eta)^2} \right).
\end{equation}
By combining \eqref{eq:uplifting:diff1}, \eqref{eq:upbd:tildeB} and \eqref{eq:bounddelta1} we conclude that 
\begin{eqnarray}\label{eq:uplifting:diff}
\lefteqn{|\bx_{i,\calB_{1,+}}^{\top}
\starbA^{-1}
\bB\bB^{\top}
\starbA^{-1}\bx_{i,\calB_{1,+}}} \nonumber \\
&&- ~ \bx_{i,\calB^+}^{\top}
\tilde{\bA}^{-1}
\tilde{\bB}\tilde{\bB}^{\top}
\tilde{\bA}^{-1}\bx_{i,\calB^+}| \nonumber \\
&=&\|\bx_{i}\|^2 \|\tilde{\bB}\|^2 \|\bDelta\|\left( \|\bDelta\|+2\|\bar{\bA}^{\dagger}\|\right) \nonumber \\
&\leq& \|\bx_i\|^2 (\omega^s)^2 \|\bDelta\|  (\|\bDelta\| + \frac{1}{\lambda \eta}), \nonumber \\
&\leq& \|\bx_i\|^2 (\omega^s)^2 \gamma_1^s(\alpha) \left(\gamma_1^s(\alpha) + \frac{1}{\lambda \eta}
\right),
\end{eqnarray}
under event $\calA_L$. 

As discussed before, we eventually show that as $\alpha \rightarrow \infty$, $\gamma_1^s(\alpha) \rightarrow 0$ and hence the upper bound in \eqref{eq:uplifting:diff} will go to zero. 
Hence, in the rest of the proof, we aim to show that $\bx_{i,\calB^+}^{\top}
\tilde{\bA}^{-1}
\tilde{\bB}\tilde{\bB}^{\top}
\tilde{\bA}^{-1}\bx_{i,\calB^+}$ is small. 

In the next step we would like to remove the dependence of the input argument of $\ddot{\ell}$ on $\bx_i$. Define,
 \begin{align*}
		\check{\bA}  &:={\rm diag}\left[
        \lambda
        \left\{
        {\ddot{r}}(\hat{\bbeta}_{/ i}^{\alpha})\right\}\right] 
        +\bX_{/i,\calB^+}^\top
		{\rm diag}\left[
        \ddot{\ell}_{/i}(\hat{\bbeta}^{\alpha}_{/i})
        \right]\bX_{/i,\calB^+},\\
		\check{\bB}&:=\bX_{/i,\calB^+}^\top
        {\rm diag}\left[
        \ddot{\ell}_{/i}(\hat{\bbeta}_{/i}^{\alpha})
        \right]
  \bX_{/i,\calB^-}.
	\end{align*}

Define 
\[
\bDelta_\ell := {\rm diag}\left[
        \ddot{\ell}_{/i}(\hat{\bbeta}^{\alpha})
        \right]- {\rm diag}\left[
        \ddot{\ell}_{/i}(\hat{\bbeta}_{/i}^{\alpha})
        \right] 
\]
We have
\begin{eqnarray*}
    \tilde{\bA} = \check{\bA} + \bX_{/i,\calB^+}^\top
		\bDelta_\ell \bX_{/i,\calB^+}
\end{eqnarray*}
According to Lemma \ref{lem:woodtwice},
\begin{align}
&\tilde{\bA}^{-1} \nonumber\\
  =&~  \check{\bA}^{-1} - \check{\bA}^{-1} \bX_{/i,\calB^+}^\top \bDelta_\ell \bX_{/i,\calB^+}\check{\bA}^{-1} \nonumber \\
&+ \check{\bA}^{-1} \bX_{/i,\calB^+}^\top \bDelta_\ell \bX_{/i,\calB^+} \tilde{\bA}^{-1} \bX_{/i,\calB^+}^\top \bDelta_\ell \bX_{/i,\calB^+}\check{\bA}^{-1} 
\end{align}
Also, we have
\begin{eqnarray}
		\tilde{\bB} - \check{\bB} =  \bX_{/i,\calB^+}^\top
        \bDelta_\ell
  \bX_{/i,\calB^-}. 
\end{eqnarray}
Define $\bdelta_B = \tilde{\bB} - \check{\bB}$ and $\bdelta_{A^{-1}} = \tilde{\bA}^{-1} - \check{\bA}^{-1}$. Then, we have
\begin{align}\label{eq:TotalTermAfterLift}
&\bx_{i,\calB^+}^{\top}
\tilde{\bA}^{-1}
\tilde{\bB}\tilde{\bB}^{\top}
\tilde{\bA}^{-1}\bx_{i,\calB^+} \nonumber \\
=&~\bx_{i,\calB^+}^{\top}
\tilde{\bA}^{-1}
\tilde{\bB}\tilde{\bB}^{\top}
\bdelta_{A^{-1}}\bx_{i,\calB^+} \nonumber \\
&+ \bx_{i,\calB^+}^{\top}
\tilde{\bA}^{-1}
\tilde{\bB}\bdelta_B^{\top}
\check{\bA}^{-1}\bx_{i,\calB^+} \nonumber \\ &+\bx_{i,\calB^+}^{\top}
\tilde{\bA}^{-1}
\bdelta_{B}\check{\bB}^{\top}
\check{\bA}^{-1}\bx_{i,\calB^+}\nonumber \\ &+ \bx_{i,\calB^+}^{\top}
\bdelta_{A^{-1}}
\check{\bB}\check{\bB}_2^{\top}
\check{\bA}^{-1}\bx_{i,\calB^+} \nonumber \\
&+\bx_{i,\calB^+}^{\top}
\check{\bA}^{-1}
\check{\bB}\check{\bB}^{\top}
\check{\bA}^{-1}\bx_{i,\calB^+}. 
\end{align}

Our goal is to show that each of the terms in \eqref{eq:TotalTermAfterLift} are small for large values of $n,p$. The two terms $\bx_{i,\calB^+}^{\top}
\tilde{\bA}^{-1}
\tilde{\bB}\tilde{\bB}_2^{\top}
\bdelta_{A^{-1}}\bx_{i,\calB^+}$ and $  
\bx_{i,\calB^+}^{\top}
\bdelta_{A^{-1}}
\check{\bB}\check{\bB}^{\top}
\check{\bA}^{-1}\bx_{i,\calB^+}$ can be bounded in very similar ways and will have similar upper bounds. Also, the two terms 
$\bx_{i,\calB^+}^{\top}
\tilde{\bA}^{-1}
\tilde{\bB}\bdelta_B^{\top}
\check{\bA}^{-1}\bx_{i,\calB^+ }$ and $
\bx_{i,\calB^+}^{\top}
\tilde{\bA}^{-1}
\bdelta_{B}\check{\bB}^{\top}
\check{\bA}^{-1}\bx_{i,\calB^+}$ can be handled in similar fashion and will have similar upper bounds. 
Hence, we only study the following three terms: \\
\begin{enumerate}
\item Finding an upper bound for $\bx_{i,\calB^+}^{\top}
\check{\bA}^{-1}
\check{\bB}\check{\bB}^{\top}
\check{\bA}^{-1}\bx_{i,\calB^+}$: \\
Note that if it were not for the dependance of $\calB^+$ on $\bx_i$ we could claim that $\bx_{i,\calB^+}$ is independent of $\check{\bA}^{-1}
\check{\bB}\check{\bB}^{\top}
\check{\bA}^{-1}$, and we could use the Hanson-Wright inequality (Lemma \ref{lem:HansonWrightIn}). So, in order to remove the dependence to be able to use the Hanson-Wright inequality, we will use union bound on sets $\calB^+$ and $\calB^-$ in the following way. First note that $|\calB^+|\geq p-d_n$ and $|\calB^-| \leq d_n$. Let $\calT^+$ and $\calT^-$ denote two fixed sets of size larger than $p-d_n$, and smaller than $d_n$  (not dependent on $\bx_i$) respectively. We define
\begin{align}
\check{\check{\bA}}:=
&~{\rm diag}\left[
        \lambda
        \left\{
        {\ddot{r}}_{\alpha}(\hat{\bbeta}_{/ i}^{\alpha})\right\}\right] 
        +\bX_{/i,\calT^+}^\top
		{\rm diag}\left[
        \ddot{\ell}_{/i}(\hat{\bbeta}^{\alpha}_{/ i})
        \right]\bX_{/i,\calT^+}. \nonumber \\
  \check{\check{\bB}} 
  =&~ \bX_{/i,\calT^+}^\top
        {\rm diag}\left[
        \ddot{\ell}_{/i}(\hat{\bbeta}^{\alpha}_{/ i})
        \right] \bX_{/i,\calT^-}     
\end{align}
For fixed $\calT^+$ and $\calT^{-}$,  $\bx_{i, \calT}$ is independent of $\check{\check{\bA}}$ and $\check{\check{\bB}}$. Therefore if 
\[
\bG = \check{\check{\bA}}^{-1}\check{\check{\bB}}\check{\check{\bB}}^\top \check{\check{\bA}}^{-1}, 
\]
then from the Hanson-Wright inequality, Lemma \ref{lem:HansonWrightIn}, we have:
\begin{align}\label{eq:HansWright}
&\PP (|\bx_{i, \calT^+}^\top \bG \bx_{i, \calT^+}
- 
\EE(\bx_{i, \calT^+}^\top \bG \bx_{i, \calT^+}| \ \bX_{/ i}, \by_{/ i} )|> t \ | \ \bX_{/ i}, \by_{/ i} ) \nonumber \\
\leq&~ 2 {\rm e}^{-c \Big( \frac{(p-d_n)^2 t^2}{  \|\bG\|_{HS}^2}
\wedge 
\frac{(p-d_n)t}{ \|\bG\|_2 }\Big) }. \hspace{2.5cm}
\end{align}
To use this bound we have to calculate the three terms: $\|\bG\|_2$, $\|\bG\|_{HS}$, and\\ $\EE(\bx_{i, \calT^+}^\top \bG \bx_{i, \calT^+}| \ \bX_{/ i}, \by_{/ i} )$:
\begin{enumerate}
\item $\|\bG \|_2$:\\
It is straightforward to see that (see e.g. the derivation of \eqref{eq:Bbdup})
\begin{equation}\label{eq:Gspectbd}
\|\bG \|_2 \leq \frac{\|\check{\check{\bB}}\|^2}{\sigma^2_{\min}(\check{\check{\bA}})}\leq  \frac{\omega^s}{4\lambda^2 \eta^2},
\end{equation}
under event $\calA_L$. 
\item $\|\bG\|_{HS}^2$: \\
The rank of matrix $\calG$ is at most $d_n$ (the maximum size of $\calT^-$). Hence,
\begin{equation}\label{eq:G_HS_bd}
\|\bG\|^2_{HS} \leq d_n \|\bG \|_2^2 = \frac{d_n (\omega^s)^2}{16\lambda^4 \eta^4}
\end{equation}
\item $\EE(\bx_{i, \calT^+}^\top \bG \bx_{i, \calT^+})$: \\
Let $\bSigma_{\calT}$ the covariance matrix of $\bx_{i, \calT}$. We have
\begin{eqnarray}\label{eq:boundtrace}
\lefteqn{|\EE\bx_{i, \calT^+}^\top \bG \bx_{i, \calT^+} |=\left| {\rm Tr} (\bSigma_{\calT}^{1/2} \bG \bSigma_{\calT}^{1/2}) \right| }\nonumber \\
&\overset{(a)}{\leq}& 
d_n \| \bSigma_{\calT}^{1/2} \bG \bSigma_{\calT}^{1/2}\|  
\leq d_n \| \bSigma_{\calT}\| \|\bG\| \nonumber \\ 
&\leq& d_n \rho_{\max} \|\bG\| 
\leq \frac{d_n \rho_{\max} \omega^s}{4\lambda^2 \eta^2}, 
\end{eqnarray}
with probability $1-q_n$. Here, to obtain Inequality $(a)$, we used the fact that the rank of $\bSigma^{1/2} \bG \bSigma^{1/2}$ is less than or equal to the rank of $\bG$ which is less than or equal to $d_n$.  
\end{enumerate}
Furthermore, from Lemma \ref{lem:maxsingularvalue} we have   
\begin{equation}\label{eq:XtopXspec}
\PP (\sigma_{\max}(\bX^{\top} \bX) \geq (\sqrt{n}+ 3\sqrt{p})^2 \rho_{\max} ) \leq {\rm e}^{-p},
\end{equation}
where $\rho_{\max} = \sigma_{\max}(\Sigma)$. Let the event $\cal{E}$ denote the event  $\sigma_{\max}(\bX^{\top} \bX) < (\sqrt{n}+ 3\sqrt{p})^2 \rho_{\max}$. We know that
\[
\PP (\calE) > 1-{\rm e}^{-p}. 
\]
We remind the reader that 
$$
\omega^s =   \|\bX^{\top}\bX\|
\sup_{t\in[0,1]}
\max_{
1\le i\le n}
\ddot{\ell}_i(
t\hat{\bbeta}^{\alpha}
+(1-t)\hat{\bbeta}_{/i}^{\alpha}.
$$
Combining \eqref{eq:HansWright}, \eqref{eq:Gspectbd}, \eqref{eq:G_HS_bd}, \eqref{eq:boundtrace}, \eqref{eq:XtopXspec} and using the following definitions:
\begin{eqnarray}
\gamma_3 &:=& \frac{  (\sqrt{n} +3 \sqrt{p})^2 \rho_{\max}}{4 \eta^2 \lambda^2} 
\end{eqnarray}
we have
\begin{align}
&\PP (\bx_{i, \calT}^\top \bG \bx_{i, \calT^+}> t + {\EE}\bx_{i, \calT^+}^\top \bG \bx_{i, \calT^+} ) \nonumber \\
\leq&~
{\PP} (\bx_{i, \calT^+}^\top \bG \bx_{i, \calT^+}> t + {\EE}\bx_{i, \calT^+}^\top \bG \bx_{i, \calT^+}, \calE ) + {\PP} (\calE^c) \nonumber \\
\leq&~
2{\rm e}^{-c \Big( \frac{(p-d_n))^2 t^2 }{d_n \gamma_3^2}
\wedge \frac{(p-d_n)t}{ \gamma_3} \Big)} + {\rm e}^{-p}. \end{align}
Hence, we have
\begin{align}
&\PP (\bx_{i, \calT^+}^\top \bG \bx_{i, \calT^+}> t + d_n \gamma_3 ) \nonumber \\ 
\leq&~
2{\rm e}^{-c \Big( \frac{(p-d_n)^2 t^2 }{d_n \gamma_3^2}\wedge \frac{(p-d_n)t}{ \gamma_3} \Big)} + {\rm e}^{-p}. 
\end{align}
If we define the event $\tilde{\calE}$ as the event that \eqref{eq:assn_loss} holds, then we have
\begin{align}
&\PP ( 
\max_{\underset{|\calT^+| \geq p-d_n}{\calT^+}} 
\max_{\underset{|\calT^-| \leq d_n}{\calT^-}} 
\bx_{i, \calT^+}^\top \bG \bx_{i, \calT^+}> t + d_n \gamma_3 ) \nonumber \\
&\leq \PP (
\max_{\underset{|\calT^+| \geq p-d_n}{\calT^+}} 
\max_{\underset{|\calT^-| \leq d_n}{\calT^-}} 
\bx_{i, \calT^+}^\top \bG \bx_{i, \calT^+}> t + d_n \gamma_3  | \calE, \tilde{\calE}) 
+ \PP (\calE) + \PP (\tilde{\calE}) \nonumber \\
&\leq 
\sum_{\underset{|\calT^+| \geq p-d_n}{\calT^+}} 
\sum_{\underset{|\calT^-| \leq d_n}{\calT^-}}  
\PP( \bx_{i, \calT^+}^\top \bG \bx_{i, \calT^+}> t + d_n \gamma_3 | \calE, \tilde{\calE} )
+ \PP(\calE) + \PP(\tilde{\calE}) \nonumber \\
&\leq  2 d_n {p \choose d_n}^2 \left( 2{\rm e}^{-c \Big( \frac{(p-d_n)^2 t^2 }{d_n\gamma_3^2}\wedge \frac{(p-d_n)t}{ \gamma_3} \Big)}  \right) + {\rm e}^{-p} + q_n\nonumber \\
&\overset{(a)}{\leq}  {\rm e}^{4 d_n \log \frac{ep}{d_n}} 
\left( 2{\rm e}^{-c \Big( \frac{(p-d_n)^2 t^2 }{d_n \gamma_3^2}\wedge 
\frac{(p-d_n)t}{ \gamma_3} \Big)}  \right) +{\rm e}^{-p}+q_n,
\end{align}
where to obtain inequality (a) we used Stirling approximation together with the assumption $d_n < {\rm e}^{d_n \log (ep/d_n)}$. 
By setting $t = \frac{d_n \log ^2 d_n}{p-d_n}$ we can obtain
\begin{align}\label{eq:final:termcomplicated}
&\PP (\bx_{i,\calB^+}^{\top}
\check{\bA}^{-1}
\check{\bB}\check{\bB}^{\top}
\check{\bA}^{-1}\bx_{i,\calB^+} > \frac{d_n \log^2 p}{p-d_n } + d_n\gamma_3) \nonumber \\
&\leq {\rm e}^{d_n \log \frac{ep}{d_n}} \left( 2{\rm e}^{-c \Big( 
\frac{d_n \log^4 p }{\gamma_3^2}\wedge \frac{d_n \log^2 p}{ \gamma_3} \Big)} \right)+ q_n + {\rm e}^{-p}. 
\end{align}
The probability in \eqref{eq:final:termcomplicated} is small for large values of $d_n$ and $p$.

\item Finding an upper bound for $\bx_{i,\calB^+}^{\top}
\tilde{\bA}^{-1}
\tilde{\bB}\tilde{\bB}^{\top}
\bdelta_{A^{-1}}\bx_{i,\calB^+} $: \\

Define
\begin{equation}\label{eq:vcheck}
{\check{\bv}} := \bX_{/i,\calB^+} \check{\bA}^{-1} \bx_{i,\calB^+},
\end{equation}
and 
\begin{eqnarray}
    \bu^{\top} &:=& -\bx_{i,\calB^+}^{\top}
\tilde{\bA}^{-1}
\tilde{\bB}\tilde{\bB}^{\top}
 \check{\bA}^{-1} \bX_{/i,\calB^+}^\top \nonumber \\
    \tilde{\bu}^{\top} &:=& \bx_{i,\calB^+}^{\top}
\tilde{\bA}^{-1}
\tilde{\bB}\tilde{\bB}^{\top}\check{\bA}^{-1} \bX_{/i,\calB^+}^\top \bDelta_{\ell} \bX_{/i,\calB^+} \tilde{\bA}^{-1} \bX_{/i,\calB^+}^\top
\end{eqnarray}

Using Lemma \ref{lem:woodtwice} we have
\begin{align}
 &|\bx_{i,\calB^+}^{\top}
\tilde{\bA}^{-1}
\tilde{\bB}\tilde{\bB}_2^{\top}
\bdelta_{A^{-1}}\bx_{i,\calB^+}|  \nonumber \\
\overset{(a)}{\leq}
&~|\bx_{i,\calB^+}^{\top}
\tilde{\bA}^{-1}
\tilde{\bB}\tilde{\bB}^{\top}
 \check{\bA}^{-1} \bX_{/i,\calB^+}^\top \bDelta_\ell \bX_{/i,\calB^+}\check{\bA}^{-1}\bx_{i,\calB^+}| \nonumber \\
&+|\bx_{i,\calB^+}^{\top}
\tilde{\bA}^{-1}
\tilde{\bB}\tilde{\bB}^{\top}\check{\bA}^{-1} \bX_{/i,\calB^+}^\top \bDelta_{\ell} \bX_{/i,\calB^+} \tilde{\bA}^{-1} \bX_{/i,\calB^+}^\top \bDelta_\ell \bX_{/i,\calB^+}\check{\bA}^{-1} \bx_{i,\calB^+}| \nonumber \\
 =&~ |\bu^\top \bDelta_\ell \check{\bv} | +|\tilde{\bu}^{\top} \bDelta_\ell \check{\bv}|     \nonumber \\
 \overset{(b)}{\leq}
 &~
 \sqrt{\sum_i \bDelta_{\ell, ii}^2} \sqrt{\sum_i \bu_i^2 \check{\bv}_i^2} 
 + \sqrt{\sum_i \Delta_{\ell, ii}^2} \sqrt{\sum_i \tilde{\bu}_i^2 \check{\bv}_i^2} \nonumber \\
 \overset{(c)}{\leq} 
 &~
 \sqrt{\sum_i \bDelta_{\ell, ii}^2} (\sum_i \bu_i^4)^{\frac{1}{4}} (\sum_i  \check{\bv}_i^4)^{\frac{1}{4}} 
 + \sqrt{\sum_i \Delta_{\ell, ii}^2} (\sum_i \tilde{\bu}_i^4)^{\frac{1}{4}} (\sum_i  \check{\bv}_i^4)^{\frac{1}{4}} \nonumber
\end{align}
\begin{align}\label{eq:bound:secondterm}
 \leq&~ \sqrt{\sum_i \bDelta_{\ell, ii}^2} (\sum_i \bu_i^2)^{\frac{1}{2}} (\max_{i} |\check{\bv}_i|)^{\frac{1}{2}} (\sum_i  \check{\bv}_i^2)^{\frac{1}{4}} 
 + \sqrt{\sum_i \bDelta_{\ell, ii}^2} (\sum_i \tilde{\bu}_i^2)^{\frac{1}{2}} (\max_{i} |\check{\bv}_i|)^{\frac{1}{2}} (\sum_i  \check{\bv}_i^2)^{\frac{1}{4}}.
 \end{align}

Note that since $\sigma_{\min} (\tilde{\bA}) > \frac{1}{2 \lambda \eta}$ and $\sigma_{\min} (\check{\bA}) > \frac{1}{2 \lambda \eta}$ we have
\begin{align}\label{eq:boundu_variables}
\|\bu\|_2 &\leq \| \bx_i\|_2 \frac{\|\tilde{\bB}\|^2 \|\bX\|}{(2 \lambda \eta)^2}, \nonumber \\
\|\tilde{\bu}\|_2 &\leq \|\bx_i\|_2 \frac{\|\tilde{\bB}\|^2 \|\bX\|^2  \|\bDelta_\ell\|}{2 (\lambda \eta)^3}.
\end{align}
Furthermore,
\begin{align}\label{eq:DeltaEllBound}
\|\bDelta_\ell\| 
=&~ \sqrt{\sum_j \Delta_{\ell, ii}^2}  
\overset{(a)}{\leq}
~
\polylog(n)\|\hat{\bbeta}^\alpha - \hat{\bbeta}_{/i}^\alpha\|_2 
\overset{(b)}{\leq}~
\polylog(n)  \frac{|\dot{\ell} (\hat{\bbeta}^{\alpha})| \|\bx_i\|_2}{2 \lambda \eta} 
\nonumber\\
\leq
&~
\frac{\polylog(n) \|\bx_i\|}{2 \eta \lambda} 
\end{align}
under the event $\calA_L\cap\calA_s$. Here Inequality (a) is a result of \eqref{eq:assn_loss} in Assumption A4. Note that we are using the same notation for all the terms that are polynomial functions of $\log(n)$. To obtain Inequality (b) we have used Part 2 of Lemma \ref{lem:beta-size}.   
Finally,
\begin{align}\label{eq:vcheckl2bound_1}
\|\check{\bv}\|_2 
= \|
 \bX_{/i,\calB^+} \check{\bA}^{-1} \bx_{i,\calB^+}\| 
 \leq&~
 \frac{\|\bX\| \|\bx_i\|_2}{\sigma_{\min} (\check{\bA})} \leq \frac{\|\bX\| \|\bx_i\|_2}{ 2\lambda \eta}.  
\end{align}
To finish our bound for \eqref{eq:bound:secondterm} we have to bound $\|{\check{\bv}}\|_\infty$. Note that
\[
{\check{\bv}} = \bX_{/i,\calB^+} \check{\bA}^{-1} \bx_{i,\calB^+}.
\]
The main difficulty in bounding this term is that due to the dependence of $\calB^+$ on $\bx_i$ it is hard to characterize the distribution of the elements of ${\check{\bv}}$. To remove this dependency we again want to use the union bound on the different choices of $\calB^+$. Towards this goal, suppose that for a fixed set $\calT$ of size larger than $p-d_n$ we define:
\[
\ccbv = \bX_{/i, \calT} \check{\check{\bA}}^{-1} \bx_{i,\calT},
\]
where 
\begin{align}
\check{\check{\bA}}&={\rm diag}\left[
        \lambda
        \left\{
        {\ddot{r}}(\hat{\bbeta}_{/i}^{\alpha})\right\}\right] +\bX_{/i,\calT}^\top
		{\rm diag}\left[
        \ddot{\ell}_{/i}(\hat{\bbeta}^{\alpha}_{/i})
        \right]\bX_{/i,\calT}. \nonumber \\
\end{align}
Note that the distribution of $\ccbv$ given $\bX_{/i}, \by_{/i}$ is $N\left(0, \frac{1}{n}\bX_{/i, \calT} \check{\check{\bA}}^{-2} \bX_{/i, \calT}^{\top}\right)$. Furthermore, we have 
\[
\| \bX_{/i, \calT} \check{\check{\bA}}^{-2} \bX_{/i, \calT}^{\top}\|
\leq  \frac{\|\bX_{/i} \bX_{/i}^{\top}\|}{ 4 \eta^2 \lambda^2}.
\]
Hence, using Corollary \ref{cor:maxGauss}
we have
\begin{align}\label{eq:uppervcc}
\PP\left(\max_i \ccbv_i >    \sqrt{ \frac{\| \bX_{/i} \bX_{/i}^{\top} \|}{n \lambda \eta}\log 2n } + t \ | \ \bX_{/i} , \by_{/i}\right) 
&\leq 2 {\rm e}^{-  \frac{4 n \lambda^2 \eta^2 t^2}{\|\bX_{/i} \bX_{/i}^{\top}\| }}. 
\end{align}
It is straightforward to check that
\[
\max_i \|\bX_{/i}\bX_{/i}^{\top}\|  \leq \|\bX \bX^{\top}\|.
\]
Furthermore, according to Lemma \ref{lem:maxsingularvalue} with probability larger than  ${\rm e}^{-p}$  we have 
\begin{equation}\label{eq:upperbd:wishart}
{\PP} (\|\bX^{\top}\bX\| \geq (\sqrt{n}+ 3\sqrt{p})^2 \rho_{\max} ) \leq {\rm e}^{-p}. 
\end{equation}
Define 
\begin{equation}
\gamma_4(n) := \sqrt{ \frac{(\sqrt{n}+ 3\sqrt{p})^2 \rho_{\max}}{n \lambda \eta}\log 2n },
\end{equation}
and let the event $\calE$ denote the event that $\|\bX^{\top}\bX\| \leq (\sqrt{n} + 3 \sqrt{p})^2 \rho_{\max}$. Then, from \eqref{eq:upperbd:wishart} and \eqref{eq:uppervcc} we obtain
\begin{align}
&\PP \left(\max_i \ccbv_i >    \gamma_4(n) + t \right) \nonumber \\ 
\leq&~ \PP\left(\max_i \ccbv_i >    \gamma_4(n) + t , \calE \right) + \PP(\calE^c) \nonumber \\
\leq&~ {\rm e}^{-p} + 2 {\rm e}^{-  \frac{4n \lambda^2 \eta^2 t^2}{(\sqrt{n}+ 3\sqrt{p})^2 \rho_{\max} }}. 
\end{align}
Note that we expect $\rho_{\max} = O(\frac{1}{p})$. Hence, $\gamma_4(n)$ is expected to be $O\left(\sqrt{\frac{\log n}{n}}\right)$. 
So far, we have only considered a bound for $\check{\check{\bv}}$ in which a fixed set $\calT$ is only considered, despite the fact that the quantity we are interested in is $\|\check{\bv}\|_\infty$ in which $\calT$ is replaced with set $\calB^+$. To resolve the issue we use the union bound. We have
\begin{align}
&  \PP \left(\max_{\calT: |\calT|> p-d_n}
\max_i \ccbv_i >    \gamma_4(n) + t \right) \nonumber \\
\leq&~\!\!\!\! \!\!\!\! \sum_{\calT: |\calT| > p-d_n}
\!\!\!\! \PP \left(\max_i \ccbv_i >   \gamma_4(n)+ t \right) \nonumber \\
\leq&~
d_n{p \choose p-d_n} \left(
{\rm e}^{-p} + 2 {\rm e}^{-  \frac{n \lambda^2 \eta^2 t^2}{(\sqrt{n}+ 3\sqrt{p})^2 \rho_{\max} }}\right) \nonumber \\
\leq&~ {\rm e}^{ 2d_n \log \frac{e p}{d_n}}  \left(
{\rm e}^{-p} + 2 {\rm e}^{-  \frac{4 n \lambda^2 \eta^2 t^2}{(\sqrt{n}+ 3\sqrt{p})^2 \rho_{\max} }}\right). 
\end{align}
Hence, we conclude that
\begin{align}\label{eq:bound_vcheck}
 &\PP \left( \max_i \check{\bv}_i >    \gamma_4(n) + t \right) \nonumber \\
\leq&~ {\rm e}^{ 2d_n \log \frac{e p}{d_n}}  \left(
{\rm e}^{-p} + 2 {\rm e}^{-  \frac{n \lambda^2 \eta^2 t^2}{(\sqrt{n}+ 3\sqrt{p})^2 \rho_{\max} }}\right). 
\end{align}
Setting $t =\frac{1}{\lambda\eta} \sqrt{\frac{d_n \log^2 p}{ n}}$ and combining this equation with \eqref{eq:bound:secondterm}, \eqref{eq:boundu_variables}, \eqref{eq:DeltaEllBound}, and \eqref{eq:vcheckl2bound_1}, we conclude that if we define
\begin{align}\label{eq:def:gamma5}
\gamma_5 (n) &:= \frac{\polylog(n) \|\bx_i\|_2 }{2 \lambda \eta} \left(\gamma_4 (n) + 
\dfrac{1}{\lambda\eta}
\sqrt{\frac{d_n\log^2 p}{n}}\right) \nonumber \\
& \times \frac{\|\bX\| \|\bx_i\|}{2 \lambda \eta} \frac{\|\bx_i\| \|\tilde{\bB}\|^2 \|\bX\|}{4 \lambda^2 \eta^2} \left(1+ \frac{\|\bX\| c_3(n) c_4(n)}{4 \lambda^2 \eta^2} \right),
\end{align}
then
\begin{align}\label{eq:gamma5-prob}
&\PP(|\bx_{i,\calB^+}^{\top}
\tilde{\bA}^{-1}
\tilde{\bB}\tilde{\bB}^{\top}
\bdelta_{A^{-1}}\bx_{i,\calB^+}| > \gamma_5(n)) \nonumber \\
\leq&~ q_n + \check{q}_n + {\rm e}^{ 2d_n \log \frac{e p}{d_n}}  \left(
{\rm e}^{-p} + 2 {\rm e}^{-  \frac{d_n\log^2p}{(\sqrt{n}+ 3\sqrt{p})^2 \rho_{\max} }}\right). 
\end{align}
As discussed before, we have $\gamma_4(n) = O(\sqrt{\frac{\log n}{n}})$. Also, all the norms $\|\bX\|, \|\bx_i\|, \|\tilde{\bB}\|$ are $O_P(1)$, hence we expect $\gamma_5(n)$ to go to zero at the rate $\sqrt{d_n{\polylog} (n)/n}$. This heuristic argument will be be made rigorous later. \\

 \item $\bx_{i,\calB^+}^{\top}
\tilde{\bA}^{-1}
\tilde{\bB}\bdelta_B^{\top}
\check{\bA}^{-1}\bx_{i,\calB^+}$: \\
Using a similar argument as the one presented in \eqref{eq:bound:secondterm} we obtain:
\begin{align}\label{eq:uppboundABdelBA1}
&\lefteqn{\bx_{i,\calB^+}^{\top}
\tilde{\bA}^{-1}
\tilde{\bB}\bdelta_B^{\top}
\check{\bA}^{-1} \bx_{i,\calB^+} } \nonumber \\
=&~\bx_{i,\calB^+}^{\top}
\tilde{\bA}^{-1}
\tilde{\bB} \bX_{/ i, \calB^-} \bDelta_\ell \bX_{/ i, \calB^+}
\check{\bA}^{-1}\bx_{i,\calB^+} \nonumber \\
=&~\bx_{i,\calB^+}^{\top}
\tilde{\bA}^{-1}
\tilde{\bB} \bX_{/ i, \calB^-} \bDelta_\ell {\check{\bv}}
= \check{\bu} \bDelta_\ell {\check{\bv}}  \nonumber \\
\leq&~\sqrt{\sum_i \bDelta_{\ell, ii}^2} \| \check{\bu}\| (\max_{i} |\check{\bv}_i|)^{\frac{1}{2}} \|\check{\bv}\|^{\frac{1}{2}}.
\end{align}
In the above equations we have used 
\[
\check{\bu}^\top : = \bx_{i,\calB^+}^{\top}
\tilde{\bA}^{-1}
\tilde{\bB} \bX_{/ i, \calB^-},
\]
and 
\[
{\check{\bv}} = \bX_{/i,\calB^+} \check{\bA}^{-1} \bx_{i,\calB^+}.
\]
According to \eqref{eq:DeltaEllBound} and \eqref{eq:bound_vcheck} we have
\begin{eqnarray}\label{eq:uppboundABdelBA2}
 \sqrt{\sum_j \Delta_{\ell, ii}^2} \leq  \frac{\polylog(n) \|\bx_i\|_2}{2 \lambda \eta},
\end{eqnarray}
with probability larger than $1- q_n - \check{q}_n$, and
\begin{eqnarray}\label{eq:uppboundABdelBA3}
\|\check{\bv}\|_2 \leq \frac{\|\bX\| \|\bx_i\|_2}{ 2\lambda \eta},  
\end{eqnarray}
and
\begin{align}\label{eq:uppboundABdelBA4}
\PP\left( \max_i \check{\bv}_i >    \gamma_5(n) + t \right) 
\leq &~ {\rm e}^{ 2d_n \log \frac{e p}{d_n}}  \left(
{\rm e}^{-p} + 2 {\rm e}^{-  \frac{4 n \lambda^2 \eta^2 t^2}{(\sqrt{n}+ 3\sqrt{p})^2 \rho_{\max} }}\right).  
\end{align}
Hence, the only remaining term to bound is $\|\check{\bu}\|$. It is straightforward to see that
\begin{align}\label{eq:uppboundABdelBA5}
\|\check{\bu}\|_2 \leq \frac{\|\bx_i\|\|\tilde{\bB}\| \| \bX\|}{\sigma_{\min} (\tilde{\bA})} \leq \frac{\|\bx_i\|\|\tilde{\bB}\| \| \bX\|}{2\lambda \eta}. 
\end{align}
Hence, combining \eqref{eq:uppboundABdelBA1}, \eqref{eq:uppboundABdelBA2}, \eqref{eq:uppboundABdelBA3},  \eqref{eq:uppboundABdelBA4}, and \eqref{eq:uppboundABdelBA5} we will have that if we set $t = \frac{1}{\lambda\eta}\sqrt{\frac{d_n \log^2 p}{n}}$, and define
\begin{align}
\gamma_6(n) : =&~ \frac{\polylog(n) \|\bx_i\|^3\|\bX\|^2\|\tilde{\bB}\| }{(2 \lambda \eta)^3} 
\left(\gamma_5(n) + 
\frac{1}{\lambda\eta}
\sqrt{\frac{d_n \log^2 p}{n}}\right),
\end{align}
then
\begin{align}\label{eq:gamma6-prob}
&\PP(|\bx_{i,\calB^+}^{\top}
\tilde{\bA}^{-1}
\tilde{\bB}\bdelta_B^{\top}
\check{\bA}^{-1}\bx_{i,\calB^+}| >\gamma_6(n) ) \nonumber \\
\leq&~ q_n + \check{q}_n +{\rm e}^{ 2s \log \frac{e p}{s}}  \left(
{\rm e}^{-p} + 2 {\rm e}^{-  \frac{4d_n\log^2p}{(\sqrt{n}+ 3\sqrt{p})^2 \rho_{\max} }}\right).  
\end{align}
\end{enumerate}

Plugging in the bounds from equations~\eqref{eq:final:termcomplicated},\eqref{eq:gamma6-prob} and similarly bounding the remaining terms in \eqref{eq:TotalTermAfterLift}, we obtain
\begin{align*}
&\bx_{i,\calB^+}^{\top}
    \tilde{\bA}^{-1}
    \tilde{\bB}\tilde{\bB}_2^{\top}
    \tilde{\bA}^{-1}
    \bx_{i,\calB^+}\\
    \le&~
    2\gamma_5(n)+2\gamma_6(n)+\dfrac{d_n\log^2p}{p-d_n}
    +d_n\gamma_3\\
    \le&~
    \dfrac{d_n\log^2p}{p-d_n}
    +
    \dfrac{ (\sqrt{n}+3\sqrt{p})^2d_n\rho^2_{\max}}{4\eta^2\lambda^2}\\
    &~+
    \frac{\polylog(n)\|\bx_i\|^3 \|\bX\|^2\omega^s}{ 
    4\lambda^3 \eta^3} \left(
1+
\dfrac{\omega^s}{2\lambda\eta}
+ \frac{\omega^s\|\bX\| \polylog(n)}{8 \lambda^3 \eta^3} \right)
\nonumber\\
&\times \left(
    \sqrt{ \frac{(\sqrt{n}+ 3\sqrt{p})^2 \rho_{\max}}{n \lambda \eta}\log 2n }
    + 
\dfrac{1}{\lambda\eta}
\sqrt{\frac{d_n\log^2 p}{n}}\right).
\end{align*}
with probability at least
\begin{align*}
&1-{\rm e}^{d_n\log \frac{ep}{d_n}}
\left(
2{\rm e}^{
-c\left(
\frac{d_n\log^4p}{\gamma_3^2}
\wedge
\frac{d_n\log^2p}{\gamma_3}
\right)
}
\right)
-q_n
-\check{q}_n-{\rm e}^{-p}
-2{\rm e}^{ 2d_n \log \frac{e p}{d_n}}  \left(
{\rm e}^{-p} + 4 {\rm e}^{-  \frac{d_n\log^2p}{(\sqrt{n}+ 3\sqrt{p})^2 \rho_{\max} }}\right)\\
\ge&~
1-
{\rm e}^{-d_n\log p}-q_n-\check{q}_n
\end{align*}
provided ${\rm e}\le d_n\le \frac{p}{C}$ for a sufficiently large constant $C>0$, where we use the definition of $\gamma_3$, $\gamma_4$, $\gamma_5$, $\gamma_6$ along with the fact that $\rho_{\max}=O(p^{-1})$.

We now return to equations \eqref{eq:psi1-term1} and \eqref{eq:uplifting:diff}, along with the bound derived immediately above, to obtain
\begin{align*}\label{eq:qformlem-messy}
    &\bx_{i,\calB_{1,+}}^{\top}
    \bA^{-1}\bB\bB^{\top}\bA^{-1}
    \bx_{i,\calB_{1,+}}\\
    \le&~
    \|\bx_i\|^2(\omega^s)^2
    \left(
    \dfrac{\gamma_0^s(\alpha)
    }{\lambda\eta}
    +(\gamma_0^s(\alpha))^2
    +(\gamma_1^s(\alpha))^2
    \right)\\
    &~+\dfrac{d_n\log^2p}{p-d_n}
    +
    \dfrac{ (\sqrt{n}+3\sqrt{p})^2d_n\rho^2_{\max}}{4\eta^2\lambda^2}\\
    &~+
    \frac{\polylog(n) \|\bx_i\|^3 \|\bX\|^2\omega^s}{ 
    4\lambda^3 \eta^3} \left(
1+
\dfrac{\omega^s}{2\lambda\eta}
+ \frac{\omega^s\|\bX\| \polylog(n)}{8 \lambda^3 \eta^3} \right)
\nonumber\\
&\times \left(
    \sqrt{ \frac{(\sqrt{n}+ 3\sqrt{p})^2 \rho_{\max}}{n \lambda \eta}\log 2n }
    + 
\dfrac{1}{\lambda\eta}
\sqrt{\frac{d_n\log^2 p}{n}}\right)\numberthis
\end{align*}
with probability at least $1-
p^{-d_n}-q_n-\check{q}_n$. We now simplify the above bound by plugging in high probability bounds on the respective parameters. To this end, we first state:
\begin{itemize}
\item By Assumption A1, $\rho_{\max}=\sigma_{\max}(\Sigma)\le C_X/p$.

\item By Lemmas \ref{lem:xi-row-conc} and \ref{lem:maxsingularvalue}, 
$$
\max_{1\le i\le n}\|\bx_i\|\le 2\sqrt{C_X}
\text{ and }
\|\bX^{\top}\bX\|\le (\sqrt{\gamma_0}+3)^2C_X
$$
with probability at least $1-(n+1){\rm e}^{-p/2}$.

\item Next,
\begin{align*}
 \omega_s=&~ \|\bX^{\top}\bX\|
 \sup_{t\in[0,1]}
 \max_{1\le i \le n}
 \ddot{\ell}(t\hat{\bbeta}^{\alpha}
 +(1-t)
 \hat{\bbeta}_{/i}^{\alpha}
 )\\
 \le&~  (\sqrt{n}+3\sqrt{p})^2\rho_{\max}\polylog(n)\\
 \le&~  C_X(\sqrt{\gamma_0}+3)^2\polylog(n)   
\end{align*}
with probability at least $1-{\rm e}^{-p}$, by Lemma~\ref{lem:maxsingularvalue} and Assumption A4.

\item By equation~\eqref{eq:def-gamma0s}, we have
\begin{align*}
    \gamma_0^s(\alpha)\le \dfrac{1}{p}
\end{align*}
provided $\alpha\ge Cp/\lambda^2\eta(1-\eta)\kappa_0$ for sufficiently large $p$ and a sufficiently large constant $C>0$.

\item By equation~\eqref{eq:def-gamma1s} and the above inequalities, we obtain
\begin{align*}
\gamma_1^s(\alpha)
\le&~ \dfrac{16}{\lambda\alpha(1-\eta)\kappa_0}
\left(
1+\dfrac{1}{\lambda\eta}( C_X)^2(\sqrt{\gamma_0}+3)^4
\polylog(n)
\right)
\le~\dfrac{1}{p}
\end{align*}
provided $\alpha\ge Cp\polylog(n)/\lambda^2\eta(1-\eta)\kappa_0$ for sufficiently large $p$ and a sufficiently large constant $C>0$.
\end{itemize}

Plugging in all these bounds into \eqref{eq:qformlem-messy} we conclude that there is a sufficiently large numerical constant $C>0$, depending only on $\gamma_0$, $C_X$ and $ $ such that
\begin{align*}
   \bx_{i,\calB_{1,+}}^{\top}
    \bA^{-1}\bB\bB^{\top}\bA^{-1}
    \bx_{i,\calB_{1,+}}
    \le&~
    \dfrac{\polylog(n)}{\lambda^3\eta^3(1\wedge\lambda\eta)^3}
    \sqrt{\dfrac{d_n}{n\lambda\eta}}
    +\dfrac{Cd_n}{n\lambda^2\eta^2}
    +\dfrac{d_n\log^2p}{p-d_n}
    +
    \dfrac{C}{p\lambda\eta}
\end{align*}
with probability at least $1-
(n+1){\rm e}^{-\frac{p}{2}}
-
p^{-d_n}-q_n-\check{q}_n$. \qed



\subsection{Proof of Lemma~\ref{lem:small-nzro-coef} }\label{ssec:proof:lemma5}
We prove the lemma only for $\hbbeta$, because the same proof and constants apply to the leave-one-out estimators $\hbbeta_{/i}$. 

Recall that in Lemmas \ref{lem:wasserstein-bd-MM18}, \ref{lem:num-large-subgrad_MM18} and \ref{lem:sparsity-bd-MM18}, $\hat{\mu}$ is the empirical distribution of $\hbbeta$, $\mu^*$ is the distribution of $\frac{b_*}{b_* + 2\lambda\eta\tau_*}{\rm soft}(\tau_*Z+\Theta, \frac{\lambda(1-\eta)\tau_*}{b_*})$ with $(\Theta, Z)\sim \frac{1}{p}\displaystyle\sum_{k=1}^p \delta_{\beta^*_k} \otimes N(0,1)$, and $(b_*, \tau_*)$ is the unique solution of the equations (\ref{eq:tau_*}) and (\ref{eq:b_*}).
Note that by Lemma~\ref{lem:bd_tau_b}, there exists $0<b_{\min}<b_{\max}$ and $\sigma<\tau_{\max}$ depending only on model parameters such that $\sigma<\tau_*<\tau_{\max}$ and $b_{\min}<b_*<b_{\max}$.

Let us first define a smoothed indicator function as follows:
\begin{eqnarray}	
	h_\zeta(x)
    =\begin{cases}
	0 &,\abs{x}\geq \kappa_1+\zeta\\
	1-\dfrac{\abs{x}-\kappa_1}{\zeta} &,\kappa_1\le \abs{x}< \kappa_1+\zeta\quad\\
	1 &,\abs{x}< \kappa_1
    \end{cases}
\end{eqnarray}
Note that $h_{\zeta}$ is $\frac{1}{\zeta}$-Lipschitz, and $\mathbbm{1}_{[-\kappa_1,\kappa_1]}\leq h_{\zeta}\leq \mathbbm{1}_{[-\kappa_1-\zeta,\kappa_1+\zeta]}$. We then have
	\begin{align}\label{eq:upperbdd:smalcoeff}
		&\frac{1}{p} \# \{ k: 0<\abs{\hbeta_k}\leq \kappa_1\} \nonumber \\
		=&~\hat{\mu}([-\kappa_1,\kappa_1])-\hat{\mu}(\{0\})\nonumber \\
		=&~
		\int \mathbbm{1}_{[-\kappa_1,\kappa_1]}(x)d\hat{\mu}(x)-\hat{\mu}(\{0\})\nonumber \\
		\le&~
		\int h_\zeta(x)d\hat{\mu}(x)-\hat{\mu}(\{0\})\nonumber \\
		\le&~
		\int h_\zeta(x)d\mu_*(x)
		+
		\abs*{
			\int h_\zeta(x)d\mu_*(x)
			-
			\int h_\zeta(x)d\hat{\mu}_*(x)
		}
        -\hat{\mu}(\{0\})\nonumber \\
		\le&~
		\mu_*([-\kappa_1-\zeta,\kappa_1+\zeta])
		-\mu_*(\{0\})
        +
		\abs*{
			\int h_\zeta(x)d\mu_*(x)
			-
			\int h_\zeta(x)d\hat{\mu}_*(x)
		}+
		\abs*{
			\mu_*(\{0\})
			-\hat{\mu}(\{0\})}.
	\end{align}
 We now aim to obtain upper bounds for the final three terms in \eqref{eq:upperbdd:smalcoeff}. First note that $0$ is the unique discontinuity of $\mu_*$, therefore
	\begin{align*}
		\mu_*([-\kappa_1-\zeta,\kappa_1+\zeta])-\mu_*(\{0\})
		\le 2(\kappa_1+\zeta)\max_{x\neq 0}f_{*}(x)
	\end{align*}
	where $f_*$ is the density of the absolutely continuous part of $\mu_*$ w.r.t. the Lebesgue measure. It can be verified directly via definition that
	\begin{align*}
		f_*(x)
  =&~
    \frac{1}{p}\left( \frac{1}{\tau_*} + \frac{2\lambda\eta}{b_*} \right)\sum_{k=1}^p \phi\left( x \left( \frac{1}{\tau_*} 
    +\frac{2\lambda\eta}{b_*} \right) 
    + {\rm sign}(x)\frac{\lambda(1-\eta)}{b_*} - \frac{\beta_k^*}{\tau_*} \right)
	\end{align*}
	where $\phi(\cdot)$ is the standard Gaussian density, and therefore 
 \begin{align*}
 f_*(x)
 \le&~ (\sqrt{2\pi})^{-1}\left( \frac{1}{\tau_*} + \frac{2\lambda\eta}{b_*} \right) 
 \le~ (\sqrt{2\pi})^{-1} \left( \frac{1}{\sigma} + \frac{2\lambda_{max}}{b_{min}} \right) =: C_f    
 \end{align*}
 So we have
 \begin{equation}\label{eq:mustar:bound}
	\mu_*([-\kappa_1-\zeta,\kappa_1+\zeta])-\mu_*(\{0\}) \leq 2 C_f(\kappa_1+\zeta).
 \end{equation}
	To bound the term $\abs*{\int h_\zeta(x)\mu_*(x)-\int h_\zeta(x)\hat{\mu}_*(x)}$ in \eqref{eq:upperbdd:smalcoeff} we use the following approach. First, note that $h_\zeta$ is  $1/\zeta$-Lipschitz. Thus,
	\begin{align}\label{eq:upperbd:secterm:prob}
		\abs*{
			\int h_\zeta(x)\mu_*(x)
			-
			\int h_\zeta(x)\hat{\mu}_*(x)
		}	
		\le&~
		\dfrac{1}{\zeta}
		\sup_{g\in L_1}
		\abs*{
			\int g(x)d\mu_*(x)
			-
			\int g(x)d\hat{\mu}_*(x)
		}\nonumber \\
		=&~ W_1(\hat{\mu},\mu_*)/\zeta
		\le W_2(\hat{\mu},\mu_*)/\zeta.
	\end{align}
In these equations, $W_q$ denotes the Wasserstein-q distance between two measures for $q=1,2$. The last inequality uses monotonicity of $W_q$ in $q$: for $0<p<q$:
\begin{align*}
    W_p(\mu,\nu):= \inf_{X\sim \mu, Y\sim\nu} \norm{X-Y}_{\calL_p}
    &\leq \inf_{X\sim \mu, Y\sim\nu} \norm{X-Y}_{\calL_q}
    =: W_q(\mu,\nu).
\end{align*}
Furthermore, by Lemma \ref{lem:wasserstein-bd-MM18}, for any $0<\epsilon<\frac{1}{2}$, we have
 \begin{equation}
 W_2(\hat{\mu},\mu_*)/\zeta
		\le \sqrt{\eps}/\zeta
 \end{equation}
	with probability at least $1-C_1\eps^{-2}{\rm e}^{-c_1p\eps^3(\log\eps)^{-2}}$, for some constants $C_1, c_1$. Combining \eqref{eq:upperbd:secterm:prob} and \eqref{eq:W2dist:bd} we conclude that 
 \begin{equation}\label{eq:W2dist:bd}
\abs*{\int h_\zeta(x)\mu_*(x)-\int h_\zeta(x)\hat{\mu}_*(x)} \leq \sqrt{\eps}/\zeta 
 \end{equation}
with probability at least $1-C_1\eps^{-2}{\rm e}^{-c_1p\eps^3(\log\eps)^{-2}}$.
  
For the last term of \eqref{eq:upperbdd:smalcoeff} we use Lemma \ref{lem:sparsity-bd-MM18} to get: 
\begin{align}
\label{eq:lastpiece:smallnonzeros}
    |\hat{\mu}(0) - \mu_*(0)|= |\|\hbbeta\|_0 - s_*| <\eps'
\end{align}
with probability at least $1-C_2\eps'^{-6}{\rm e}^{-c_2p\eps'^6}$.
	
Using \eqref{eq:upperbdd:smalcoeff}, \eqref{eq:mustar:bound}, \eqref{eq:W2dist:bd}, and \eqref{eq:lastpiece:smallnonzeros} together we have
$$\frac{1}{p}
	\abs*{ \{ k: 0<\abs{\hbeta_k}\leq \kappa_1\}}\le 2C_f(\kappa_1+\zeta)+\sqrt{\eps}/\zeta + \eps' $$
with probability at least $1-C_1\eps^{-2}{\rm e}^{-c_1p\eps^3(\log\eps)^{-2}} - C_2\eps'^{-6}{\rm e}^{-c_2p\eps'^6}$.
We observe first that the bound is minimized at $\zeta=(2C_f)^{-\frac12}\eps^{\frac14}$ without changing the probability. 
	
Next, we set $\eps = 2(c_1p)^{-\frac13}\log p$ and $\eps' = 3^{\frac13}(c_2p)^{-\frac16}(\log p)^{\frac16}$ to get:
\begin{align*}
    &~\abs*{\{ k: 0<\abs{\hbeta_k}\leq \kappa_1  \}} \\
    \le&~
    2 C_f p \kappa_1
    +C_3p^{\frac{11}{12}}(\log p)^{\frac14}
    +C_4p^{\frac56}(\log p)^{\frac16}\\
    \leq&~
    C p \kappa_1
    +Cp^{\frac{11}{12}}(\log p)^{\frac14}\\
    \leq&~  Cp^{\frac{11}{12}}(\log p)^{\frac14}
\end{align*}
with probability at least $1-C'p^{-7}$, for sufficiently large $p$ and some constants $C, C'$, provided $\kappa_1 = o(p^{-\frac{1}{12}}(\log p)^{\frac14})$ By an identical argument, an analogous statement holds for $\hat{\beta}_{/i,k}$. The proof is completed by a union bound over $1\le i \le n$. \qed

\subsection{Proof of Proposition \ref{prop:assumptionA4}} \label{proof:prop:A4}
Without loss of generality we assume $C=1$. The first step is to show that $\bx_i^\top\hbbeta_{/i}=O_p(\polylog(n))$. 
  Throughout this proof we use the following notations:
  \begin{align}
h(\bbeta)&=\sum_{j=1}^n\ell(y_i;\bx_i^{\top}\bbeta)
+\lambda(1-\eta)\sum_{i=1}^p|\beta_i|
+\lambda\eta\bbeta^{\top}\bbeta, \nonumber \\
h_{\slash i}(\bbeta)&=\sum_{j\neq i}^n\ell(y_i;\bx_i^{\top}\bbeta)
+\lambda(1-\eta)\sum_{i=1}^p|\beta_i|
+\lambda\eta\bbeta^{\top}\bbeta, \nonumber \\
	h_{\alpha}(\bbeta)&=
	\sum_{j=1}^n\ell(y_i;\bx_i^{\top}\bbeta)
	+\lambda r_{\alpha} (\bbeta), \nonumber\\
	h_{\alpha, \slash i}(\bbeta)&=
	\sum_{j\neq i}^n\ell(y_i;\bx_i^{\top}\bbeta)
	+\lambda r_{\alpha} (\bbeta),
\end{align}
where $r_\alpha(\cdot)$ is define in \eqref{eq:r_alpha:def}.
    \begin{enumerate}
        \item[(a)] First note that for all $i$,
        \begin{align*}
            \lambda\eta \|\hbbeta_{/i}\|_2^2\le&~\sum_j \ell(y_j|\bx_j\top\hbbeta_{/i}) + \lambda(1-\eta)\|\hbbeta_{/i}\|_1+\lambda \eta \|\hbbeta_{/i}\|_2^2
            \leq~ \sum_j \ell(y_j|0), 
        \end{align*}
        where the last inequality is due to the fact that $h_{\slash i}(\hbbeta_{/i}) \leq h_{\slash i}(\mathbf{0})$. 
        Under the event that $\forall i, |y_i|\leq \polylog(n)$ which holds with probability at least $1-nq_n^{(y)}$ according to the assumptions,  we have
        \begin{align*}
            &\max_i \|\hbbeta_{/i}\|^2 
            \leq \frac{1}{\lambda\eta}\sum_j \ell(y_j|0)\\
            &\leq \frac{1}{\lambda\eta}\sum_j (|y_j|^m+1)
            \leq n(\polylog(n))^m.
        \end{align*}
        Therefore
        \begin{align}\label{eq:boundforlinearpart}
            \PP&(\max_i |\bx_i^\top\hbbeta_{/i}|>t)
            \leq \sum_i \EE \PP(|\bx_i^\top\hbbeta_{/i}|>t | \bX_{/i}, \by_{/i})\nonumber \\
            &=\sum_i \EE\PP(|N(0,\frac{\|\hbbeta_{/i}\|^2}{n})|>t| \bX_{/i}, \by_{/i})\nonumber \\
            &= \sum_i \EE\PP(|N(0,1)|> \frac{t\sqrt{n}}{\|\hbbeta_{/i}\|} | \bX_{/i}, \by_{/i})\nonumber \\
            &\leq \PP(\max_i \|\hbbeta_{/i}\|^2>n(\polylog(n))^m) \nonumber \nonumber \\
            &+ 
              n\PP\left(|N(0,1)|>\frac{t}{(\polylog(n))^{m/2}}
              \right). 
        \end{align}
        Let $t=(\polylog(n))^{m/2}\cdot 2\sqrt{\log (n)}:=\polylog(n)$.  Then, we can use \eqref{eq:boundforlinearpart} to obtain
        \begin{align}\label{eq:boundlin:leavei}
            \PP&(\max_i |\bx_i^\top\hbbeta_{/i}|>{\polylog}(n) ) \nonumber \nonumber\\
            &\leq nq_n^{(y)} + 2ne^{-\frac12 (2\sqrt{\log(n)})^2}\nonumber\\
            &\leq nq_n^{(y)} + \frac{2}{n}.
        \end{align}
     With a similar strategy we can also prove that for any $\alpha$ we have
       \begin{align}\label{eq:boundlin:leaveismooth}
            \PP(\max_i |\bx_i^\top\hbbeta^{\alpha}_{/i}|>{\polylog}(n) ) \leq nq_n^{(y)} + \frac{2}{n}.
        \end{align}
     
        \item[(b)] Now we set $\alpha =1$ and work within the event $\Xi$ under which all the following hold:
        \begin{enumerate}
        \item $\max_i |y_i|\leq \polylog(n)$, 
        \item $\max_i |\bx_i^\top\hbbeta_{/i}|\leq \polylog(n)$,
        \item $\max_i |\bx_i^\top\hbbeta_{/i}^1|\leq\polylog(n)$.
        \item $\max_i \|\bx_i\|\leq 2\sqrt{C_X}$.
        \end{enumerate}
       Note that by combining the assumption of theorem with \eqref{eq:boundlin:leavei}, \eqref{eq:boundlin:leaveismooth}, and Lemma \ref{lem:xi-row-conc} we have
        $$\PP(\Xi) \geq 1-3nq_n^{(y)}-\frac{4}{n} - ne^{-p/2}.$$
        Under $\Xi$ we have
        \begin{align*}
        \dot{\ell}_i(\hbbeta_{/i}^1) &= \dot{\ell}_i(y_i | \bx_i^\top \hbbeta_{/i}^1) \nonumber \\
        &\leq 1+(\polylog(n))^m + (\polylog(n))^m \nonumber \\
        &=\polylog(n).
        \end{align*}

        Next, consider the following first order optimality conditions for $\hbbeta^1$ and $\hbbeta_{/i}^1$:
        \begin{align*}
            \sum_j \bx_j \dot{\ell}_i(\hbbeta^1) + \lambda \dot{r}_\alpha(\hbbeta^1) =0,\\
            \sum_{j\neq i} \bx_j \dot{\ell}_i(\hbbeta_{/i}^1) + \lambda \dot{r}_\alpha(\hbbeta_{/i}^1) =0
        \end{align*}
        By subtracting the two equations and using mean value theorem we have
        \begin{align}
            \bX^\top \diag (\ddot{\ell}_j(\bxi)) \bX (\hbbeta^1 - \hbbeta_{/i}^1) + \lambda \diag(\ddot{r}_\alpha(\bxi))(\hbbeta^1 - \hbbeta_{/i}^1)
            &= -\bx_i \dot{\ell}_i(\hbbeta_{/i}^1),
        \end{align}
        where $\bxi$ and $\tilde{\bxi}$ are  convex combinations of $\hbbeta^1$ and $\hbbeta_{/i}^1$.        Therefore
        \begin{align*}
            \hbbeta_{/i}^1 - \hbbeta^1
            &= \left[
                \bX^\top \diag (\ddot{\ell}_j(\bxi)) \bX  + \lambda \diag(\ddot{r}_\alpha(\tilde{\bxi}))
            \right]^{-1}\bx_i \dot{\ell}_i(\hbbeta_{/i}^1).
        \end{align*}
        Hence, it is straightforward to see that  \begin{align}\label{eq:difference_leaveone}
            \|\hbbeta^1 - \hbbeta_{/i}^1 \|
            &\leq \frac{1}{2\lambda\eta}\|\bx_i\| | \dot{\ell}_i(\hbbeta_{/i}^1)| \nonumber \\
            &\leq \frac{1}{2\lambda\eta} 2\sqrt{C_X} \polylog(n)\nonumber \\
            &= \polylog(n)
        \end{align}
        given that $\frac{1}{\lambda\eta} = O(\polylog(n))$. We can now use \eqref{eq:difference_leaveone} to obtain
        \begin{align}
            \|\hbbeta - \hbbeta_{/i}\| &\leq \|\hbbeta - \hbbeta^1 \| + \|\hbbeta^1 - \hbbeta_{/i}^1\|+ \|\hbbeta_{/i}^1 -\hbbeta_{/i} \| 
            \leq \polylog(n)+2\sqrt{\frac{4p\log(2)}{\eta}}.
        \end{align}
        Recall that 
        \[
            \calD := \cup_{1 \le i\le n, t \in [0,1]} \calB(t\hbbeta +(1-t)\hbbeta_{/i}, \starepsilon),
        \]
        where $\calB(x,r)$ is the ball with center $x$ and radius $r$. Hence we can  conclude that $\forall i, \forall \bv \in  \calD$:
        \begin{align}
        \|\bv-\hbbeta_{/i}\|
        \leq&~ \|\bv-\hbbeta\| + \|\hbbeta - \hbbeta_{/i}\|
        \leq~ \starepsilon+\polylog(n)+8\sqrt{\frac{p\log(2)}{\eta}},
        \end{align}
        and 
        \begin{align*}
            |\bx_i^\top\bv|
            \leq&~ |\bx_i^\top\hbbeta_{/i}| + |\bx_i^\top(\bv-\hbbeta_{/i})|\\
            \leq&~ \polylog(n) + \|\bx_i\| \|\bv-\hbbeta_{/i} \|\\
            \leq&~ \polylog(n) + 2\sqrt{C_X} (\starepsilon+2\polylog(n))+8\sqrt{\frac{4C_x p\log(2)}{\eta}}\\
            =&~ \polylog(n).
        \end{align*}
        Hence,
        \begin{align*}
        \dot{\ell}_i(\bv)&\leq 1+ |y_i|^m + |\bx_i^\top\bv|^m \nonumber \\
        &\leq 1+ (\polylog(n))^m + (\polylog(n))^m\nonumber \\
        &=\polylog(n)
        \end{align*}
        Using the same arguments one can show that
        \begin{align*}
            \ddot{\ell}_i(\bv)\leq\polylog(n);
            \dddot{\ell}_i(\bv)\leq\polylog(n).
        \end{align*}
        This completes the proof. \qed
    \end{enumerate}

\section{Conclusion}\label{sec:conc}

In this paper, we have introduced a novel theoretical framework that offers error bounds for the disparity between the computationally intensive leave-one-out risk estimate (LO) and its more computationally efficient approximation (ALO). We focus in a regime where $n/p$ and SNR remain fixed and bounded regardless of how large $n$ and $p$ growth. For problems in the generalized linear model family such as linear Gaussian, Poisson and logistic, we bound the error between ALO and LO in terms of intuitive metrics such as perturbation size of leave-$i$-out active sets. Next, for least squares problems with elastic-net regularization, we show that these purturbations scales sub-linearly with $n$ and $p$, and consequently, the difference $|{\rm ALO-LO}|$ approaches zero as $n,p \rightarrow \infty$.



\section*{Acknowledgments}
Arian Maleki would like to thank NSF (National Science Foundation) for their generous support through grant number DMS-2210506. Kamiar Rahnama Rad would like to thank NSF (National Science Foundation) for their generous support through grant number DMS-1810888.

\bibliographystyle{plainnat}
\bibliography{references}


\appendix

\section{Study of the Elastic Net Estimator}\label{app-enet}
\subsection{Objective}
As mentioned in Section 3 of the main part of the paper, our proofs use concentration of measure results for the empirical distribution of the regression coefficeints and the subgradient vector. Some of the results we use in our paper are due to \cite{miolane2021distribution}. However, the results of \cite{miolane2021distribution} are stated for the LASSO estimator and not the elastic-net. Hence, the results we require are different from those presented in \cite{miolane2021distribution}. However,  the changes do not constitute significant advancements that would warrant the derivation of elastic-net results as a major contribution. As a result, we have included these findings in a dedicated section, which will serve as an online appendix to our paper. This section will not be part of the formal submission to a journal but is included for the sake of completeness. 

Throughout this appendix we will mainly focus on Theorem 3.1, Theorem E.5 and Theorem F.1 of \cite{miolane2021distribution}. For the sake of brevity we do not present the proof with every details, since the proof technique is very similar to that of \cite{miolane2021distribution}. We only focus on the differences.
Consider the elastic net problem:
$$\hbbeta := \underset{\bbeta \in \RR^p}{\argmin} 
~\cL(\bbeta)
    = \underset{\bbeta \in \RR^p}{\argmin}
    ~\frac{1}{2n}\norm{\by - \bX\bbeta}_2^2 
        + \frac{\lambda}{n}\left( \norm{\bbeta}_1- \norm{\bbeta^*}_1 \right)
        + \frac{\eta}{n}\left( \norm{\bbeta}_2^2- \norm{\bbeta^*}_2^2 \right). $$
        Note that to simplify the proof we have subtracted $\|\bbeta^*\|_1$ and $\|\bbeta^*\|_2^2$ and used a different scaling than the one presented in the paper. However, as is obvious, these changes do not have any effect on $\hat{\bbeta}$. Furthermore, the definition here is slightly different from the one in the main paper, where the LASSO penalty is $\lambda(1-\eta)$, and the ridge penalty is $\lambda\eta$. The difference is not substantial and is only for notational brevity.
    Let $\hat{\bw} = \hbbeta - \bbeta^*$ denote the estimation error. Using the assumption $\by=\bX\bbeta^* + \sigma\bz$, the problem can be written as:
\begin{align}
    \hat{\bw} &= \underset{\bw \in \RR^p}{\argmin}\; \cC(\bw)\nonumber \\
    &= \underset{\bw \in \RR^p}{\argmin} \frac{1}{2n}\norm{\sigma\bz - \bX\bw}_2^2
        +\frac{\lambda}{n}\left( \norm{\bw+\bbeta^*}_1- \norm{\bbeta^*}_1 \right)
        + \frac{\eta}{n}\left( \norm{\bw+\bbeta^*}_2^2- \norm{\bbeta^*}_2^2 \right). \label{eq:defineCW}
\end{align}
 We assume there exist $0<\lambda_{\min}<\lambda_{\max}$, $\eta_{\max}>0$ such that $\lambda \in [\lambda_{\min},\lambda_{\max}]$ and $\eta \in (0,\eta_{\max}]$. 

Suppose that we are interested in the asymptotic distribution of the elements of $\hat{\bw}$. The main approach that can help us in characterizing the distribution is the convex Gaussian minimax theorem that we would like to introduce briefly next.

\setcounter{section}{1}
\subsection{Convex Gaussian Minimax Theorem }\label{ssec:cgmt:generic}
In this appendix, we require a few notations that we aim to introduce in this section. 
Consider the mean square error of the elastic net, i.e. $\frac{1}{p}\|\hat{\bw}\|_2^2$. It has been shown in \cite{maleki2010approximate, donoho2011noise, donoho2009message, pmlr-v40-Thrampoulidis15, thrampoulidis2018precise} that under the asymptotic settings $n/p \rightarrow \delta$ and under Assumptions B1-B5 of our paper, the mean square error converges to $\delta (\tau_*^2- \sigma^2)$, where $\tau_*^2$ is a saddle point of the following function:
\begin{align}\label{eq:scaler}
\psi(\tau,b) &= \left(\frac{\sigma^2}{\tau}+\tau\right)\frac{b}{2} - \frac{b^2}{2} + \frac1n\EE \min_{\bw\in\RR^p}\left\{ \frac{b}{2\tau}\norm{\bw}_2^2 - b \bg^\top\bw + \lambda\left( \norm{\bw+\bbeta^*}_1- \norm{\bbeta^*}_1 \right)\right\} \nonumber \\
&+ \eta\left( \norm{\bw+\bbeta^*}_2^2- \norm{\bbeta^*}_2^2 \right). 
\end{align}
Define
$$\hat{\bw}^f(\tau,b) := \frac{b}{b+2\eta\tau}
{\rm soft}
\left(\tau \bg + \bbeta^*,\frac{\lambda\tau}{b}\right)
- \bbeta^*$$
where $\bg\sim N(0,I_p)$ and $$[{\rm soft}(\bx,r)]_i = ((|x_i|-r)_+\sgn(x_i))_{i=1}^p$$ is the element-wise soft thresholding function. Let $(\tau_*, b_*)$ denote the unique saddle point of $\psi(\tau,b)$. Define
\[
\hat{\bw}^f:=\hat{\bw}^f(\tau^*,b^*)
\]
The following theorems state that the saddle point exists, is unique, and is bounded.

\begin{lemma}
$\max_{b\geq 0}\min_{\tau\geq \sigma}\psi(\tau,b)$ is achieved at a unique couple $(\tau_*, b_*)$ which is also the unique solution of the following system:
\begin{align}\label{eq:fix point eqs}
    \left\{
    \begin{array}{l}
    \tau^2 = \sigma^2 + \frac1n \EE \norm{\hat{\bw}^f}_2^2,\\
    b = \tau - \frac1n \EE g^\top \hat{\bw}^f.
    \end{array}
    \right.
\end{align}
\label{lem: fix point equations}
\end{lemma}

\begin{lemma}\label{lem:bd_tau_b}
There exist $b_{min}>0$, $\tau_{max}>0$, $b_{max}>0$ that depend only on $\delta,  \sigma,\xi$ such that $b_{min}\leq b_*\leq b_{max}$ and $\sigma<\tau_*\leq \tau_{max}$
\end{lemma}

The proof of these theorems are postponed to Section \ref{subseq:scalar optimization}.

The convex Gaussian minimax framework makes the connection between $\frac{1}{p} \| \hat{\bw}\|_2^2$ and $\tau_*^2$ through a few steps that we clarify below. Again consider the optimization problem:
\begin{align*}
    \hat{\bw} &= \underset{\bw \in \RR^p}{\argmin}\; \cC(\bw)\\
    &= \underset{\bw \in \RR^p}{\argmin} \frac{1}{2n}\norm{\sigma\bz - \bX\bw}_2^2
        +\frac{\lambda}{n}\left( \norm{\bw+\bbeta^*}_1- \norm{\bbeta^*}_1 \right)
        + \frac{\eta}{n}\left( \norm{\bw+\bbeta^*}_2^2- \norm{\bbeta^*}_2^2 \right).
\end{align*}
Note that we can rewite this optimization problem as the following saddle point problem using dual representation of the $l_2$ norm:
\begin{align*}
     \hat{\bw}:= \arg \underset{\bw \in \RR^p}{\min}  \max_{\bu}   \frac{1}{n}\bu^{\top}  \bX\bw- \frac{1}{2n} \bu^{\top} \bu - \frac1n \bu^{\top} \sigma\bz
        +\frac{\lambda}{n}\left( \norm{\bw+\bbeta^*}_1- \norm{\bbeta^*}_1 \right)
        + \frac{\eta}{n}\left( \norm{\bw+\bbeta^*}_2^2- \norm{\bbeta^*}_2^2 \right).
\end{align*}
According to the Convex Gaussian minimax framework we can construct a simpler auxilary saddle point problem that can provide useful information about $\hat{\bw}$. To clarify this point, define
\begin{eqnarray}
\Phi(X) &:=&  \underset{\bw \in \mathcal{S}_w}{\min}  \max_{\bu \in \mathcal{S}_u}   \frac{1}{n}\bu^{\top}  \bX\bw- \frac{1}{2n} \bu^{\top} \bu - \frac1n \bu^{\top} \sigma\bz
        +\frac{\lambda}{n}\left( \norm{\bw+\bbeta^*}_1- \norm{\bbeta^*}_1 \right) \nonumber \\
        &+& \frac{\eta}{n}\left( \norm{\bw+\bbeta^*}_2^2- \norm{\bbeta^*}_2^2 \right)\\
        &=& \underset{\bw \in \mathcal{S}_w}{\min}  \max_{\bu \in \mathcal{S}_u} \frac{1}{n^{3/2}}\bu^\top (\tilde{\bX}, -\bz)\begin{pmatrix}\bw \\ -\sqrt{n}\sigma \end{pmatrix} - \frac{1}{2n} \bu^{\top} \bu +\frac{\lambda}{n}\left( \norm{\bw+\bbeta^*}_1- \norm{\bbeta^*}_1 \right) \nonumber \\
        &+& \frac{\eta}{n}\left( \norm{\bw+\bbeta^*}_2^2- \norm{\bbeta^*}_2^2 \right)\label{eq:PO}
\end{eqnarray}
where $\calS_w$ and $\calS_u$ are two convex, compact sets. Note that $(\tilde{\bX}, -\bz)$ is a matrix with i.i.d. $N(0,1)$ entries. Define the following auxillary optimization problem:
\begin{eqnarray}
&&\phi(\bg, \bh) \nonumber\\
&: =& \underset{\bw \in \calS_w}{\min}  \max_{\bu \in \calS_u}  \frac{1}{n}\sqrt{\frac{\|\bw\|_2^2}{n}+\sigma^2} \bh^\top \bu - \frac{1}{n^{3/2}}\|\bu\| \bg^{\top} \bw + \frac1n \|\bu\|g'\sigma -\frac{1}{2n} \bu^{\top} \bu
         \nonumber \\
        &&
        +\frac{\lambda}{n}\left( \norm{\bw+\bbeta^*}_1- \norm{\bbeta^*}_1 \right)
        + \frac{\eta}{n}\left( \norm{\bw+\bbeta^*}_2^2- \norm{\bbeta^*}_2^2 \right)\label{eq:AO}
\end{eqnarray}
where $\bh\sim N(0,\II_n)$, $\bg\sim N(0,\II_p)$, $g'\sim N(0,1)$, and all of them are independent and also independent of $\bz$.

According to the Gaussian minimax theorem, i.e. Theorem 3 of \cite{pmlr-v40-Thrampoulidis15}, when $\calS_w$ and $\calS_u$ are convex compact sets we have 
\begin{eqnarray}\label{eq:CGMT}
\mathbb{P} (\Phi(X) < t) \leq 2\mathbb{P} (\phi(\bg,\bh) < t)\\
\mathbb{P} (\Phi(X) > t) \leq 2\mathbb{P} (\phi(\bg,\bh) > t)
\end{eqnarray}
Using this theorem, and by using a proper choice for $\calS_w$ we can analyze certain properties of $\hat{\bw}$ through the minimizer of the easier auxillary function:
\begin{eqnarray}
L(\bw) &:=& \frac12 \left( \sqrt{\frac{\norm{\bw}_2^2}{n} + \sigma^2} \frac{\norm{\bh}_2}{\sqrt{n}} - \frac1n \bg^\top\bw + \frac{g'\sigma}{\sqrt{n}} \right)_+^2 
+ \frac{\lambda}{n}\left( \norm{\bw+\bbeta^*}_1- \norm{\bbeta^*}_1 \right) \nonumber \\
&&+ \frac{\eta}{n}\left( \norm{\bw+\bbeta^*}_2^2- \norm{\bbeta^*}_2^2 \right).
\end{eqnarray}
Note that $L(\bw)$ is obtained from $\phi(\bg, \bh)$, when the supremum with respect to $\bu \in \mathbb{R}^{p}$ is calculated. 

Let us explain how $L(\bw)$ can be connected with the scaler saddle point problem presented in \eqref{eq:scaler}. Actually $\min_{\bw}L(\bw)$ concentrates around $\max_{b\geq 0} \min_{\tau\geq \sigma}\psi(\tau,b)$. Below is a heuristic argument. For simplicity we denote $r(\bw)=\lambda(\|\bw+\bbeta^*\|_1-\|\bbeta^*\|_1)+\eta(\|\bw+\bbeta^*\|_2-\|\bbeta^*\|_2)$


First notice that $\frac{\|\bh\|_2}{\sqrt{n}}$ concentrates around 1, and $\frac{g'\sigma}{\sqrt{n}}$ concentrates around 0, so heuristically, they can be removed from the expression which lead to: 
\begin{align*}
    \min_{\bw}L(\bw) 
    &\approx \min_{\bw} \frac12 \left(\sqrt{\frac{\|\bw\|_2^2}{n}+\sigma^2}-\frac1n \bg^\top\bw\right)_+^2 + \frac1n r(\bw)
\end{align*}
Next we use the fact that $a_+^2 = \max_{b\geq 0}ab-\frac12 b^2$ and $\sqrt{\frac{\|\bw\|_2^2}{n}+\sigma^2}=\min_{\tau\geq \sigma}\frac{\frac{\|\bw\|_2^2}{n}+\sigma^2}{2\tau} + \frac{\tau}{2}$ to obtain
\begin{align*}
    \min_{\bw}L(\bw) 
    &\approx\min_{\bw}\max_{b\geq 0}\min_{\tau\geq \sigma}
    \left( \frac{\frac{\|\bw\|_2^2}{n}+\sigma^2}{2\tau} + \frac{\tau}{2}- \frac1n \bg^\top\bw\right)b-\frac12 b^2 + \frac1n r(\bw)\\
    &\overset{(a)}{=} \max_{b\geq 0}\min_{\tau\geq \sigma}
    \left(\frac{\sigma^2}{\tau}+\tau\right)\frac{b}{2}-\frac12 b^2 + \frac1n \min_{\bw} \left\{ \frac{b}{2\tau}\|\bw\|_2^2 - \bg^\top\bw+r(\bw) \right\}\\
    &:=\max_{b\geq 0}\min_{\tau\geq \sigma} F(\tau, b, \bw)
\end{align*}
Step (a) requires some delicate arguments which are omitted here. They intend to prove that the minimum ad maximum operation are interchangeable. Finally, one can show that $\max\limits_{b\geq 0}\min\limits_{\tau\geq \sigma}F(\tau,b,\bw)$ concentrates around 
\[\max_{b\geq 0}\min_{\tau\geq \sigma}\EE F(\tau,b,\bw)=\max_{b\geq 0}\min_{\tau\geq \sigma}\psi(\tau,b) = \psi(\tau^*,b^*).\]
Therefore, $\min_{\bw}L(\bw)$ concentrates around $\psi(b_*,\tau_*)$.

Now we will use the above arguments for a specific choice of $\mathcal{S}_w$ that allows us to obtain the mean square error of $\hat{\bw}$. For this purpose we define $\mathcal{S}_w = \calD(\eps) := \{ \bw\in\RR^p: \frac1p \norm{\bw - \hat{\bw}^f}_2^2>\eps \}$. We first mention the roadmap of the proof to help the readers navigate through the following theorems:
\begin{itemize}
\item We first show that the minimizer of $L(\bw)$ is with high probability in a ball of radius $\epsilon^2$ around $\hat{\bw}^f$, and hence $\min_{\bw} L(\bw)$ is with high probability very close to $\psi(b_*, \tau_*)$. This is done in Lemma \ref{lem:local stability of L}, Lemma \ref{lem:conc_L(wf)}, and Corollary \ref{col:conc_minL}. 
\item In the next step we use the minimax theorem to show that the minimizer of $\mathcal{C}(\bw)$ defined in \eqref{eq:defineCW} is with high probability in the complement of $\calD(\epsilon)$. This is proved in Lemmas \ref{lem:link_local_stability_AO_PO} and \ref{lem:conc_minC}. One can then use Lemma \ref{lem:conc_minC} to obtain information about $\frac{1}{p} \|\hat{\bw}\|_2^2$, which is the same as the mean square error of $\hat{\bbeta}$. 
\end{itemize}

Our first result aims to connect the minimizer of $L(\bw)$ with $\hat{\bw}^f$.

\begin{lemma}\label{lem:local stability of L} 
There exist
    $\gamma, C,c>0$ depending only on $\Omega$, such that $\forall \eps \in [0,1]$, 
    $$\PP\left( \min_{\bw\in \calD(\eps)}L(\bw) < \min_{\bw\in\RR^p} L(\bw) +\gamma\eps \right)\leq \frac{C}{\eps}e^{-cn\eps^2}$$
    where $\calD(\eps) := \{ \bw\in\RR^p: \frac1p \norm{\bw - \hat{\bw}^f}_2^2>\eps \}$. \footnote{Throughout the discussion, the term `constants' refers to quantities that rely only on the following set of model parameters
$\Omega : = (\delta, \sigma, \xi, \lambda_{\min}, \lambda_{\max}, \eta_{\max})$. Recall that $\delta = n/p$, and $\xi = \frac{1}{\sqrt{p}}\|\bbeta^*\|_2 $; $\lambda_{\min}, \lambda_{\max}$ control the scale of $\lambda$, and $\eta_{\max}$ controls the scale of $\eta$. We do not assume $\eta$ to be bounded away from 0, to be  consistent with our main text.
}
\end{lemma}
The proof is essentially the same as the proof of Theorem B.1 of \cite{miolane2021distribution}, up to minor modifications to some constants. Hence, we do not repeat the proof here..

\begin{lemma}\label{lem:conc_L(wf)}
There exist constants $C,c>0$ such that for all $\epsilon\in [0,1]$,
    $$ \PP \left( | L(\hat{\bw}^f) - \psi(\tau_*,b_*) |>\eps \right) \leq Ce^{-cn\eps^2}$$
\end{lemma}
\begin{proof}
    We only need to prove it for $\eps \leq \eps_0$ for some small constant $\eps_0$ depending only on model parmeters $\Omega$ . The reason is that the probability is non-increasing in $\eps$ so we have a naive bound $Ce^{-cn\eps_0^2}$ for $\eps_0\leq\eps\leq 1$. This flat bound, combined with the sub-Gaussian bound for small $\eps$, is further bounded by $Ce^{-cn\eps_0^2\eps^2}$ for all $\eps \in [0,1]$ and this is the bound we desire. 

    Using the fix point equations \eqref{eq:fix point eqs} we have the following simplification for $\psi(\tau_*,b_*)$:
    \begin{align*}
        \psi(\tau_*,b_*)
        =\frac12 b_*^2 
            + \frac{\lambda}{n}\EE\norm{\hat{\bw}^f+\bbeta^*}_1 
            - \frac{\lambda}{n}\norm{\bbeta^*}_1
            + \frac{\eta}{n}\EE\norm{\hat{\bw}^f+\bbeta^*}_2^2
            - \frac{\eta}{n}\norm{\bbeta^*}_2^2
    \end{align*}    
    Define
    $$\hat{b}^f := \left( \sqrt{\frac{\norm{\hat{\bw}^f}_2^2}{n} + \sigma^2} \frac{\norm{\bh}_2}{\sqrt{n}} - \frac1n \bg^\top\hat{\bw}^f + \frac{g'\sigma}{\sqrt{n}} \right) $$
    And denote $\hat{b}^f_+ = \hat{b}^f \rm 1_{\hat{b}^f>0}$, we have
    \begin{align*}
        |L(\hat{\bw}^f) - \psi(\tau_*,b_*)|
    &\leq\frac12\left| (\hat{b}^f_+)^2 -  b_*^2 \right|\\
    &+ \frac{\lambda}{n}\left|\norm{\hat{\bw}^f+\bbeta^*}_1- \EE\norm{\hat{\bw}^f+\bbeta^*}_1 \right|\\
    &+ \frac{\eta}{n}\left|\norm{\hat{\bw}^f+\bbeta^*}_2^2- \EE\norm{\hat{\bw}^f+\bbeta^*}_2^2\right|\label{eq:three_terms_L_minus_psi}\numberthis
    \end{align*}
    The last two terms are relatively easy to bound. Notice that $\bg\to \hat{\bw}^f=\frac{b_*}{b_*+2\eta\tau_*}{\rm soft}(\tau_*\bg+\bbeta^*,\frac{\lambda\tau_*}{b_*})-\bbeta^*$ is $\tau_{\max}$ Lipschitz, so
    \begin{enumerate}
        \item $\|\hat{\bw}^f+\bbeta^*\|_1$ is $Cn^{-1}$ sub-Gaussian, because $\bg \mapsto \frac1n \|\hat{\bw}^f+\bbeta^*\|_1$ is $Cn^{-1/2}$ Lipschitz.
        \item $\|\hat{\bw}^f+\bbeta^*\|_2^2$ is $Cn^{-1}$ sub-exponential because $\bg \to n^{-1/2}\|\hat{\bw}^f+\bbeta^*\|_2$ is $Cn^{-1/2}$-Lipschitz. Therefore $n^{-1/2}\|\hat{\bw}^f+\bbeta^*\|_2$ is $Cn^{-1}$-sub-Gaussian, and hence $\frac1n \|\hat{\bw}^f+\bbeta^*\|_2^2$ is $Cn^{-1}$ sub-exponential.
    \end{enumerate}
    Therefore for the second and third terms of \eqref{eq:three_terms_L_minus_psi}, we have the following concentrations. For $\eps\in [0,1]$, there exist constants $C,c>0$ such that
    \begin{align}
        \PP\left(\frac{\lambda}{n}\left|\norm{\hat{\bw}^f+\bbeta^*}_1- \EE\norm{\hat{\bw}^f+\bbeta^*}_1 \right|>\eps\right)
        \leq C e^{-cn\eps^2}\label{eq:second_term_L_minus_psi}\\
        \PP\left(\frac{\eta}{n}\left|\norm{\hat{\bw}^f+\bbeta^*}_2^2- \EE\norm{\hat{\bw}^f+\bbeta^*}_2^2\right|>\eps\right)
        \leq C e^{-cn\eps^2}\label{eq:third_term_L_minus_psi}
    \end{align}
    
    Now for the first term $(\hat{b}^f_+)^2 -  b_*^2 $, first we have 
    \begin{align}\label{eq:bhatminusb}
        &|\hat{b}^f - b_*|\nonumber\\
        =&~ \left| \left(\sqrt{\frac{\norm{\hat{\bw}^f}_2^2}{n}+\sigma^2} - \sqrt{\EE\frac{\norm{\hat{\bw}^f}_2^2}{n}+\sigma^2}\right)\frac{\norm{\bh}_2}{\sqrt{n}}+ \tau_*\left( \frac{\norm{\bh}_2}{\sqrt{n}}-1 \right) -   \frac1n (\bg^\top\hat{\bw}^f - \EE\bg^\top\hat{\bw}^f ) + \frac{\sigma g'}{\sqrt{n}}\right|.\nonumber \\
    \end{align}
    We aim to establish concentration results for each of the four terms that appear in \eqref{eq:bhatminusb}. For the first term, we have
    \begin{align*}
        &\PP\left( \left|\sqrt{\frac{\norm{\hat{\bw}^f}_2^2}{n}+\sigma^2}-\sqrt{\EE\frac{\norm{\hat{\bw}^f}_2^2}{n}+\sigma^2}  \right|\frac{\norm{\bh}_2}{\sqrt{n}}>\frac{\eps}{4} \right)\\
        \leq&~ \PP (\norm{\bh}_2>2\sqrt{n}) + \PP\left( \left| \sqrt{\frac{\norm{\hat{\bw}^f}_2^2}{n}+\sigma^2}-\sqrt{\EE\frac{\norm{\hat{\bw}^f}_2^2}{n}+\sigma^2} \right| >\frac{\eps}{8}\right)\\
        \leq&~ Ce^{-cn} + \PP\left(\frac1n\left| \norm{\hat{\bw}^f}_2^2 - \EE\norm{\hat{\bw}^f}_2^2 \right| >\frac{\eps}{b}\cdot 2\sigma \right)\\
        \leq&~ Ce^{-cn} + Ce^{-cn\eps^2} 
    \end{align*}
    for small $\eps$. The second inequality above uses sub-Gaussian concentration of $\norm{\bh}_2$, and the fact that $\sqrt{x+\sigma^2}$ is $\frac{1}{2\sigma}$ Lipschitz in $x$. To obtain the last inequality first note that $\bg\to\|\hat{\bw}^f\|$ is $\tau_{max}$-Lipschitz. Hence, $\frac{\|\hat{\bw}^f\|_2}{\sqrt{n}}$ is $C/n$ sub-Gaussian, and $\frac1n \|\hat{\bw}^f\|_2^2$ is $C/n$ sub-exponential. We can then use Bernstein inequality of sub-exponential random variables (e.g. Theorem 2.8.1 of \cite{vershynin2018high}) to establish the last inequaltiy. 
    
    This implies that 
    $$\PP(\hat{b}^f<0) \leq \PP(|\hat{b}^f-b_*|>b_*)\leq Ce^{-cn}$$
    and similarly
    $$\PP(\hat{b}^f>b_{max}+1)\leq Ce^{-cn}$$
    so
    $$\PP(|\hat{b}^f_+ - b_*|>\eps)\leq\PP(\hat{b}^f<0) + \PP(|\hat{b}^f_ - b_*|>\eps)\leq Ce^{-cn\eps^2} $$
    then we have
    \begin{align*}\label{eq:first_term_L_minus_psi}
        \PP\left(\frac12 |(\hat{b}^f_+)^2-b_*^2|>\frac{\eps}{4}\right)
        &= \PP\left( |\hat{b}^f_+ - b_*|\cdot |\hat{b}^f_+ + b_*|  >\frac{\eps}{2} \right)\\
        &\leq \PP(\hat{b}^f>b_{max}+1) + \PP\left(|\hat{b}^f_+ - b_*| >\frac{\eps}{2(2b_{max}+1)} \right)\\
        &\leq Ce^{-cn} + Ce^{-c\eps^2}\\
        &\leq Ce^{-cn\eps^2}\numberthis
    \end{align*}
  for small $\eps$. Now inserting  \eqref{eq:second_term_L_minus_psi}, \eqref{eq:third_term_L_minus_psi}, and \eqref{eq:first_term_L_minus_psi} back into \eqref{eq:three_terms_L_minus_psi} we have our final result
    
    $$\PP(| L(\hat{\bw}^f) - \psi(\tau_*,b_*) | >\eps)\leq Ce^{-cn\eps^2}$$
    for small $\eps$. 
\end{proof}

\begin{lemma}\label{lem:L(wf)_near_minL}
For all $R>0$ there exists constants $C,c>0$ that only depend on $(\Omega,R)$ such that for all $\eps \in (0,1]$,
\[
    \PP(L(\hat{\bw}^f)>\min_{\|\bw\|_2\leq \sqrt{n}R}L(\bw) + \eps)\leq \frac{C}{\eps}e^{-cn\eps^2}
\]
\end{lemma}
\begin{proof}
    The proof is essentially the same as that of Proposition B.2 in \cite{miolane2021distribution} and thus omitted.
\end{proof}

\noindent We would now like to combine Lemma \ref{lem:local stability of L}, Lemma \ref{lem:conc_L(wf)} and Lemma \ref{lem:L(wf)_near_minL} to prove the following result:

\begin{corollary} \label{col:conc_minL}
There exist $C,c>0$ depending only on $\Omega$ such that for all $\eps \in [0,1]$, 
    $$\PP (|\min_{\bw}L(\bw) - \psi(b_*,\tau_*)|>\eps) \leq \frac{C}{\eps}e^{-cn\eps^2} $$
\end{corollary}
\begin{proof}
    $L(\bw)$ is $\frac{2\eta}{n}$-strictly convex and $L(\bw)\to+\infty$ when $\norm{\bw}_2\to+\infty$. Therefore $L(\bw)$ possesses a unique global minimizer $\bw^*$. 

    By Lemma \ref{lem:local stability of L}, the event $\{ \frac1p \norm{\bw^* - \hat{\bw}^f}_2^2\leq 1 \}$ has probability at least $1-Ce^{-cn}$. On this event we have
    $$\norm{\bw^*}_2\leq \norm{\hat{\bw}^f}_2 + \norm{\bw^* - \hat{\bw}^f}_2\leq\norm{\hat{\bw}^f}_2+\sqrt{p}  $$

    Recall that $\frac{\|\hat{\bw}^f\|}{\sqrt{n}} $ is $C/n$ sub-Gaussian so $\PP(\frac{1}{\sqrt{p}}\norm{\hat{\bw}^f}_2 \leq\frac{1}{\sqrt{p}}\EE\norm{\hat{\bw}^f}_2+1)\geq 1-Ce^{-cn} $. Therefore
    \begin{align*}
        \norm{\bw^*}_2&\leq \EE\norm{\hat{\bw}^f}_2 + C\sqrt{n}\\
        &\leq \sqrt{\EE \norm{\hat{\bw}^f}_2^2} + C\sqrt{n}\\
        &\leq \sqrt{n(\tau_{max}^2 - \sigma^2)} + C\sqrt{n}\\
        &=C\sqrt{n}\label{eq:conc_minL_C}\numberthis
    \end{align*}
    where the second inequality uses Jensen's Inequality, and the third uses Lemma \ref{lem: fix point equations}.

    Let event $A$ be the intersection of above events, i.e. $A:=\{\frac1p \norm{\bw^* - \hat{\bw}^f}_2^2\leq 1, \; \frac{1}{\sqrt{p}}\norm{\hat{\bw}^f}_2 \leq\frac{1}{\sqrt{p}}\EE\norm{\hat{\bw}^f}_2+1\}$ then $A$ has probability at least $1-Ce^{-cn}$. On event $A$ we have:
    
    \[
        \min_{\bw}L(\bw) = \min_{\norm{\bw}_2\leq C\sqrt{n}} L(\bw).
    \]
    This means
    \begin{align}
        \PP\left(L(\hat{\bw}^f) > \min_{\bw}L(\bw) + \frac{\eps}{2}
        \right) \leq \PP\left(L(\hat{\bw}^f) > \min_{\norm{\bw}_2\leq C\sqrt{n}}L(\bw) + \frac{\eps}{2}\right) + Ce^{-cn}
        \label{eq:conc_minL_1}
    \end{align}
    So we have
    \begin{align*}
        &\PP (|\min_{\bw}L(\bw) - \psi(b_*,\tau_*)|>\eps)\\
        \leq&~ \PP\left(|\min_{\bw}L(\bw) - L(\hat{\bw}^f)|>\frac{\eps}{2}\right) + \PP\left(|L(\hat{\bw}^f) - \psi(b_*,\tau_*) |>\frac{\eps}{2}\right)\\
        \leq&~ \PP\left(L(\hat{\bw}^f) > \min_{\bw}L(\bw) + \frac{\eps}{2}
        \right) + Ce^{-cn\eps^2}\\
        \leq&~ Ce^{-cn\eps^2} + Ce^{-cn}  + \PP\left(L(\hat{\bw}^f) > \min_{\norm{\bw}_2\leq C\sqrt{n}}L(\bw) + \frac{\eps}{2}\right)\\
        \leq&~Ce^{-cn\eps^2} + Ce^{-cn} + Ce^{-cn\eps^2}\\
        \leq&~ Ce^{-cn\eps^2}
    \end{align*}
    for small $\eps$. The second inequality uses Lemma \ref{lem:conc_L(wf)} and the fact that $L(\hat{\bw}^f) > \min_{\bw}L(\bw)$. The third inequality use \eqref{eq:conc_minL_1}. The penultimate inequality uses Lemma \ref{lem:L(wf)_near_minL} with $R$ chosen to be the constant $C$ in \eqref{eq:conc_minL_C}. Note that all above constants $C,c$ may vary line by line but depend only on model parameters in $\Omega$.
\end{proof}

\begin{lemma}\label{lem:conc_minC}
    There exists $C,c>0$ depending only on $\Omega$ such that for all $\eps\in [0,1]$, 
    \[
        \PP\left(\left| \min_{\bw}\cC(\bw) - \psi(\tau_*, b_*) \right|\geq \eps\right)\leq \frac{C}{\eps}e^{-cn\eps^2}
    \]
\end{lemma}
\begin{proof}
    Using the convex Gaussian minimax theorem stated in \eqref{eq:CGMT} and Corollary \ref{col:conc_minL}, we have
    \begin{align*}
        \PP\left(\left| \min_{\bw}\cC(\bw) - \psi(\tau_*, b_*) \right|\geq \eps\right)
        &\leq 2 \PP\left(\left| \min_{\bw}\L(\bw) - \psi(\tau_*, b_*) \right|\geq \eps\right)
        \leq \frac{C}{\eps}e^{-cn\eps^2}. 
    \end{align*}
\end{proof}

\begin{lemma}\label{lem:link_local_stability_AO_PO}
There exists constants $C,c$ depending only on $\Omega$ such that for all closed set $D\in \RR^p$, $\forall \eps \in (0,1]$:
\[
\PP(\min_{\bw\in \calD_{\epsilon}}\cC(\bw)\leq \min_{\bw}\cC(\bw)+\eps)
\leq 2\PP(\min_{\bw\in \calD_{\epsilon}}L(\bw)\leq \min_{\bw}L(\bw)+3\eps)+\frac{C}{\eps}e^{-cn\eps^2}
\]
\end{lemma}
\begin{proof}
    The proof is essentially the same as the proof of Proposition C.1 in \cite{miolane2021distribution} and thus omitted here.
\end{proof}

\subsection{The asymptotic distribution and sparsity}

As we mentioned in our paper we need to evaluate a few quantities such as the number of non-zero elements $\hat{\bbeta}$. Clearly, the empirical disribution of $\hat{\bbeta}$ can be used for this purpose. Hence, if we can evaluate what the empirical distribution of $\hat{\beta}$ converges to, then we can hopefully obtain accurate bounds on e.g., $\|\hat{\bbeta}\|_0$. 

Let $\hmu$ be the empirical distribution of the couple $(\hbbeta, \bbeta^*)$. According to the results that have appeared in the approximate message framework and CGMT framework \cite{maleki2010approximate, donoho2011noise, donoho2009message, pmlr-v40-Thrampoulidis15, thrampoulidis2018precise, wang2022does}, we expect $\hmu$ to converge to 
$\mu^*$ that is the distribution of the couple 
$$\Big(\frac{b_*}{b_*+2\eta\tau_*}{\rm soft}\Big(\tau_*Z+\Theta, \frac{\lambda\tau_*}{b_*}\Big), \Theta\Big)$$
where $(Z,\Theta) \sim \calN(0,1)\bigotimes\frac1p \sum \delta_{\beta_k^*}$, and $(\tau_*,b_*)$ is the unique saddle point of $\psi(\tau, b)$ defined in \eqref{eq:scaler}. The following theorem is a finite sample size confirmation of this claim:

\begin{theorem}
    \label{thm:W2 of elasticnet}
    There exists constants $C,c > 0$ depending only on $\Omega$ such that for all $\eps\in (0,\frac12]$,
    $$\PP\left( W_2(\hmu,\mu^*)^2\geq \eps \right)\leq C \eps^{-2} e^{cp\eps^3 (\log \eps)^{-2}}$$
\end{theorem}
\begin{proof}
The proof of this Theorem is very similar to the proof of Theorem 3.1 of \cite{miolane2021distribution}, and hence will be skipped. Please note that there will be some minor changes due to the fact that our regulaizer is elastic net, i.e. $\lambda \|\bbeta\|_1+ \eta \|\bbeta\|_2^2$ compared to LASSO in \cite{miolane2021distribution}, i.e. $\lambda\|\bbeta\|_1$. For that reason as we described in Section \ref{subseq:scalar optimization}, our scaler optimization function $\psi(\tau,b)$ is slightly different from the corresponding function in \cite{miolane2021distribution}. 
\end{proof}

As we described before, one of our goals is to use $\hat{\mu}$ to evaluate the properties of $\hat{\bbeta}$. Hence, in most of our results we are more interested in the empirical law of $\hat{\bbeta}$, denotes as $\hmu_1$, rather than $\hmu$. However, it turns out that we can simply obtain a bound for $W_2(\hmu_1,\mu_1^*)^2$, where  $\mu^*_1$ is the law of $\frac{b_*}{b_*+2\eta\tau_*}soft(\tau_*Z+\Theta, \frac{\lambda\tau_*}{b_*})$ using Theorem 
\ref{thm:W2 of elasticnet}.

\begin{corollary}
    There exists constants $C,c > 0$ depending only on $\Omega$ such that for all $\eps\in (0,\frac12]$,
    $$\PP\left( W_2(\hmu_1,\mu_1^*)^2\geq \eps \right)\leq C \eps^{-2} e^{cp\eps^3 (\log \eps)^{-2}}$$
\end{corollary}
\begin{proof}
We have
    \begin{align}\label{eq:prop:wasserstein}
        W_2^2(\hmu,\mu^*) &= \underset{\scriptsize
        \begin{tabular}{c}
             $(X_1,X_2)\sim \hmu$  \nonumber \\
             $(Y_1,Y_2)\sim \mu^*$ 
        \end{tabular}}{\inf} \EE [((X_1-Y_1)^2 + (X_2-Y_2)^2]\nonumber \\
        &\geq \underset{\scriptsize
        \begin{tabular}{c}
             $(X_1,X_2)\sim \hmu$ \nonumber \\
             $(Y_1,Y_2)\sim \mu^*$ 
        \end{tabular}}{\inf} \EE (X_1-Y_1)^2\nonumber \\
        &= \underset{\scriptsize
        \begin{tabular}{c}
             $X_1\sim \hmu_1$ \nonumber \\
             $Y_1\sim \mu^*_1$ 
        \end{tabular}}{\inf}\EE (X_1-Y_1)^2 \nonumber \\
        &= W_2^2(\hmu_1,\mu^*_1)
    \end{align}
    Hence,  \eqref{eq:prop:wasserstein} combined with Theorem \ref{thm:W2 of elasticnet} completes the proof. 
\end{proof}

As we discussed in our main paper, we also need to bound the size of sets for which we have bounds on the magnitude of the subgradients of the $\ell_1$-norm. The rest of this section is dedicated to explaining how the sizes of such sets can be bounded.  Define 
$$\hat{\bv}:= \frac{1}{\lambda}\left[ \bX^\top\by - (\bX^\top\bX + \eta \bI_p)\hat{\bbeta} \right]$$
where $\bI_p$ is the identity matrix on $\RR^{p\times p}$. It can be shown that $\hat{\bv}$ is a subgradient of $\norm{\hat{\bbeta}}_1$. In fact, the first order condition of $\hat{\bbeta}$ gives
$$0\in \bX^\top (\by - \bX\hat{\bbeta}) + \lambda \partial \norm{\hat{\bbeta}}_1 + 2\eta\hat{\bbeta}.
$$
By simple algebra this is equivalent to
$\hat{\bv} \in \partial \norm{\hat{\bbeta}}_1.$
Note that if $|\hat{\bv}_i|<1$ , then  $\hat{\bbeta}_i$ has to be zero. Hence, analyzing $\hat{\bv}$ provides an upper bound on $\|\hat{\bbeta}\|_0$. Let $\mu^*_1$ denote  the distribution of 
$$\frac{b_*}{b_*+2\eta\tau_*}{\rm soft}(\tau_*Z+\Theta, \frac{\lambda\tau_*}{b_*})$$
as defined previously. Define 
$$s_* = \mu_*(\{0\}) = \frac1p \sum_{k=1}^p 
    \left[\Phi\left(\frac{\lambda}{b_*} - \frac{\beta_k^*}{\tau_*}\right) - \Phi\left(-\frac{\lambda}{b_*} - \frac{\beta_k^*}{\tau_*}\right)\right]$$
Note that if $\hat{\bbeta}_i \neq 0$ then $|\hat{\bv}_i| = 1$. If there are not many zero coefficients with subgradients whose magnitude is close to $1$, we should expect, $\frac1p \sum_i \mathbbm{1}_{\{  |\hat{\bv}_i|= 1 \}}$ to be close to $s^*$. The following theorem confirms this:   

\begin{theorem}
    \label{thm:conc_sparsity_subgrad}
    There exist constants $C,c>0$ depending only on $\Omega$ such that, for all $\eps \in [0,1]$,
    $$\PP\left( \frac1p \sum_i \mathbbm{1}_{\{  |\hat{\bv}_i|\geq 1-\eps \}}\geq s_* + 2(1+\frac{\lambda}{b_{\min}})\eps \right)\leq \frac{C}{\eps^3}e^{-cn\eps^6},$$
    where $b_{\min}>0$ is the lower bound of $b_*$ in Lemma \ref{lem:bd_tau_b}. 
\end{theorem}

\begin{proof}[Sketch of proof]
    The proof is similar to that of Lemma \ref{lem:local stability of L} using the convex Gaussian minimax theorem (CGMT) that was stated in \eqref{eq:CGMT}, and is essentially the same as the proof of Theorem E.5 in \cite{miolane2021distribution}. Therefore we provide a sketch of proof here while omitting the details. First we conctruct the primary optimization (PO) with $\hbv$ being its unique optimizer. Then we identify the auxillary optimization (AO) of CGMT and study the local stability of AO, similar to Lemma \ref{lem:local stability of L}. Finally we use CGMT to connect the local stability of AO to that of PO.
    Define the primary optimizatio (PO) as the following:
    \begin{align*}\label{eq:PO_subgrad}
        \cV(\bv) = \min_{\bw} \frac{1}{2n} \|\bX\bw-\sigma\bz\|_2^2+\frac{\lambda}{n}\bv^\top(\bbeta^*+\bw) - \frac{\lambda}{n}\|\bbeta^*\|_1 + \frac{\eta}{n}\|\bw+\bbeta^*\|_2^2 - \frac{\eta}{n}\|\bbeta^*\|_2^2.
    \end{align*}
    It can be verified using dual norm and interchangeability of min-max that  
    $$\hbv = \argmax_{\|\bv\|_{\infty}\leq 1}\cV(\bv).$$
    Hence, the goal would be to use this optimization and CGMT to provide useful information about $\hbv$. As we described before, we expect $\frac1p \sum_k \rm{1}_{\{ |v_k|\geq 1-\eps \}}$ to be close to the number of nonzero coefficients and that should be close to $s_*$. Hence, we set 
    \[
    D_\eps:=\left\{ \bv: \|\bv\|_{\infty}\leq 1, \frac1p \sum_k \rm{1}_{\{ |v_k|\geq 1-\eps \}}\geq s_*+ 2\left(1+\frac{\lambda}{b_{min}}\right)\eps \right\}.
    \]
    We have 
    \[
    \PP\left( \frac1p \sum_i \mathbbm{1}_{\{  |\hat{\bv}_i|\geq 1-\eps \}}\geq s_* + 2(1+\frac{\lambda}{b_{\min}})\eps \right)
    =\PP(\hbv\in D_\eps)
    \leq \PP\left( \max_{D_\eps} \cV(\bv)\geq \max_{\|\bv\|_{\infty}\leq 1} \cV(\bv)-\eps' \right).
    \numberthis
    \label{eq:local_stability_AO_subgrad}
    \]
    for any $\eps'>0$. Note that $\eps'$ will be decided after the analysis of PO is finished. Using the same arguments as the ones presented  in Section \ref{ssec:cgmt:generic}, we can obtain the following auxillary optimization for this problem:
    \begin{align*}
        V(\bv) =&~
        \min_{\bw} \frac12 \left( \sqrt{\frac{\norm{\bw}_2^2}{n} + \sigma^2} \frac{\norm{\bh}_2}{\sqrt{n}} - \frac1n \bg^\top\bw + \frac{g'\sigma}{\sqrt{n}} \right)_+^2 
        + \frac{\lambda}{n}\left( \bv^\top(\bw+\bbeta^*)- \norm{\bbeta^*}_1 \right) \nonumber \\
        &+ \frac{\eta}{n}\left( \norm{\bw+\bbeta^*}_2^2- \norm{\bbeta^*}_2^2 \right)
        \label{eq:PO_subgrad}\numberthis
    \end{align*}
    where $\bg\sim N(0,\II_p)$, $\bh\sim N(0,\II_n)$ and $g'\sim N(0,1)$, independent with each other. Directly working with $D_\eps$ is quite difficult, but recall in Lemma~\ref{lem:local stability of L} we have defined a set $\{\bw: \frac1p \|\bw-\hat{\bw}^f\|_2^2> \eps\}$ and it was easier to work with. In fact, we can define a similar set 
    \[
        \tilde{D}_\eps:=\{ \bv\in\RR^p: \|\bv\|_{\infty}\leq 1, \frac1p \|\bv-\hbv^f\|_2^2\geq \eps \},
    \]
    with
    \[
        \hbv^f:=-\frac{b_*}{\lambda\tau_*}(\hbw^f-\tau_*\bg)
        -\frac{2\eta}{\lambda}(\hbw^f+\bbeta^*).
    \]
    And one can then show that, for some constants $C,c,\gamma>0$:
    \begin{align}
        \PP(\max_{\bv\in \tilde{D_\eps}}V(\bv)\geq \max_{\bv}V(\bv)-c\eps)\leq \frac{C}{\eps}e^{-cn\eps^2}
        \label{eq:conc_sparsity_subgrad_Dtilde}
    \end{align}

    The proof is essentially the same as that of Theorem E.7 of \cite{miolane2021distribution} and thus omitted here.

    The next step is to substitute the $\tilde{D}_\eps$ in \eqref{eq:conc_sparsity_subgrad_Dtilde} back to $D_\eps$. The goal is to prove that, for some constants $C,c,\gamma>0$, for all $\eps\in (0,1]$:
    \[
        \PP(\max_{\bv\in D_\eps}V(\bv)\geq \max_{\bv}V(\bv)-3\gamma\eps^3)\leq \frac{C}{\eps^3}e^{-cn\eps^6}
        \label{eq:local_stability_AO_subgrad2}
        \numberthis
    \]
    The proof is essentially the same as that of Lemma E.9 of \cite{miolane2021distribution} and thus omitted.
    
    Then we connect \eqref{eq:local_stability_AO_subgrad2} with its $\cV(\bv)$ version via CGMT. Notice that by interchanging min-max and using dual norm expressions, we have
    \[
        \max_{\|\bv\|_{\infty}\leq 1}\cV(\bv) = \min_{\bw}\cC(\bw),\quad 
        \max_{\|\bv\|_{\infty}\leq 1}V(\bv) = \min_{\bw}L(\bw).
    \]
    Hence, we have
    \begin{align*}
        &\PP\left( \max_{D_\eps} \cV(\bv)\geq \max_{\|\bv\|_{\infty}\leq 1} \cV(\bv)-3\gamma\eps^3 \right)\\
        \leq&~ \PP\left( \max_{\|\bv\|_{\infty}\leq 1} \cV(\bv)          <\psi(\tau_*,b_*)-\gamma\eps^3 \right) 
            + \PP\left( \max_{D_\eps} \cV(\bv)\geq \psi(\tau_*,b_*)-2\gamma\eps^3 \right)\\
        \leq&~ \PP\left( \min_{\bw}\cC(\bw)          <\psi(\tau_*,b_*)-\gamma\eps^3 \right) 
            + 2\PP\left( \max_{D_\eps} V(\bv)\geq \psi(\tau_*,b_*)-2\gamma\eps^3 \right).
            \label{eq:subgrad_2}\numberthis
    \end{align*}
    Now by Lemma \ref{lem:conc_minC}, the first term is bounded by $\frac{C}{\eps^3}e^{-cn\eps^6}$. For the second term, we have
    \begin{eqnarray}
        \lefteqn{\PP\left( \max_{D_\eps} V(\bv)\geq \psi(\tau_*,b_*)-2\gamma\eps^3 \right)} \nonumber \\
        &\leq& \PP\left( \max_{D_\eps} V(\bv)\geq \max_{\|\bv\|_{\infty}\leq 1}V(\bv)-3\gamma\eps^3 \right)
            +\PP(\max_{\|\bv\|_{\infty}\leq 1}V(\bv)>\psi(\tau_*,b_*)+\gamma\eps^3)\nonumber\\
        &\leq& \frac{C}{\eps^3}e^{-cn\eps^6} + \PP(\min_{\bw}L(\bw)>\psi(\tau_*,b_*)+\gamma\eps^3)\nonumber \\
        &\leq& \frac{C}{\eps^3}e^{-cn\eps^6} + \frac{C}{\eps^3}e^{-cn\eps^6}\label{eq:subgrad_3}
    \end{eqnarray}
    The penultimate inequality uses \eqref{eq:local_stability_AO_subgrad2} and the last inequality uses Corollary \ref{col:conc_minL}. Putting \eqref{eq:local_stability_AO_subgrad}, \eqref{eq:subgrad_2} and \eqref{eq:subgrad_3} together we have 
    \[
        \PP\left( \frac1p \sum_i \mathbbm{1}_{\{  |\hat{\bv}_i|\geq 1-\eps \}}\geq s_* + 2(1+\frac{\lambda}{b_{\min}})\eps \right)
        \leq \frac{C}{\eps^3}e^{-cn\eps^6}.
    \]
\end{proof}

The final theorem that we would like to mention in the appendix is a concentration result on the number of nonzero elements of $\hat{\bbeta}$. 

\begin{theorem}\label{thm:betahat-l0}
    There exist constants $C,c$ depending only on $\Omega$ such that for all $0<\eps<1$,
    $$\PP\left(  \abs*{ \frac1p \norm{\hbbeta}_0 - s_* }\geq \eps \right)\leq C\eps^{-6}e^{-cp\eps^6}$$
\end{theorem}
\begin{proof}
    Note that if an elements of $\hat{\bbeta}_i \neq 0$, then its corresponding subgradient has to be either $1$ or $-1$. Hence, the upper bound of $\frac1p \|\hbbeta\|_0$ is a direct result of Lemma \ref{thm:conc_sparsity_subgrad}.   For the lower bound, the proof follows a similar PO-AO-local stability path and is essentially the same as the proof of Theorem F.1 in \cite{miolane2021distribution}.
\end{proof}

\setcounter{section}{1}
\subsection{Study of the Scalar Optimization}
\label{subseq:scalar optimization}

 Let $h_r(x)$ be the Huber loss
$$
h_r(x) := \frac{1}{2r}\norm{{\rm soft}(\bx,r)-\bx}_2^2 + \norm{{\rm soft}(\bx,r)}_1.
$$
We remind the reader that
\begin{align}\label{eq:scaler2}
\psi(\tau,b) &= \left(\frac{\sigma^2}{\tau}+\tau\right)\frac{b}{2} - \frac{b^2}{2} \nonumber \\
&+ \frac1n\EE \min_{\bw\in\RR^p}\left\{ \frac{b}{2\tau}\norm{\bw}_2^2 - b \bg^\top\bw + \lambda\left( \norm{\bw+\bbeta^*}_1- \norm{\bbeta^*}_1 \right) 
+ \eta\left( \norm{\bw+\bbeta^*}_2^2- \norm{\bbeta^*}_2^2 \right)\right\},
\end{align}
and that
$$\hat{\bw}^f(\tau,b) := \frac{b}{b+2\eta\tau}
{\rm soft}
\left(\tau \bg + \bbeta^*,\frac{\lambda\tau}{b}\right)
- \bbeta^*$$

Inserting back $\bw = \hat{\bw}^f(\tau,b)$, it can be shown that
$$
\psi(\tau,b) = \left(\frac{\sigma^2}{\tau'} + \frac{b\tau'}{b-2\eta\tau'}\right)\frac{b}{2} - \frac{b^2}{2} - \eta\sigma^2 + \frac{\lambda}{n} \EE h_{\frac{\lambda\tau'}{b}}(\tau'\tilde{\bg} + \bbeta^*) - \frac{\lambda}{n}\norm{\bbeta^*}_1 - \frac{b\tau'}{2n}\EE\norm{\tilde{\bg}}_2^2 
$$
where $\tau' = \frac{b}{b+2\eta\tau}\tau$ and $\tilde{\bg} = \bg - \frac{2\eta}{b}\bbeta^*$.
The variable $\hat{\bw}^f$ and function $\psi(\tau,b)$ play an important role in our analysis later. In this section, we study the saddle point of $\psi(\tau,b)$:
$$(\tau_*, b_*):= \underset{b\geq 0}{\argmax}\;\underset{\tau\geq \sigma}{\argmin} \psi(\tau,b)$$
In the rest of this section we prove Lemma \ref{lem: fix point equations} and Lemma \ref{lem:bd_tau_b}.

\begin{proof}[Proof of Lemma \ref{lem: fix point equations}]

    Note that $\psi$ is convex-concave and differentiable with respect to $(\tau,b)$, and the differentiation can be taken inside the expectation. Using the formulae
    \begin{align*}
        &\frac{\partial}{\partial r} h_r(x) = -\frac{1}{2r^2}[{\rm soft}(x,r)-x]^2\\
        &\frac{\partial}{\partial x} h_r(x) = \frac1r \left(x- {\rm soft}(x,r)\right)
    \end{align*}
   one can obtain the following:
   \begin{align*}
       &\frac{\partial}{\partial \tau}\psi(\tau,b)
       = \frac{b}{2\tau^2}\left(\tau^2 - \sigma^2 - \frac1n \EE\norm{\hat{\bw}^f(\tau,b)}_2^2\right),\\
       &\frac{\partial}{\partial b}\psi(\tau,b)
       =\tau - b - \frac1n \EE \bg^\top\hat{\bw}^f(\tau,b).
   \end{align*}
   First, for each $b\geq 0$ consider $\min_{\tau\geq \sigma}\psi(\tau,b)$. Let $f_b(\tau) = \frac{\partial}{\partial \tau}\psi(\tau,b)$. We have
   \[
        f_b(\tau) = \frac{b}{2}\left(1-\frac{\sigma^2}{\tau^2} - \frac1n \sum_{k=1}^p \EE\left[ \frac{b}{b+2\eta\tau}{\rm soft}(g_k+\frac{\beta_k^*}{\tau},\frac{\lambda}{b})- \frac{\beta_k^*}{\tau}\right]^2\right).
   \]
   Hence, we have
   \begin{itemize}
       \item $f_b(\tau)$ is differentiable since the integrand is almost surely differentiable, and the distribution of $\bg$ is continuous
       \item $f'_b(\tau)>0$ because $\psi(\cdot,b)$ is strictly convex
       \item $f_b(\sigma) = -\frac{b}{2n} \EE\|\hat{\bw}^f(\tau,b)\|^2<0$ since $\hat{\bw}^f(\tau,b)$ is non-degenerate
       \item $f_b(+\infty)=1$ using the Dominated Convergence Theorem. 
   \end{itemize}
   Therefore $\forall b\geq 0$, $\exists$ unique $\tau_0(b)>\sigma$ such that $f_b(\tau_0(b))=0$. Moreover $f_b(\tau)<0$ on the left and $>0$ on the right. Hence, $\tau_0(b)$ is the unique minimizer of $\psi(\tau,b)$ at $b$. Moreover, by the Implicit Function Theorem, $\tau_0(b)$ is differentiable.

   Next we define $G(b) = \psi(\tau_0(b),b) = \min_{\tau\ge \sigma}\psi(\tau,b)$ so that 
   \[
        \max_{b\geq 0}\min_{\tau\ge \sigma}\psi(\tau,b) = \max_{b\geq 0} G(b)
   \]
   Its derivative is given by
   \begin{align*}
       g(b):&=G'(b)\\
       &=\left.\frac{\partial}{\partial b}\psi(\tau,b) \right|_{\tau=\tau_0(b)} + \left.\frac{\partial}{\partial \tau} \psi(\tau,b)\right|_{\tau=\tau_0(b)} \tau_0'(b)\\
       &= \tau_0(b) -b - \frac1n \EE \bg^\top \hat{\bw}^f(\tau_0(b),b)
   \end{align*}
   The last line is because $\left.\frac{\partial}{\partial \tau} \psi(\tau,b)\right|_{\tau=\tau_0(b)}=0$. Next we show that $G(b)$ has a unique maximizer $b_*>0$ and it is also the unique solution of $g(b)=0$. 
   \begin{itemize}
       \item $g(b)$ is a decreasing function, because $G(b)$ is the pointwise minimum of a collection of strictly-concave functions, and is hence strictly-concave itself
       \item $\lim\inf_{b\to0_+} g(b) \ge \sigma$. To see this, first notice that $\tau_0(b)$ is the zero of $\frac{2}{b}f_b(\tau)$ and 
       \[
            \lim_{b\to0_+} \frac{2}{b}f_b(\tau) = 1-\tau^{-2}(\sigma^2 + \frac1n \|\bbeta^*\|^2).
       \]
       Therefore, we have
       \[
            \lim_{b\to0_+} \tau_0(b) = \sigma^2 + \frac1n \|\bbeta^*\|^2,
       \]
       By using Fatou's lemma we obtain
        \begin{align*}
            \lim\inf_{b\to0_+} g(b) &= \lim_{b\to0_+}\tau_0(b) - \lim\sup_{b\to0_+}\frac1n \EE \bg^\top \hat{\bw}^f(\tau_0(b),b)\\
            &\geq \sigma^2 + \frac1n \|\bbeta^*\|^2 - \frac1n \sum_{k=1}^p \EE \left[ g_k \left( 0  -\beta_k^* \right) \right]\\
            &= \sigma^2 + \frac1n \|\bbeta^*\|^2\\
            &\geq \sigma^2.
        \end{align*}
        \item As mentioned before $G(b)$ is the pointwise minimum of a collection of strictly concave functions so it is itself strictly concave and therefore admits a unique maximizer $b_*\geq 0$. By the last point  $\lim\inf_{b\to0_+} g(b) \ge \sigma^2>0$. Therefore the maximizer cannot be $0$, and hence $b_*>0$. Since $G(b)$ is differentiable $b_*$ must be a zero of $g(b)$, and the zero must be unique as the maximizer is unique. 
   \end{itemize}
    We conclude that the unique saddle point is $(\tau_0(b_*),b_*)$ and it satisfies (\ref{eq:fix point eqs}).
\end{proof}

\medskip

\begin{proof}[Proof of Lemma \ref{lem:bd_tau_b}]  We remind the reader of the following notation: 
$$ \hat{\bw}^f : = \hat{\bw}^f (\tau_*, b_*).
$$
We divide the proof into the following steps: 

\paragraph{Step 1 (Lower bound for $b_*$):}
    Using the same notations as the ones used in the proof of Lemma \ref{lem: fix point equations}, we have $G(b)=\min_{\tau\geq \sigma}\psi(\tau,b)$ and $g(b)=G'(b)$. Using the fact that $\EE Z \cdot {\rm soft}(Z+a,r)=\Phi(a-r)+\Phi(-a-r)$ (which can be verified directly via integration), we have
    \begin{align*}
        g(b)=G'(b)&=\tau_0(b) - \frac1n \EE \bg^\top \hat{\bw}^f(\tau_0(b),b)-b\\
        &=\tau_0(b) \left(1-\frac1n \frac{b}{b+2\eta\tau_0(b)}\sum_{k=1}^p\EE g_k\cdot {\rm soft}(g_k+\frac{\theta_k^*}{\tau_0(b)},\frac{\lambda}{b})\right)-b\\
        &=\tau_0(b)\left(1-\frac1n \frac{b}{b+2\eta\tau_0(b)}\sum_{k=1}^p \left[\Phi(\frac{\theta_k^*}{\tau_0(b)}-\frac{\lambda}{b}) +  \Phi(-\frac{\theta_k^*}{\tau_0(b)}-\frac{\lambda}{b})\right]\right)-b\\
        &\geq\tau_0(b)\left(1-\frac{1}{\gamma_0} \EE \left[ \Phi(\frac{\Theta}{\tau_0(b)}-\frac{\lambda}{b}) +  \Phi(-\frac{\Theta}{\tau_0(b)}-\frac{\lambda}{b}) \right] \right)-b\\
        &:=\tau_0(b)\left(1-\frac{1}{\gamma_0} \EE h_b(\frac{\Theta}{\tau_0(b)}) \right)-b,
    \end{align*}
    where $\Theta$ has uniform distribution over the elements of $\beta^*$, and $h_b(x) = \Phi(x-\alpha)+\Phi(-x-\alpha)$ is an even funcion, decreasing for $x<0$ and increasing for $x\geq 0$. Let $K=\frac{2\xi}{\sigma\sqrt{\gamma_0}}$.  By Markov inequality we have
    \begin{align*}
        \EE h_b\left(\frac{\Theta}{\tau_0(b)}\right)&\leq h_b(K)+\PP\left(\left|\frac{\Theta}{\tau_0(b)}\right|\geq K\right)\\
        &\leq h_b(K) + \frac{\EE\Theta^2}{\tau_0^2(b)K^2}\\
        &\leq h_b(K) + \frac{\gamma_0}{4}.
    \end{align*}
    Finally it can be verified directly that $\lim_{b\to 0_+}h_b(K)=0$. Henec we can find $b_0>0$ depending only on $\xi, \lambda, \sigma, \gamma_0$ such that $\forall b\leq b_0, h_b(K)\leq \frac{\gamma_0}{4}$. Hence $\forall b\leq b_0$ we have
    \begin{align*}
        g(b)&\geq \tau_0(b)(1-\frac{1}{\gamma_0}(\frac{\gamma_0}{4}+\frac{\gamma_0}{4}))-b\\
        &= \frac12 \tau_0(b) -b\\
        &\geq \frac{\sigma}{2} -b.
    \end{align*}
    If we let $b_{min}=\min\{b_0, \frac{\sigma}{4}\}$, then $\forall b \leq b_{min}$, $g(b)\geq \frac{\sigma}{4}>0$. Since $b_*$ is the zero of $g(b)$ and $g(b)$ is decreasing, we conclude that $b_*>b_{min}$. This finishes the proof of the first step. 
    
\paragraph{Step 2 ($\tau_*<\tau_{max}$): }
    First note that if we insert $\bw = 0$ in the definition of $\psi(\tau,b)$, then we will have $\psi(\tau,b)\leq \left(\frac{\sigma^2}{\tau} + \tau\right)\frac{b}{2} - \frac{b^2}{2}$. Hence,
    \[
        \psi(\tau_*,b_*)\leq \max_{b\geq 0} \min_{\tau\geq\sigma}\left(\frac{\sigma^2}{\tau} + \tau\right)\frac{b}{2} - \frac{b^2}{2} = \frac{\sigma^2}{2}.
        \label{eq:bd_tau_b_bd_psi}
        \numberthis
    \]
    Next, we show that $\psi(\tau,b)$ has an increasing lower bound $g(\tau)$ for all $b\geq b_{min}$. In fact,
    \begin{eqnarray}\label{eq:psi_discussion}
       \lefteqn{ \psi(\tau,b)= \left(\frac{\sigma^2}{\tau} + \tau\right)\frac{b}{2} - \frac{b^2}{2} + \frac1n \EE \left\{ 
            \norm{\hat{\bw}^f}_2^2(\frac{b}{2\tau} + \eta) - (b\bg - 2\eta\beta^*)^\top \hat{\bw}^f + \lambda\norm{\hat{\bw}^f + \bbeta^*}_1 - \lambda \norm{\bbeta^*}_1
        \right\}}\nonumber \\
        &=&\frac{b}{2\tau}(\sigma^2 + \tau^2) - \frac{b^2}{2} + (\frac{b}{2\tau} + \eta) (\tau^2 - \sigma^2) - b(\tau-b) + \frac{2\eta}{n}\EE\bbeta^{*\top}\hat{\bw}^f + \frac{\lambda}{n}\EE\norm{\hat{\bw}^f+\bbeta^*}_1 - \frac{\lambda}{n}\norm{\bbeta^*}_1\nonumber \\
        &=&\eta(\tau^2-\sigma^2) + \frac{b^2}{2} + \frac{2\eta}{n}\EE\bbeta^{*\top}\hat{\bw}^f + \frac{\lambda}{n}\EE\norm{\hat{\bw}^f+\bbeta^*}_1 - \frac{\lambda}{n}\norm{\bbeta^*}_1. 
    \end{eqnarray}
    The first two terms in the last line of the above equation are nonnegative when $\tau = \tau_*$. For the third term,
    \begin{align*}
        \frac{2\eta}{n}\EE\bbeta^{*\top}\hat{\bw}^f
        &=\frac{2\eta}{n}\frac{b_*\tau_*}{b_* + 2\eta\tau_*}\sum_i \beta_i^*\cdot\EE {\rm soft}\left(Z+\frac{\beta_i^*}{\tau_*}, \frac{\lambda}{b_*}\right)\\
        &\geq -\frac{2\eta}{n}\frac{b_*\tau_*}{b_* + 2\eta\tau_*}\sum|\beta_i^*|\cdot\left( \frac{|\beta_i^*|}{\tau_*} + \frac{\lambda}{b_*} \right)\\
        &\geq -\frac{2\eta}{n}\frac{b_*}{b_* + 2\eta\tau_*}\norm{\bbeta^*}_2^2 - -\frac{2\eta}{n}\frac{\lambda\tau_*}{b_* + 2\eta\tau_*}\norm{\bbeta^*}_1\\
        &\geq - \frac{2\eta_{\max}}{\delta}\xi^2 - \frac{\lambda_{\max}}{\delta}\xi\\
        &:=-C_1 \label{eq:bd_tau_b_third_term_in_psi}\numberthis
    \end{align*}
    In the above equations, the first inequality uses $|\EE {\rm soft}(Z+x,r)|\leq |x|+r$, the third uses Cauchy-Schwarz inequality, and the last uses $\frac{b_*}{b_*+2\eta\tau_*}\leq 1$, $\frac{\eta\tau_*}{b_*+2\eta\tau_*}\leq \frac12$.


    For the fourth term of \eqref{eq:psi_discussion}, using the notation $\Theta$ taking values uniformly from the elements of $\bbeta^*$ and $Z\sim N(0,1)$ and independent from $\Theta$, we have
    \begin{align*}
        \frac{\lambda}{n} \EE\|\hat{\bw}^f+\bbeta^*\|_1
        &= \frac{\lambda}{\gamma_0}\frac{b_*\tau_*}{b_*+2\eta\tau_*}\EE \left|{\rm soft}(Z+\frac{\Theta}{\tau_*},\frac{\lambda}{b_*})\right|\\
        &\geq \frac{\lambda}{\gamma_0}\frac{b_{\min}\tau_*}{b_{\min}+2\eta\tau_*}\EE \left| {\rm soft}(Z+\frac{\Theta}{\tau_*},\frac{\lambda}{b_{\min}}) \right|.
    \end{align*}
    Now consider the function $h(a)=\EE|{\rm soft}(Z+a,r)|$ with $r=\frac{\lambda}{b_{min}}$. We have $h(0)=\EE|{\rm soft}(Z,r)|:=C>0$ and by Monotone Convergence Theorem we have $\lim\limits_{a\to\infty}h(a)=+\infty$. Therefore there exists a constant $C_2>0$ that depends only on $r=\frac{\lambda}{b_{\min}}$ such that $h(a)\geq C_2$. Hence, we have
    \begin{align*}
        \frac{\lambda}{n} \EE\|\hat{\bw}^f+\bbeta^*\|_1
        &\geq \frac{\lambda}{\gamma_0}\frac{b_{min}\tau_*}{b_{min}+2\eta\tau_*} \EE h(\frac{\Theta}{\tau_*})\\
        &\geq \frac{\lambda}{\gamma_0}\frac{b_{min}\tau_*}{b_{min}+2\eta\tau_*}C_2
        \label{eq:bd_tau_b_fourth_term_in_psi}\numberthis
    \end{align*}   
    Finally, for the last term of \eqref{eq:psi_discussion} we have
    \[
        \frac{\lambda}{n}\|\bbeta^*\|_1\leq \frac{\lambda\xi}{\gamma_0}
        \leq \frac{\lambda_{\max}\xi}{\gamma_0}.
        \label{eq:bd_tau_b_fifth_term_in_psi}\numberthis
    \]
    Combining \eqref{eq:bd_tau_b_bd_psi}, \eqref{eq:bd_tau_b_third_term_in_psi},  \eqref{eq:bd_tau_b_fourth_term_in_psi} and \eqref{eq:bd_tau_b_fifth_term_in_psi} together, we have
    $$\eta(\tau_*^2-\sigma^2)-C_1 + \frac{\lambda}{\gamma_0}\frac{b_{\min}\tau_*}{b_{\min}+2\eta\tau_*}C_2 - \frac{\lambda_{\max}\xi}{\gamma_0}  
    \leq\psi(\tau_*,b_*)\leq \frac{\sigma^2}{2} $$
    If we define $C_1'=\frac{\ sigma^2}{2}+\eta_{\max}\sigma^2+C_1+\frac{\lambda_{\max}\xi}{\gamma_0}$, $C_2'=\frac{\lambda_{\max}}{\gamma_0}b_{\min}C_2$, we can rephrase the above inequality into
        $$\eta\tau_*^2 + \frac{C_2'\tau_*}{b_{\min}+2\eta\tau_*}\leq C_1'$$
    It can be verified directly that, the minimum of the left hand side over $\eta\geq 0$ is $\max\{ (\sqrt{2C_2'}-\frac{b_{\min}}{2}), \frac{C_2'}{b_{\min}}\} \tau_*:=C_3\tau_*$. Hence, we conclude that $\tau_*<\tau_{\max} = \frac{C_1'}{C_3}$ depending on model parameters. 

    \paragraph{Step 3 (Upper bound for $b_*$):} First note that
    $$\frac1n \EE \bg^\top \hat{\bw}^f = \frac1n \frac{b_*\tau_*}{b_* + 2\eta \tau_*}\sum_i \PP\left( \left|g_i + \frac{\beta_i^*}{\tau_*}\right| > \frac{\lambda}{b_*} \right)\geq 0. $$
    Therefore
    $$b_* = \tau_* - \frac1n \EE \bg^\top \hat{\bw}^f \leq \tau_{max}.$$

    \paragraph{Step 4 (Lower bound for $\tau_*$):} It is trivial that $\tau_*>\sigma$ and is hence skipped.
\end{proof}
\end{document}